%% file: manuscript.tex
\documentclass[
  a4paper,oneside,DIV=12,
  11pt,
  headsepline,
  ]{scrartcl}

\input{packages}

\input{layout}
\input{shortcuts}
\input{shortcuts_this}

\newcommand{\headertitle}{{%
  The Geometric Unitary Kudla Conjecture
}}
\newcommand{\headerauthors}{%
  M.~Raum
}
\title{%
  The Geometric Unitary Kudla Conjecture
}
\author{%
  Martin Raum
  \thanks{The author was partially supported by Vetenskapsr\aa det Grant~2023-04217.}%
}
\ifdraft{\date{\today\ at\ \thistime}}{\date{}}

\begin{document}

\thispagestyle{scrplain}
\begingroup
\deffootnote[1em]{1.5em}{1em}{\thefootnotemark}
\maketitle
\endgroup

\begin{abstract}
  \small
  \noindent
  {\tbf Abstract:}
  We prove that, over an arbitrary CM field, every symmetric formal Fourier--Jacobi series converges and equals the Fourier--Jacobi expansion of a genuine Hermitian Hilbert modular form.
  As an application, we show that the Chow-valued Kudla generating series of special cycles on unitary Shimura varieties for Hermitian lattices over CM fields of signature~$(p,1)$ at one infinite place and~$(p+1,0)$ at all others is modular of weight~$p+1$ for a Weil representation, establishing the geometric unitary Kudla Conjecture in arbitrary codimension.
  This removes the modularity hypothesis from the arithmetic inner product formula by Li--Liu.
  \\[.3\baselineskip]
  \noindent
  \textsf{\textbf{%
      Kudla Conjecture%
    }}%
  \noindent
  \ {\tiny$\blacksquare$}\ %
  \textsf{\textbf{%
      Unitary Shimura varieties%
    }}%
  \noindent
  \ {\tiny$\blacksquare$}\ %
  \textsf{\textbf{%
      Weighted special cycles%
    }}%
  \noindent
  \ {\tiny$\blacksquare$}\ %
  \textsf{\textbf{%
      Hermitian Hilbert modular forms%
    }}
  \\[.2\baselineskip]
  \noindent
  \textsf{\textbf{%
      MSC Primary:
      14C25%
    }}%
  \ {\tiny$\blacksquare$}\ %
  \textsf{\textbf{%
      MSC Secondary:
      11F50,
      11F55%
    }}
\end{abstract}

\vspace{-2.5\baselineskip}
\renewcommand{\contentsname}{}
\setcounter{tocdepth}{\sectiontocdepth}
\tableofcontents
\setcounter{tocdepth}{\subsectiontocdepth}
\vspace{1.5\baselineskip}

\subfile{sections/00_introduction.tex}

\subfile{sections/01_hilbert_jacobi_forms.tex}

\subfile{sections/02_hermitian_modular_forms.tex}

\subfile{sections/03_convergence.tex}

\subfile{sections/04_vector_valued.tex}

\subfile{sections/05_kudla_conjecture.tex}

\vspace{1.5\baselineskip}
\phantomsection
\addcontentsline{toc}{section}{References}
\markright{References}
\label{sec:references}
\sloppy
\printbibliography[heading=none]

\filbreak
\Needspace*{5\baselineskip}
\noindent%
\rule{\textwidth}{0.15em}
\\\nopagebreak

{\small\noindent
  Martin Raum\\\nopagebreak
  Department of Mathematical Sciences\\\nopagebreak
  Chalmers University of Technology and University of Gothenburg\\\nopagebreak
  SE-412 96 Gothenburg, Sweden\\\nopagebreak
  E-mail: \url{martin@raum-brothers.eu}\\\nopagebreak
  Homepage: \url{https://martin.raum-brothers.eu}
}%

\ifdraft{%
  \listoftodos%
}

\end{document}

%% file: packages.tex
\usepackage{etoolbox}
\usepackage{ifdraft}
\usepackage{iftex}
\usepackage{subfiles}

\ifluatex
  \usepackage[utf8]{luainputenc}
\else
  \usepackage[utf8]{inputenc}
  \usepackage[T1]{fontenc}
\fi

\usepackage[english]{babel}
\usepackage[autostyle=true]{csquotes}

\usepackage{lmodern}
\usepackage{fourier}

\usepackage{amsfonts}
\usepackage{mathrsfs} %
\usepackage{dsfont}

\usepackage{amssymb}

\usepackage{amsmath}
\usepackage{mathtools}
\usepackage[thmmarks, amsmath, amsthm]{ntheorem}
\usepackage{nccmath}    %
\usepackage{tensor}

\usepackage[draft=false]{scrlayer-scrpage}
\usepackage{needspace}

\usepackage[cmyk,dvipsnames]{xcolor}

\ifdraft{%
  \usepackage[textsize=scriptsize,colorinlistoftodos]{todonotes}
  \usepackage{lineno}
  \usepackage{scrtime}
}{}

\usepackage{booktabs}
\usepackage[inline]{enumitem}
\usepackage{lettrine}
\usepackage{url}

\usepackage[backend=biber,style=numeric-comp,giveninits,url=false,sortcites=true,maxbibnames=99]{biblatex}
\usepackage[final,
  pdfborder={0 0 0}, colorlinks=true,
  linkcolor=BrickRed, citecolor=ForestGreen, urlcolor=RoyalBlue]{hyperref}

%% file: layout.tex
\ifdraft{%
  \linenumbers
  \newcommand*\patchAmsMathEnvironmentForLineno[1]{%
    \expandafter\let\csname old#1\expandafter\endcsname\csname #1\endcsname
    \expandafter\let\csname oldend#1\expandafter\endcsname\csname end#1\endcsname
    \renewenvironment{#1}%
    {\linenomath\csname old#1\endcsname}%
    {\csname oldend#1\endcsname\endlinenomath}}%
  \newcommand*\patchBothAmsMathEnvironmentsForLineno[1]{%
    \patchAmsMathEnvironmentForLineno{#1}%
    \patchAmsMathEnvironmentForLineno{#1*}}%
  \AtBeginDocument{%
    \patchBothAmsMathEnvironmentsForLineno{equation}%
    \patchBothAmsMathEnvironmentsForLineno{align}%
    \patchBothAmsMathEnvironmentsForLineno{flalign}%
    \patchBothAmsMathEnvironmentsForLineno{alignat}%
    \patchBothAmsMathEnvironmentsForLineno{gather}%
    \patchBothAmsMathEnvironmentsForLineno{multline}%
  }}

\clearmainofpairofpagestyles
\automark[section]{subsection}

\renewcommand{\subsectionmark}[1]{}

\lohead{{\small\normalfont \headertitle}}
\rohead{{\small \headerauthors}}

\RedeclareSectionCommand[%
  font=\Large\sffamily\bfseries,%
  beforeskip=1\baselineskip,%
  afterskip=0.5\baselineskip,%
  indent=0em,%
  afterindent=false,%
  tocbeforeskip=0.3\baselineskip plus 1pt minus 1pt%
]{section}

\RedeclareSectionCommands[%
  font=\normalfont\bfseries,%
  beforeskip=3pt,%
  afterskip=-1em%
]{subsection,subsubsection}

\RedeclareSectionCommands[%
  font=\normalfont\itshape,%
  beforeskip=.5\baselineskip,%
  afterskip=-1em,%
  indent=0pt,%
]{paragraph}

\AtEveryBibitem{\clearfield{doi}}
\AtEveryBibitem{\clearfield{isbn}}
\AtEveryBibitem{\clearfield{issn}}
\AtEveryBibitem{\clearfield{pages}}
\AtEveryBibitem{\clearlist{language}}

\renewcommand*{\bibfont}{\footnotesize}
\renewbibmacro*{in:}{}
\DeclareFieldFormat
[article,inbook,incollection,inproceedings,patent,thesis,unpublished,misc]
{title}{#1}

\setitemize{itemsep=0.02\baselineskip}

\newenvironment{enumerateroman}{
  \begin{enumerate}[label=(\roman*),%
  leftmargin=2.5em,itemindent=0pt,%
  labelindent=.5em,labelwidth=1.5em,labelsep=!,%
  nosep]
  }{
  \end{enumerate}
}

\newenvironment{enumerateromanseries}[1]{
  \begin{enumerate}[label=(\roman*),series=#1,%
  leftmargin=2.5em,itemindent=0pt,%
  labelindent=.5em,labelwidth=1.5em,labelsep=!,%
  noitemsep]
  }{
  \end{enumerate}
}

\newenvironment{enumeratearabic*}{
  \begin{enumerate*}[label=(\arabic*)] %
    }{
  \end{enumerate*}
}

\newenvironment{enumerateroman*}{
  \begin{enumerate*}[label=(\roman*)] %
    }{
  \end{enumerate*}
}

\numberwithin{equation}{section}

\theoremnumbering{arabic}
\newtheorem{theoremcounter}{theoremcounter}[section]
\theoremnumbering{Alph}
\newtheorem{maintheoremcounter}{maintheoremcounter}

\theoremstyle{plain}

\newtheorem{corollary}[theoremcounter]{Corollary}
\newtheorem{lemma}[theoremcounter]{Lemma}

\newtheorem{proposition}[theoremcounter]{Proposition}
\newtheorem{theorem}[theoremcounter]{Theorem}

\theoremstyle{plain}

\newtheorem{maincorollary}[maintheoremcounter]{Corollary}
\newtheorem{maintheorem}[maintheoremcounter]{Theorem}

\theoremstyle{definition}

\newtheorem{definition}[theoremcounter]{Definition}

\theoremstyle{remark}

\newtheorem{remark}[theoremcounter]{Remark}

\theoremstyle{nonumberremark}

\newtheorem{mainremark}{Remark}

%% file: shortcuts.tex
\newcommand{\tx}{\text}
\newcommand{\thdash}{\nbd th}
\newcommand{\stdash}{\nbd st}

\newcommand{\nbd}{\nobreakdash-\hspace{0pt}}

\newcommand{\fref}[2]{\hyperref[#2]{#1~\ref*{#2}}}

\makeatletter
\newcommand{\writelabel}[1]{#1\def\@currentlabel{#1}}
\makeatother

\newcommand{\minwidthmathbox}[2]{%
  \mathmakebox[{\ifdim#1<\width\width\else#1\fi}]{#2}%
}

\newcommand{\tbf}{\bfseries}

\newcommand{\bbA}{\ensuremath{\mathbb{A}}}

\newcommand{\cB}{\ensuremath{\mathcal{B}}}

\newcommand{\cD}{\ensuremath{\mathcal{D}}}

\newcommand{\cM}{\ensuremath{\mathcal{M}}}

\newcommand{\cO}{\ensuremath{\mathcal{O}}}

\newcommand{\cS}{\ensuremath{\mathcal{S}}}

\newcommand{\fraka}{\ensuremath{\mathfrak{a}}}
\newcommand{\frakb}{\ensuremath{\mathfrak{b}}}

\newcommand{\frakd}{\ensuremath{\mathfrak{d}}}
\newcommand{\frake}{\ensuremath{\mathfrak{e}}}
\newcommand{\frakf}{\ensuremath{\mathfrak{f}}}
\newcommand{\frakg}{\ensuremath{\mathfrak{g}}}

\newcommand{\frakk}{\ensuremath{\mathfrak{k}}}

\newcommand{\frakp}{\ensuremath{\mathfrak{p}}}

\newcommand{\frakt}{\ensuremath{\mathfrak{t}}}

\newcommand{\frakx}{\ensuremath{\mathfrak{x}}}
\newcommand{\fraky}{\ensuremath{\mathfrak{y}}}
\newcommand{\frakz}{\ensuremath{\mathfrak{z}}}

\newcommand{\rmd}{\ensuremath{\mathrm{d}}}

\newcommand{\rmf}{\ensuremath{\mathrm{f}}}

\newcommand{\rmr}{\ensuremath{\mathrm{r}}}

\newcommand{\rmt}{\ensuremath{\mathrm{t}}}

\newcommand{\rmC}{\ensuremath{\mathrm{C}}}

\newcommand{\rmE}{\ensuremath{\mathrm{E}}}
\newcommand{\rmF}{\ensuremath{\mathrm{F}}}
\newcommand{\rmG}{\ensuremath{\mathrm{G}}}
\newcommand{\rmH}{\ensuremath{\mathrm{H}}}

\newcommand{\rmJ}{\ensuremath{\mathrm{J}}}
\newcommand{\rmK}{\ensuremath{\mathrm{K}}}
\newcommand{\rmL}{\ensuremath{\mathrm{L}}}
\newcommand{\rmM}{\ensuremath{\mathrm{M}}}

\newcommand{\rmP}{\ensuremath{\mathrm{P}}}

\newcommand{\rmS}{\ensuremath{\mathrm{S}}}
\newcommand{\rmT}{\ensuremath{\mathrm{T}}}
\newcommand{\rmU}{\ensuremath{\mathrm{U}}}
\newcommand{\rmV}{\ensuremath{\mathrm{V}}}

\newcommand{\rmZ}{\ensuremath{\mathrm{Z}}}

\newcommand{\td}{\tilde}
\newcommand{\wtd}{\widetilde}
\newcommand{\ov}{\overline}
\newcommand{\ul}{\underline}

\newcommand{\defcol}{\mathrel{:}}
\newcommand{\defeq}{\mathrel{:=}}

\newcommand{\condsep}{\mathrel{\;:\;}}
\newcommand{\quantsep}{\mathrel{\;.\;}}

\newcommand{\mrelspace}[1]{\mathrel{\mspace{#1}}}

\NewCommandCopy{\rightarroworig}{\rightarrow}
\renewcommand{\rightarrow}
  {\protect\relbar\mrelspace{-9.7mu}\rightarroworig}
\renewcommand{\hookrightarrow}
  {\protect\lhook\mrelspace{-3.1mu}\relbar\mrelspace{-11.9mu}\rightarroworig}
\renewcommand{\twoheadrightarrow}
  {\protect\rightarroworig\mrelspace{-15mu}\rightarroworig}

\NewCommandCopy{\leftarroworig}{\leftarrow}
\renewcommand{\leftarrow}
  {\protect\leftarroworig\mrelspace{-9.7mu}\relbar}

\renewcommand{\longrightarrow}
  {\protect\relbar\mrelspace{-3.2mu}\relbar\mrelspace{-9.5mu}\rightarroworig}
\newcommand{\longhookrightarrow}
  {\protect\lhook\mrelspace{-3.1mu}\relbar\mrelspace{-3.2mu}\relbar\mrelspace{-11.7mu}\rightarroworig}
\newcommand{\longtwoheadrightarrow}
  {\protect\relbar\mrelspace{-3mu}\rightarroworig\mrelspace{-15mu}\rightarroworig}

\newcommand{\ra}{\rightarrow}
\newcommand{\hra}{\hookrightarrow}
\newcommand{\thra}{\twoheadrightarrow}

\newcommand{\lra}{\longrightarrow}
\newcommand{\lhra}{\longhookrightarrow}
\newcommand{\lthra}{\longtwoheadrightarrow}

\newcommand{\mto}{\mapsto}
\newcommand{\lmto}{\longmapsto}

\renewcommand{\Re}{\mathrm{Re}}
\renewcommand{\Im}{\mathrm{Im}}

\renewcommand{\pmod}[1]{\;(\mathrm{mod}\, #1)}

\newcommand{\Hom}{\operatorname{Hom}}
\newcommand{\id}{\mathrm{id}}

\newenvironment{psmatrix}{\left(\begin{smallmatrix}}{\end{smallmatrix}\right)}
\newcommand{\Mat}[1]{\operatorname{Mat}_{#1}}
\newcommand{\MatT}[1]{\operatorname{Mat}^\rmt_{#1}}
\newcommand{\diag}{\operatorname{diag}}

\newcommand{\lT}[1]{\tensor*[^\rmt]{#1}{}}
\newcommand{\lTr}[2]{\tensor*[^\rmt]{#1}{#2}}

\newcommand{\trace}{\operatorname{trace}}

\newcommand{\linspan}{\operatorname{span}}

\newcommand{\rk}{\operatorname{rk}}

\newcommand{\ZZ}{\ensuremath{\mathbb{Z}}}
\newcommand{\QQ}{\ensuremath{\mathbb{Q}}}
\newcommand{\RR}{\ensuremath{\mathbb{R}}}
\newcommand{\CC}{\ensuremath{\mathbb{C}}}

\newcommand{\GL}[1]{\ensuremath{\mathrm{GL}_{#1}}}
\newcommand{\Mp}[1]{\ensuremath{\mathrm{Mp}_{#1}}}

\newcommand{\SL}[1]{\ensuremath{\mathrm{SL}_{#1}}}

\newcommand{\SU}[1]{\ensuremath{\mathrm{SU}_{#1}}}

\newcommand{\U}[1]{\ensuremath{\mathrm{U}_{#1}}}

\renewcommand{\det}{\ensuremath{\mathrm{det}}}

\newcommand{\HS}{\ensuremath{\mathbb{H}}}

%% file: shortcuts_this.tex
\newcommand{\eps}{\epsilon}
\newcommand{\ga}{\gamma}
\newcommand{\Ga}{\Gamma}
\newcommand{\ka}{\kappa}
\newcommand{\la}{\lambda}

\newcommand{\rmdx}{\rmd\mspace{-2mu}x}

\newcommand{\ovsmash}[1]{\ov{#1}\vphantom{#1}}
\newcommand{\ulsmash}[1]{\ul{#1}\vphantom{#1}}

\newcommand{\ord}{\operatorname{ord}}
\newcommand{\ordnorm}{\operatorname{ord}_{\mathrm{Nm}}}
\newcommand{\ordtrace}{\operatorname{ord}_{\mathrm{Tr}}}

\newcommand{\Stab}{\operatorname{Stab}}

\newcommand{\cOE}{\cO_E}

\newcommand{\cOF}{\cO_F}
\newcommand{\cDF}{\cD_F}
\newcommand{\lD}[1]{\tensor*[^\dagger]{#1}{}}
\newcommand{\lDr}[2]{\tensor*[^\dagger]{#1}{#2}}
\newcommand{\traceF}{\trace_{F \slash \QQ}}
\newcommand{\traceE}{\trace_{E \slash F}}
\newcommand{\normF}{\operatorname{norm}_{F \slash \QQ}}
\newcommand{\detF}{\det_{F \slash \QQ}}

\newcommand{\idmat}[1]{\id_{#1}}

\newcommand{\sinv}{\operatorname{sinv}}
\newcommand{\rot}{\operatorname{rot}}
\newcommand{\trans}{\operatorname{trans}}
\newcommand{\transJ}{\operatorname{transJ}}
\newcommand{\transJU}{\operatorname{transJU}}

\newcommand{\MatD}[1]{\mathrm{Mat}^\dagger_{#1}}

\newcommand{\HSEg}{\HS_{E,g}}
\newcommand{\HJ}[1]{\HS_{#1}}

\newcommand{\Hb}[1]{\mathrm{Hb}_{#1}}

\newcommand{\Jac}[1]{\mathrm{Jac}_{#1}}

\newcommand{\rmFM}{\rmF\rmM}
\newcommand{\rmFS}{\rmF\rmS}
\newcommand{\rmJHilb}[1]{\rmJ^{\mathrm{Hilb}}_{#1}}

\newcommand{\rmVclosed}{\ovsmash{\rmV}}
\newcommand{\rmVopen}{\mathring{\rmV}}
\newcommand{\rmTP}{\rmT\rmP}
\newcommand{\rmEV}{\rmE\rmV}

\newcommand{\rmGrM}{\rmG\rmr^-}
\newcommand{\rmCH}{\rmC\rmH}
\newcommand{\thetaKudla}{\theta^{\mathrm{Kudla}}}
\newcommand{\disc}{\mathrm{disc}}

\newcommand{\piCH}{\pi^{\rmCH}}
\newcommand{\pirho}{\pi^\rho}

\NewCommandCopy{\slashorig}{\slash}
\renewcommand{\slash}{\mathbin{\slashorig}}
\newcommand{\bigslash}{\mathbin{\big\slashorig}}
\newcommand{\Bigslash}{\mathbin{\Big\slashorig}}

%% file: sections/00_introduction.tex
\Needspace*{4em}
\phantomsection
\label{sec:introduction}
\addcontentsline{toc}{section}{Introduction}
\markright{Introduction}
\lettrine[lines=2,nindent=.2em]{\tbf T}{\,he} Kudla program relates special cycles on Shimura varieties to automorphic forms.
It emerged in work of Kudla~\cite{kudla-1997,kudla-1997b,kudla-2003} and originates in a series of papers by Kudla and Millson~\cite{kudla-millson-1986,kudla-millson-1987,kudla-millson-1990} and, ultimately, Hirzebruch and Zagier~\cite{hirzebruch-zagier-1976}.
Kudla and Millson established a connection between cohomology classes associated with special cycles on orthogonal and unitary Shimura varieties on one hand and holomorphic Siegel and Hermitian modular forms on the other hand.
This cohomological phenomenon can be explained by the theta correspondence for automorphic forms.
Kudla suggested that one should refine this connection beyond cohomology, and that algebraic cycles should appear as the Fourier coefficients of modular forms.
The resulting generating series take values in the Chow group of Shimura varieties and serve as arithmetic analogues of theta series.
Their modularity encodes information that is a priori inaccessible at the level of cohomology by the Beilinson--Bloch Conjecture, which expresses Chow groups as an extension of cohomology; see~\cite{kudla-2021} for an embedding trick that establishes the Kudla Conjecture assuming the Beilinson--Bloch Conjecture.
In this paper, we prove unconditionally the geometric Kudla Conjecture for unitary Shimura varieties associated with Hermitian lattices over CM fields of signature~$(p,1)$ at one infinite place and~$(p+1,0)$ at all others.

The modularity of Kudla's generating series enables arithmetic theta liftings via the Petersson scalar product on modular forms, previously constructed conditionally on this modularity~\cite{liu-2011a,liu-2011b}.
These liftings supply cycles that are natural generalizations of Heegner points on modular curves.
The Gross--Zagier Theorem~\cite{gross-zagier-1986,gross-kohnen-zagier-1987} relates N\'eron--Tate heights of such points to derivatives of certain~$\rmL$\nbd{}functions.
Li and Liu extended this relation to an arithmetic inner product formula, which relates the intersection pairing of cycles constructed via the arithmetic theta lifting (assuming it exists), to derivatives of certain~$\rmL$\nbd{}functions~\cite{li-liu-2021,li-liu-2022}.
As such, Kudla's program yields a generalization of the Gross--Zagier Formula in the spirit of the arithmetic Gan--Gross--Prasad Conjecture~\cite[\S27]{gan-gross-prasad-2012}; see also the exposition in~\cite{rapoport-smithling-zhang-2020}.
Our result removes one of the hypotheses in Li--Liu's work.

Shortly after Kudla's initial proposal of generating series of weighted special cycles, Borcherds established a first modularity result in the case of special divisors on orthogonal Shimura varieties~\cite{borcherds-1999}.
His approach characterizes the relations among Fourier coefficients of modular forms via Serre duality on modular curves, and deduces the required rational equivalences of special divisors by constructing suitable meromorphic functions on the orthogonal Shimura variety, known as Borcherds products~\cite{borcherds-1998}.
An alternative route that circumvents the need for Borcherds products was taken by Yuan--Zhang--Zhang~\cite{yuan-zhang-zhang-2009}, who deduced the modularity of the Chow-valued divisor generating function from modularity of the cohomology-valued one and a vanishing result for a first Betti number; this includes a precursor idea for the aforementioned general result conditional on the Beilinson--Bloch Conjecture~\cite{kudla-2021}.
In the case of higher codimension, Zhang, in his thesis~\cite{zhang-2009}, leveraged Borcherds' results to establish a \emph{formal} modularity statement for special cycles of arbitrary codimension; see also~\cite{yuan-zhang-zhang-2009}.
Zhang's formal series are modular covariant under a suitable set of generators of the Siegel modular group, but they a priori lack convergence.
The geometric orthogonal Kudla Conjecture over the rational field was settled in~\cite{bruinier-2015,raum-2015c,bruinier-raum-2015}, where it is shown that convergence is, in fact, automatic for all formal Siegel modular forms of the kind arising in Zhang's work.

Hofmann and Liu independently adapted Borcherds' and Yuan--Zhang--Zhang's ideas to special divisors on unitary Shimura varieties~\cite{hofmann-2014,liu-2011a}.
While Hofmann restricted to the case of imaginary quadratic fields, Liu treats all CM fields.
In his work, Liu also extended Zhang's ideas and established that generating series of special cycles on unitary Shimura varieties yield \emph{formal} Hermitian Hilbert modular forms.
The approach to automorphic convergence from previous work, however, does not extend except in a few cases of imaginary quadratic fields treated by Xia~\cite{xia-2022}.
In the present paper we prove automatic convergence for formal Hermitian Hilbert modular forms and thus establish the geometric unitary Kudla Conjecture.

Fix a CM field~$E / F$ with ring of integers~$\cOE / \cOF$, and let~$L$ be a Hermitian~$\cOE$-lattice with dual~$L^\vee$ such that~$L_\RR \defeq L \otimes_\ZZ \RR$ has signature~$(p,1)$ at one place of~$F$ and is positive definite at all other infinite places.
We write~$L_\RR^- = L \otimes_{\cOF} \RR$ for~$L$ at the place~$F \hra \RR$ at which~$L$ is indefinite.
Let~$\rmGrM(L_\RR^-)$ be the Grassmannian of complex negative lines in~$L_\RR^-$.
For a neat, stable congruence subgroup~$\Ga \subset \U{L}(\cOF)$ the quotient~$\Ga \backslash \rmGrM(L_\RR^-)$ is a connected component of a unitary Shimura variety.
Kudla associates to orbits~$[\la]$ of~$g$-tuples~$\la = (\la_1, \ldots, \la_g)$ of vectors in~$L^\vee$ weighted special cycles~$\rmZ([\la])$ on~$\Ga \backslash \rmGrM(L_\RR^-)$.
For these and the following notions related to special cycles we refer to \fref{Section}{sec:kudla_conjecture}, where we also describe the relation to cycles on Shimura varieties associated with nearby incoherent, totally positive Hermitian spaces.
One can employ powers of the dual of the Hodge bundle to force codimension~$g$.
By then organizing weighted special cycles according to their moment matrices~$(\langle \la_i, \la_j \rangle) \slash 2$, one obtains the Chow-valued formal generating series~\eqref{eq:def:kudla_generating_series_vector_valued}
\begin{gather*}
  \thetaKudla_L(\tau)
  \in
  \CC\bigl[ \disc(L)^g \bigr]
  \otimes
  \rmCH^g\big( \Ga \backslash \rmGrM(L_\RR^-) \big) \bigl\llbracket e(\tau_{i\!j}) \bigr\rrbracket
  \tx{,}
\end{gather*}
where the first tensor factor is the representation space of the genus~$g$ incoherent Weil representation~$\wtd\rho^{(g)}_L$ and~$e(\tau_{i\!j}) = \exp(2 \pi i\, \tau_{i\!j})$ are exponentials of the entries~$\tau_{i\!j}$ of a variable~$\tau$ in the Hermitian upper half space~$\HSEg$ in~\eqref{eq:def:HSg}.

The notion of a vector-valued Hermitian Hilbert modular form is introduced in \fref{Definition}{def:hermitian_modular_form_arithmetic_type}.
We denote the space of weight~$k$, type~$\rho$ modular forms of genus~$g$ by~$\rmM^{(g)}_k(\rho)$.
The tensor product of this space with the Chow group consists of \emph{finite} sums of terms~$f \otimes \alpha$, where~$f$ is a modular form and~$\alpha$ is an algebraic cycle.
The Fourier series expansion of Hermitian Hilbert modular forms thus yields the inclusion
\begin{gather*}
  \rmM^{(g)}_{k}\bigl( \wtd\rho^{(g)}_L \bigr)
  \otimes
  \rmCH^g\big( \Ga \backslash \rmGrM(L_\RR^-) \big)
  \lhra
  \CC\bigl[ \disc(L)^g \bigr]
  \otimes
  \rmCH^g\big( \Ga \backslash \rmGrM(L_\RR^-) \big) \bigl\llbracket e(\tau_{i\!j}) \bigr\rrbracket
  \tx{,}
\end{gather*}
which we use to identify the left hand side as a subspace of the right hand side.

The main result of this paper, proved at the end of \fref{Section}{sec:kudla_conjecture}, establishes the unitary geometric Kudla Conjecture.

\begin{maintheorem}%
[Unitary geometric Kudla Conjecture]
\label{mainthm:kudla_conjecture}
The Kudla generating series\/~$\thetaKudla_L$ is the Fourier series expansion of a Hermitian Hilbert modular form of weight\/~$p+1$ and type\/~$\wtd\rho^{(g)}_L$.
More precisely,
\begin{gather*}
  \thetaKudla_L
  \in
  \rmM^{(g)}_{p+1}\big( \wtd\rho^{(g)}_L \big)
  \otimes
  \rmCH^g\big( \Ga \backslash \rmGrM(L_\RR^-) \big)
  \tx{.}
\end{gather*}
\end{maintheorem}

\begin{maincorollary}%
\label{maincor:li_liu_hypothesis_removed}
The modularity assumptions~\cite[Hypothesis~4.5]{li-liu-2021} and~\cite[Hypothesis~4.11]{li-liu-2022} of Li--Liu are satisfied.
In particular, the arithmetic inner product formula and the multiplicity bound for Chow groups of unitary Shimura varieties in~\cite[Theorem~1.7]{li-liu-2021} and~\cite[Theorem~1.5]{li-liu-2022}, previously conditional on these hypotheses, depend only on the remaining \cite[Hypothesis~6.6]{li-liu-2021}.
\end{maincorollary}

\begin{mainremark}
Li and Liu explain that~\cite[Hypothesis~6.6]{li-liu-2021} follows from the full endoscopic classification for unitary groups.
\end{mainremark}

The codimension~$1$ case of \fref{Theorem}{mainthm:kudla_conjecture} was provided by Liu~\cite{liu-2011a}, and independently by Hofmann~\cite{hofmann-2014} in the special case~$F = \QQ$.
The key new ingredient in our proof of the higher-codimension case is an automatic convergence result for formal Fourier--Jacobi series in the Hermitian Hilbert setting.
Hermitian Hilbert modular forms admit Fourier--Jacobi expansions with respect to a block decomposition of the variable~$\tau \in \HSEg$, whose coefficients are Hermitian Hilbert--Jacobi forms indexed by Hermitian matrices~$m$ of size~$h \times h$.
Further, the Fourier coefficients of Hermitian Hilbert modular forms satisfy the symmetry relation~\eqref{eq:hermitian_fourier_coefficient_symmetry}, which is induced by the action of~$\GL{g}(\cOE)$.
The notion of \emph{symmetric formal Fourier--Jacobi series} provided in \fref{Definition}{def:symmetric_formal_fourier_jacobi_series} combines these two features.
They are formal series of Hermitian Hilbert--Jacobi forms, indexed by Hermitian matrices~$m$, whose Fourier coefficients satisfy that same relation~\eqref{eq:hermitian_fourier_coefficient_symmetry}.
We write~$\rmFM^{(g,h)}_k(\rho)$ for the space of such series of weight~$k$ and type~$\rho$.
This notion formalizes the formal Hermitian Hilbert modular forms that appear in Liu's work.
The Fourier--Jacobi expansion of Hermitian Hilbert modular forms gives the natural inclusion~\eqref{eq:symmetric_formal_fourier_jacobi_series_from_modular_forms}
\begin{gather*}
  \rmM^{(g)}_k(\rho)
  \lhra
  \rmFM^{(g,h)}_k(\rho)
  \tx{.}
\end{gather*}
Our main analytic result, proved at the end of \fref{Section}{sec:vector_valued_higher_cogenus}, identifies symmetric formal Fourier--Jacobi series with the Fourier--Jacobi expansions of Hermitian Hilbert modular forms.

\begin{maintheorem}%
[Automatic convergence]
\label{mainthm:convergence}
For\/~$1 \le h < g$,~$k \in \ZZ$, and\/~$\rho$ an arithmetic type, we have
\begin{gather*}
  \rmFM^{(g,h)}_k(\rho)
  =
  \rmM^{(g)}_k(\rho)
  \tx{.}
\end{gather*}
\end{maintheorem}

Automatic convergence has received persistent attention since the 2000s.
The notion of symmetric formal Fourier--Jacobi series originates from the Siegel modular setting.
Aoki~\cite{aoki-2000} implicitly considered their automatic convergence in the case of genus~$2$ and the full Siegel modular group.
Poor and Yuen~\cite{poor-yuen-2007} independently obtained results of similar flavor, which informed later approaches, for formal Fourier series in the course of their computations of spaces of Siegel modular cusp forms.
Based on Aoki's approach, Ibukiyama, Poor, and Yuen~\cite{ibukiyama-poor-yuen-2012} established convergence of symmetric formal Fourier--Jacobi series for genus~$2$ Siegel modular forms of level~$1$.
They leveraged a coincidental coalescence of dimensions, which Wang~\cite{wang-2021} later extended to some orthogonal groups.
The work of Ibukiyama, Poor, and Yuen served as the basis for two independent proofs of the orthogonal geometric Kudla Conjecture in codimension~$2$ over~$\QQ$~\cite{raum-2015c,bruinier-2015}, both of which delivered an extension of~\cite{ibukiyama-poor-yuen-2012} to the vector-valued case using distinct methods.
Automatic convergence of all vector-valued symmetric formal Fourier--Jacobi series for the full Siegel modular group of any genus was eventually established in~\cite{bruinier-raum-2015,bruinier-raum-2025-preprint}.
In the Hermitian setting, Xia~\cite{xia-2022} proved convergence over imaginary quadratic fields whose ring of integers is norm-Euclidean by extending the techniques of~\cite{bruinier-raum-2015}.
The restriction to norm-Euclidean rings of integers is needed to generalize a key estimate in~\cite{bruinier-raum-2015}.
Recently, Aoki, Ibukiyama, and Poor~\cite{aoki-ibukiyama-poor-2024-preprint} revisited the Siegel modular case and established convergence for Siegel modular forms of general paramodular level and trivial arithmetic type in genus~$2$.
They introduced an approach that circumvents this estimate and advanced ideas from earlier work of Aoki~\cite{aoki-2000,aoki-2022} and Poor--Yuen~\cite{poor-yuen-2007}.
In parallel to this recent development, algebro-geometric approaches to the case of Siegel modular forms were suggested in joint work with Bruinier, by Flores, and by Fan~\cite{bruinier-raum-2024,flores-2024-preprint,fan-2025-preprint}, partially resorting to algebraicity results in~\cite{bruinier-raum-2015}.
Pollack~\cite{pollack-2024a-preprint,pollack-2024b-preprint} established an alternative approach to automatic convergence via coefficient estimates as opposed to vanishing results, by leveraging tail estimates that classically appear when deducing Hecke bounds.
This allowed him, independently of the present work, to settle the case of formal Hermitian modular forms of cogenus~$1$ for~$F = \QQ$, and more generally formal modular forms of cogenus~$1$ on tube domains~\cite{pollack-2026-preprint}.

We now outline the proofs of our results.
The bulk of the argument establishes \fref{Theorem}{mainthm:convergence}; \fref{Theorem}{mainthm:kudla_conjecture} then follows by verifying that the Kudla generating series satisfies the hypotheses of \fref{Theorem}{mainthm:convergence}, as we explain at the end of this outline.

Writing~$\tau_1$, $z$, $w$, and~$\tau_2$ for the block entries of~$\tau$ as in~\eqref{eq:HSg_decomposition}, a symmetric formal Fourier--Jacobi series is built from genuine Hermitian Hilbert--Jacobi forms, which are functions in~$\tau_1$, $z$, and~$w$, while the series in~$\tau_2$ is a priori only formal.
The core of our argument shows that it converges absolutely and locally uniformly.
Once convergence is established, full modularity follows from \fref{Proposition}{prop:siegel_and_klingen_parabolics_generate}, which asserts that the stabilizers of the symmetry condition~\eqref{eq:hermitian_fourier_coefficient_symmetry} and of the Fourier--Jacobi expansion together generate~$\SU{g,g}(\cOF)$.
Its proof invokes a $\rmK$\nbd{}theoretic consideration which is by now classical and can be found in a textbook by Hahn--O'Meara~\cite{hahn-omeara-1989}.
Furthermore, arguments that carry over from the Siegel modular setting reduce the vector-valued case to the scalar-valued one.
We therefore focus on trivial arithmetic types, for which the convergence proof requires two kinds of input.
First, an algebraicity result for formal Fourier--Jacobi series over the graded ring of Hilbert modular forms, and second, quantitative bounds on the Fourier coefficients of Hilbert--Jacobi forms.

We start to develop the first of these inputs in \fref{Section}{sec:hilbert_jacobi_forms}, where we establish vanishing bounds for Hilbert--Jacobi forms via specializations to torsion points.
Evaluating a Hilbert--Jacobi form at suitable torsion points produces Hilbert modular forms whose orders of vanishing can be controlled effectively, yielding uniform bounds on the vanishing order of nonzero Hilbert--Jacobi forms.
This replaces an argument in previous work~\cite{bruinier-raum-2015} that was based on the theta decomposition of Jacobi forms.
In \fref{Section}{sec:hilbert_jacobi_forms} we also recall uniform dimension estimates going back to Skoruppa and Boylan.
These bounds and estimates together supply the numerical input required for later dimension arguments.
In \fref{Section}{ssec:hermitian_modular_jacobi_forms:jacobi_forms} we transfer them from Hilbert--Jacobi forms to Hermitian Hilbert--Jacobi forms.
In \fref{Section}{ssec:convergence:algebraic} we combine these ingredients with two further tools to obtain the desired algebraicity result.
It is at this stage that the symmetry condition with respect to~$a \in \GL{g}(\cOE)$ with~$\det(a) = \det(\ov{a})$ enters for the first time: we invoke it through a reduction theory for Hermitian forms, built on classical Minkowski theory.
We also employ combinatorial tools to count the number of indices~$m$ in a formal Fourier--Jacobi series that contribute potentially independent terms.
This part of the proof applies directly for arbitrary~$1 \le h < g$ without further reduction steps.
It is noteworthy that the combinatorial estimates for formal series match the dimension estimates for Hermitian modular forms in the order of growth in the weight~$k$, though the leading constants are a priori different.
By contrast, the work of Ibukiyama, Poor, and Yuen on genus~$2$ Siegel modular forms of trivial arithmetic type relies on the fact that the combinatorial dimension estimate for symmetric formal Fourier--Jacobi series agrees exactly with the known dimension formula for Siegel modular forms.

The second main input is developed in the remaining part of \fref{Section}{sec:convergence}.
It builds on the previously established algebraicity, and in \fref{Sections}{ssec:convergence:torsion_points} and~\ref{ssec:convergence:torsion_point_subvarieties} relies on the stronger assumption that a symmetric formal Fourier--Jacobi series satisfies a monic algebraic equation over the graded ring of Hermitian Hilbert modular forms.
To explain the underlying idea, we consider the geometry of the modular variety and its uniformization~$\HSEg$.
The argument applies in cogenus~$h = 1$; the general case is handled in \fref{Section}{sec:vector_valued_higher_cogenus} via a reduction argument.
For~$h = 1$, the upper half space~$\HSEg$ admits a foliation over adapted coordinates~$\alpha, \beta \in (E \otimes_\QQ \RR)^{g-1}$ related to the previous coordinates by~$z = \alpha \tau_1 + \beta$ and~$w = \ov{\alpha}\, \lTr{\tau}{_1} + \ov{\beta}$.
A dense set of leaves, those with~$\alpha, \beta \in E^{g-1}$, gives rise to modular subvarieties of~$\SU{g,g}(\cOF) \backslash \HSEg$ for~$\SU{g-1,g-1} \times \SU{1,1}$.
In the language of Hilbert--Jacobi forms, these~$\alpha$ and~$\beta$ correspond to torsion points, terminology that we adopt for Hermitian Hilbert modular forms as well.

Restricting a formal Fourier--Jacobi series, which is a function in~$z$ and~$w$, to a leaf of this foliation produces a formal series in~$m$ of functions on~$\HS_{E,g-1}$.
For torsion points~$\alpha$ and~$\beta$, these functions are Hermitian Hilbert modular forms of genus~$g-1$ for a suitable subgroup of~$\SU{g-1,g-1}(\cOF)$.
The symmetry condition, which here enters for the second time, allows us to bound their Fourier coefficients in terms of finitely many Fourier--Jacobi coefficients, and to derive a uniform bound for their norm in terms of the index~$m$.
A standard argument from complex analysis then shows that a monic algebraic equation for a formal Fourier--Jacobi series upgrades these uniform norm bounds to a convergence statement on the dense set of modular leaves.
Combining uniform pointwise bounds with analytic convergence on a dense subset further yields convergence on the whole upper half space~$\HSEg$.
Together with the generation of~$\SU{g,g}(\cOF)$ by the relevant subgroups, this establishes that every symmetric formal Fourier--Jacobi series satisfying a monic algebraic relation over Hermitian Hilbert modular forms is itself a Hermitian Hilbert modular form.

The preceding step is where mere algebraicity would not suffice.
By a standard argument from commutative algebra, algebraicity implies that symmetric formal Fourier--Jacobi series are quotients of Hermitian Hilbert modular forms, that is, meromorphic Hermitian Hilbert modular forms.
To conclude the proof of \fref{Theorem}{mainthm:convergence} in the scalar-valued cogenus~$h = 1$ case, the remaining part of \fref{Section}{sec:convergence} shows that they are in fact holomorphic.
The key observation is that the uniform boundedness on modular leaves does not require a monic relation, but holds for all symmetric formal Fourier--Jacobi series.
To exploit this, in \fref{Section}{ssec:convergence:torsion_point_divisors} we establish a transversality result for divisors on the modular variety with respect to the foliation of~$\HSEg$.
At this stage, we already have full modular invariance of the pullback of the divisor to~$\HSEg$ at our disposal and need transversal intersection with only a single modular leaf.
We obtain such a leaf by invoking the Borel's Density Theorem for the inclusion of~$\SU{g,g}(\cOF)$ in~$\SU{g,g}(F \otimes_\QQ \RR)$ and componentwise irreducibility of the Lie algebra of~$\SU{g,g}(F \otimes_\QQ \RR)$ under the adjoint representation.
These yield a tangent vector at a point of the divisor that allows us to move to a modular leaf via a differential-geometric argument.
This completes the proof of automatic convergence in the scalar cogenus~$1$ setting.

The remaining steps, carried out in \fref{Section}{sec:vector_valued_higher_cogenus}, upgrade this to the full \fref{Theorem}{mainthm:convergence}.
We first reduce vector-valued arithmetic types to the trivial arithmetic type by pairing with Hermitian Hilbert modular forms for dual representations and passing to suitable scalar-valued components.
We then reduce higher cogenus to cogenus~$1$ by induction on~$h$, using a formal theta decomposition that expresses the relevant Fourier--Jacobi coefficients in terms of symmetric formal Fourier--Jacobi series of lower cogenus.
Combining these reductions with the cogenus~$1$ case proves \fref{Theorem}{mainthm:convergence}.

Finally, in \fref{Section}{sec:kudla_conjecture} we deduce \fref{Theorem}{mainthm:kudla_conjecture} from \fref{Theorem}{mainthm:convergence}.
We verify that~$\thetaKudla_L$ satisfies the symmetry relation of \fref{Definition}{def:symmetric_formal_fourier_jacobi_series} and that its Fourier--Jacobi coefficients are Hermitian Hilbert--Jacobi forms, reducing to the known divisor case.
Consequently,~$\thetaKudla_L$ defines a symmetric formal Fourier--Jacobi series of cogenus~$g-1$, and \fref{Theorem}{mainthm:convergence} implies \fref{Theorem}{mainthm:kudla_conjecture}.
The argument goes back to Zhang and can also be found in the work of Liu; for convenience we give a full account.
The key idea is to express the Fourier--Jacobi coefficients of the Kudla generating series as a finite sum of pushforwards of Kudla generating series on Shimura subvarieties, or equivalently, to detect modularity of these contributions via pushforward from these subvarieties.

At the end of \fref{Section}{sec:kudla_conjecture}, we briefly indicate how \fref{Corollary}{maincor:li_liu_hypothesis_removed} follows from \fref{Theorem}{mainthm:kudla_conjecture}.

\begin{mainremark}
An analogous automatic convergence result can likely be proved for Hilbert--Siegel modular forms.
The vanishing theorems in \fref{Section}{sec:hilbert_jacobi_forms} can be used verbatim together with an easier reduction theory and better available estimates in~\fref{Section}{sec:convergence}, while the sheaf theoretic and combinatorial arguments in~\fref{Section}{sec:vector_valued_higher_cogenus} require only slight adjustments, which are easier than the treatment of the CM case.
Once detailed arguments are available, this would yield the geometric orthogonal Kudla Conjecture over totally real fields by verifying the assumptions of~\cite[Theorem~1.2]{yuan-zhang-zhang-2009}.
\end{mainremark}

\subsection*{Acknowledgement}

The author deeply thanks Jan Bruinier for helpful and inspiring conversations around the topic of this work.

%% file: sections/01_hilbert_jacobi_forms.tex
\section{Hilbert--Jacobi forms}%
\label{sec:hilbert_jacobi_forms}

Throughout, given expressions~$f(x)$ and~$g(x)$, we write~$f(x) \ll g(x)$ if there is a positive constant~$c$ such that~$|f(x)| < c g(x)$ for all~$x$, where the domain of~$x$ is specified if it is not implicit.
We amend subscripts, $f(x) \ll_y g(x)$, to indicate that~$c$ depends on~$y$.
We write~$f(x) \asymp g(x)$ if~$f(x) \ll g(x)$ and~$g(x) \ll f(x)$, and~$f(x) \asymp_y g(x)$ if~$f(x) \ll_y g(x)$ and~$g(x) \ll_y f(x)$.
If~$f(x) \ll g(x)$ appears as part of an assumption, we mean that there is a constant~$c$ such that the conclusion holds for all~$x$ with~$f(x) < c g(x)$.

We write~$e(x) \defeq \exp(2 \pi i\, \trace(x))$ for a complex square matrix~$x$.
For notational convenience, we extend this to square matrices~$x,y$ of possibly different size and write~$e(x + y)$ as shorthand for~$e(x) \cdot e(y)$.

We define~$\MatT{h}(\RR) \subset \Mat{h}(\RR)$ as the subset of symmetric matrices, that is, matrices invariant under transposition~$x \mto \lT{x}$.
We let~$\MatT{h}(\ZZ)^\vee$ be the dual of~$\MatT{h}(\ZZ)$ with respect to the trace form~$(x,y) \mto \trace(x y)$ on~$\Mat{h}(\RR)$, where~$\trace$ is the matrix trace.
It consists of matrices with integral diagonal and half-integral off-diagonal entries.
We set~$x[y] \defeq \lT{y} x y$ for matrices~$x \in \MatT{g}(\CC)$ and~$y \in \Mat{g,h}(\RR)$.
In particular, $\GL{h}(\ZZ)$ acts on~$\MatT{h}(\ZZ)^\vee$ from the right via~$(m,u) \mto m[u]$, since it acts on~$\MatT{h}(\ZZ)$.
For a symmetric matrix~$x$ and~$\alpha \in \RR$, we write~$x > \alpha$ or~$x \ge \alpha$ if~$x - \alpha$ is positive definite or positive semi-definite, respectively, that is, all eigenvalues of~$x$ are strictly larger than~$\alpha$ or at least~$\alpha$.

We fix a totally real field~$F$ of degree~$n_F$ with ring of integers~$\cOF$.
We let~$\traceF$ and~$\normF$ denote the Galois trace and norm of~$F \slash \QQ$.
Further, we write~$\detF = \normF \circ \det$.
The dual of~$\cOF$ with respect to the trace form is~$\cOF^\vee$.
We have~$\cOF^\vee = \cDF^{-1}$, where~$\cDF$ is the different ideal in~$\cOF$.
After fixing an order for the real embeddings of~$F$, we have an isomorphism of~$F_\RR \defeq F \otimes_\QQ \RR$ with~$\RR^{n_F}$, which extends to~$F_\CC \defeq F \otimes_\QQ \CC$.
It yields an action of~$F$ on~$\CC^{n_F}$ by multiplication.
Component-wise comparison in~$F_\RR$ yields a partial order, which we denote by~$<$; this is compatible with the embedding of~$\QQ$ into~$F$.
We extend~$e(x) = e(\traceF(x))$ to square matrices with entries in~$F_\CC$.

\subsection{Hilbert modular forms}%
\label{ssec:hilbert_jacobi_forms:hilbert_modular_forms}

The Poincar\'e upper half plane is~$\HS \defeq \{ \tau \in \CC \condsep \Im(\tau) > 0 \}$.
We have an action~$\SL{2}(\RR) \circlearrowright \HS$ via M\"obius transformations~$\begin{psmatrix} a & b \\ c & d \end{psmatrix} \tau \defeq \frac{a \tau + b}{c \tau + d}$.
We identify the Hilbert upper half space~$\HS_F \defeq \{ \tau \in F_\CC \condsep \Im(\tau) > 0 \}$ with~$\HS^{n_F}$.
The action of~$\SL{2}(\RR)$ extends to an action~$\SL{2}(F_\RR) \circlearrowright \HS_F$, and gives rise to a family of slash actions
\begin{gather*}
  \big( f \big|_k\, \begin{psmatrix} a & b \\ c & d \end{psmatrix} \big) (\tau)
  \defeq
  \normF(c \tau + d)^{-k}\,
  f \big( \begin{psmatrix} a & b \\ c & d \end{psmatrix} \tau \big)
\end{gather*}
on functions~$f \defcol \HS_F \ra \CC$ for every~$k \in \ZZ$.
By linearity this action extends to vector-valued functions~$f \defcol \HS_F \ra V$ into any complex vector space~$V$.

We have the full Hilbert modular group~$\SL{2}(\cOF)$ and, given a positive integer~$N$, its principal congruence subgroups
\begin{gather*}
  \SL{2}(\cOF, N)
  \defeq
  \big\{
  \begin{psmatrix} a & b \\ c & d \end{psmatrix} \in \SL{2}(\cOF)
  \condsep
  a,d \equiv 1 \,\pmod{N},\;
  b,c \equiv 0 \,\pmod{N}
  \big\}
  \tx{.}
\end{gather*}
A holomorphic function~$f \defcol \HS_F \ra V$ that is invariant under the slash action of all~$\begin{psmatrix} 1 & b \\ 0 & 1 \end{psmatrix} \in \SL{2}(\cOF, N)$ admits the Fourier series expansion
\begin{gather}
  \label{eq:hilbert_modular_fourier_expansion}
  f(\tau)
  =
  \sum_{n \in \frac{1}{N} \cOF^\vee}
  c(f; n)\, e(n \tau)
  \tx{,}\quad
  c(f; n) \in V
  \tx{.}
\end{gather}
We set~$c(f; n) = 0$ if~$n$ does not contribute to this expansion.

\begin{definition}%
\label{def:hilbert_modular_form}
Let~$k \in \ZZ$ and~$N \in \ZZ_{\ge 1}$.
Then a holomorphic function~$f \defcol \HS_F \ra \CC$ is called a \emph{Hilbert modular form} of weight~$k$ and level~$N$ if:
\begin{enumerateroman}%
\item
For all~$\ga \in \SL{2}(\cOF, N)$ we have $f |_k\,\ga = f$.

\item
If~$F = \QQ$, for all~$\ga \in \SL{2}(\ZZ)$ and all~$n < 0$ the Fourier coefficients in~\eqref{eq:hilbert_modular_fourier_expansion} satisfy~$c(f |_k\,\ga; n) = 0$.
\end{enumerateroman}
\end{definition}

For~$F \ne \QQ$ the growth condition is automatic by the Koecher principle.

\begin{theorem}%
\label{thm:hilbert_koecher_principle}
Given a Hilbert modular form~$f$ of weight~$k \in \ZZ$ and level~$N$, for all~$\ga \in \SL{2}(\cOF)$ we have~$c(f |_k\, \ga; n) = 0$ if~$n \not\ge 0$.
\end{theorem}

\begin{proof}
If~$F = \QQ$, this is part of the definition, and otherwise it is the statement of Theorem~1.4 of~\cite{garrett-1990} for the special case of parallel weight~$(k,\ldots,k) \in \ZZ^{n_F}$.
\end{proof}

We define the \emph{vanishing order} of a level~$N$ modular form~$f$ at infinity by
\begin{gather}%
  \label{eq:def:vanishing_order_hilbert_modular_form}
  \ord( f )
  \defeq
  \sup \big\{
  \nu \in \RR
  \condsep
  c(f; n) = 0
  \tx{ for all } n \in F
  \tx{ with } n < \nu
  \big\}
  \tx{.}
\end{gather}
There are two further notions of vanishing order at infinity, which in the case of~$F = \QQ$ agree with the previous one.
\begin{align*}
  \ordtrace(f)
   & \defeq
  \sup\big\{
  \nu \in \RR
  \condsep
  c(f; n) = 0
  \tx{ for all } n \in F
  \tx{ with } \traceF(n) < \nu
  \big\}
  \tx{,}
  \\
  \ordnorm(f)
   & \defeq
  \sup\big\{
  \nu \in \RR
  \condsep
  c(f; n) = 0
  \tx{ for all } n \in F
  \tx{ with } \normF(n) < \nu^{n_F}
  \big\}
  \tx{.}
\end{align*}
Note that by the Koecher principle we have~$\ord(f), \ordnorm(f), \ordtrace(f) \ge 0$.

\begin{lemma}%
\label{la:vanishing_order_equivalence_hilbert_modular_form}
Given a Hilbert modular form~$f$ of weight~$k \in \ZZ$, we have
\begin{gather*}
  \ord(f)
  \asymp_F
  \ordtrace(f)
  \asymp_F
  \ordnorm(f)
  \tx{.}
\end{gather*}
\end{lemma}

The proof of \fref{Lemma}{la:vanishing_order_equivalence_hilbert_modular_form} relies on the following balancing by units, which we will reuse in several places throughout this work.

\begin{lemma}%
\label{la:unit_balancing}
Let\/~$U \subseteq \cOF^\times$ be a subgroup of finite index.
Then for every totally positive~$n \in F_\RR$ there is\/~$\eps \in U$ such that\/~$\eps\, n$ is totally positive and its components in~$F_\RR \cong \RR^{n_F}$ satisfy\/~$(\eps\, n)_i \asymp_{F,U} (\eps\, n)_j$ for all~$1 \le i, j \le n_F$.
\end{lemma}

\begin{proof}
By the Dirichlet unit Theorem, the image of~$\cOF^\times$ under the logarithmic embedding sending~$\eps$ to~$(\log |\eps_i|)_{1 \le i \le n_F}$ is a lattice of full rank~$n_F - 1$ in the trace-zero hyperplane of~$\RR^{n_F}$.
The image of~$U$ is a sublattice of finite index, hence also of full rank.
It suffices to prove the assertion for~$U \cap (\cOF^\times)^2$, which is still a finite-index subgroup of~$\cOF^\times$ and is contained in the original~$U$.
Thus we may assume that all~$\eps \in U$ are totally positive.

For every totally positive~$n \in F_\RR$ we therefore find~$\eps \in U$ whose logarithm lies within a bounded distance, depending only on~$F$ and~$U$, from
\begin{gather*}
  \Bigl(
  \mfrac{\log\, \normF(n)}{n_F} - \log n_i
  \Bigr)_{1 \le i \le n_F}
  \tx{.}
\end{gather*}
The components of~$\eps\, n$ are then pairwise of the same size up to a constant depending only on~$F$ and~$U$.
\end{proof}

\begin{proof}[Proof of \fref{Lemma}{la:vanishing_order_equivalence_hilbert_modular_form}]
We first show that
\begin{gather*}
  \ord(f) = 0
  \Leftrightarrow
  \ordnorm(f) = 0
  \Leftrightarrow
  \ordtrace(f) = 0
  \tx{.}
\end{gather*}
If~$c(f; 0) \ne 0$, then~$\ord(f) = \ordtrace(f) = \ordnorm(f) = 0$, since all vanishing orders are non-negative.
Conversely, if~$c(f; 0) = 0$ then~$\ord(f), \ordtrace(f), \ordnorm(f) > 0$ by discreteness of~$\cOF \subset F_\RR$.

Assume now that~$\ord(f) \ne 0$.
Applying \fref{Lemma}{la:unit_balancing} with~$U = (\cOF^\times)^2$, for every totally positive~$n$ we find~$\eps \in \cOF^\times$ such that the components of~$n' = \eps^2\, n \in F_\RR \cong \RR^{n_F}$ satisfy~$n'_i \asymp_F n'_j$.
This yields the estimates
\begin{gather*}
  \traceF(n') \asymp_F \normF(n')^{1 \slash n_F}
  \quad\tx{and}\quad
  \traceF(n') \ll_F n' < \traceF(n')
  \tx{.}
\end{gather*}
Further, invariance of~$f$ with respect to the action of~$\diag(\epsilon, \epsilon^{-1})$ yields~$|c(f; n)| = |c(f; \eps^2 n)|$.
Combining the above estimates with the definition of vanishing orders finishes the proof.
\end{proof}

We have the following Hecke bounds, which will be crucial in \fref{Section}{ssec:hilbert_jacobi_forms:hecke_bound}.

\begin{proposition}%
\label{prop:hecke_bound}
For a Hilbert modular form~$f$ of weight~$k$ and level~$1$, we have~$f = 0$ if\/~$\ord(f) \gg_F k$.
\end{proposition}

\begin{proof}
This is stated at the end of Section~1.7, on p.~33, of~\cite{garrett-1990}, where~$n$ in that section is~$\ordnorm(f)$.
\end{proof}

\begin{lemma}%
\label{la:accumulated_hecke_bound_twisted_permutation_type}
Given a finite set~$B$ of size~$\# B$ with a transitive right action of\/~$\SL{2}(\cOF)$, let~$f_b$ for~$b \in B$ be Hilbert modular forms of weight~$k$ and level~$N$ with the property that
\begin{gather*}
  f_b \big|_k\, \ga
  =
  c_{b,\ga}\, f_{b \ga}
  \quad\tx{for some $N$\thdash{} roots of unity }
  c_{b,\ga}
  \tx{.}
\end{gather*}
Then we have~$f_b = 0$ for all~$b \in B$ if\/~$\sum_{b \in B} \ord(f_b) \gg_F k\, \# B$.
\end{lemma}

\begin{proof}
The product~$f \defeq \prod_{b \in B} f_b$ transforms as
\begin{gather*}
  f \big|_{k\, \#B}\, \ga
  =
  \prod_{b \in B}
  f_b \big|_k\, \ga
  =
  c_\ga\, f
  \quad\tx{with }
  c_\ga
  =
  \prod_{b \in B} c_{b, \ga}
  \in
  \CC \setminus \{ 0 \}
  \tx{.}
\end{gather*}
It is a modular form for a character~$\chi$ of~$\SL{2}(\cOF)$ of weight~$k\, \# B$.
By our assumptions~$\chi$ has finite order dividing~$N$.
As a consequence, $f^N$ is a holomorphic modular form of weight~$N k\, \# B$ for the full modular group.

The vanishing order of~$f^N$ is~$N$ times the sum of the vanishing orders of all~$f_b$.
By our assumptions, it exceeds~$N k\, \# B$ up to a suitable scalar constant that only depends on~$F$, which we may choose compatible with \fref{Proposition}{prop:hecke_bound}.
Therefore~$f^N = 0$ by that proposition.
We conclude that~$f_{b_0} = 0$ for at least one~$b_0 \in B$.
Since~$\SL{2}(\cOF)$ acts transitively on~$B$, given~$b \in B$ we find~$\ga \in \SL{2}(\cOF)$ with~$b = b_0 \ga$ and hence
\begin{gather*}
  0
  =
  f_{b_0} \big|_k\, \ga
  =
  c_{b_0,\ga}\, f_{b_0 \ga}
  =
  c_{b_0,\ga}\, f_b
  \tx{.}
\end{gather*}
Since~$c_{b_0,\ga} \ne 0$, we conclude~$f_b = 0$ as desired.
\end{proof}

\subsection{Basics and definition of Jacobi forms}%
\label{ssec:hilbert_jacobi_forms:jacobi_forms_definition}

We focus on Hilbert--Jacobi forms of genus~$1$, but let the cogenus~$h \ge 1$ be arbitrary.
Since we only study spaces of Jacobi forms, as opposed to algebras, we can incorporate the Jacobi index into the Heisenberg group.
A general reference is the thesis of Boylan~\cite{boylan-2015}.

A Jacobi index is a lattice~$L$ over~$F$, that is, a finitely generated torsion-free $\cOF$\nbd{}module with symmetric bilinear form $\langle \,\cdot\, , \,\cdot\, \rangle = \langle \,\cdot\, , \,\cdot\, \rangle_L$.
We write~$\langle \la \rangle = \langle \la, \la \rangle_L$ to avoid notational redundancy.
We assume throughout that~$L$ is integral, that is, the bilinear form takes values in~$\cOF^\vee$.
The dimension of~$L_\QQ = L \otimes \QQ$ is the co-genus~$h$.
We set~$L_\RR = L \otimes \RR$ and~$L_\CC = L \otimes \CC$.
The dual~$L^\vee$ of~$L$ is defined as~$\{x \in L_\QQ \condsep \langle x, y \rangle_L \in \cOF^\vee \tx{ for all } y \in L \}$, featuring the inverse different as opposed to the ring of integers.
The trace lattice~$\traceF(L)$ of~$L$ is the lattice over~$\QQ$ with same underlying module as~$L$ and bilinear form~$\traceF \circ \langle \,\cdot\, , \,\cdot\, \rangle_L$.

Following Cassels~\cite{cassels-1997}, we define the successive minima~$\la_i(L, K)$, $1 \le i \le h$, of a totally positive definite lattice~$L$ with respect to a convex compact set~$K$ in~$L_\RR$ as follows:
\begin{gather*}
  \la_i(L, K)
  =
  \inf\bigl\{
  \la \in \RR_{> 0}
  \condsep
  \dim_F \linspan_F(\la K \cap L) \ge i
  \bigr\}
  \tx{.}
\end{gather*}
By \fref{Lemma}{la:unit_balancing}, we have
\begin{gather*}
  \la_i(L)
  \defeq
  \la_i\bigl( L, B_L(1) \bigr)
  \asymp_F
  \la_i\bigl( L, B_{L,\mathrm{Tr}}(1) \bigr)
  \asymp_F
  \la_i\bigl( L, B_{L,\mathrm{Nm}}(1) \bigr)
  \tx{,}
\end{gather*}
where~$B_L(1) = \{ x \in L_\QQ \condsep \frac{1}{2} \langle x \rangle \le 1 \}$ is the~$1$\nbd{}ball with respect to the partial order on~$F_\RR$, and~$B_{L,\mathrm{Tr}}(1)$ and~$B_{L,\mathrm{Nm}}(1)$ are defined correspondingly by composing the quadratic form with the trace and norm of~$F / \QQ$.
We conclude that~$\la_h(L) \asymp_F \la_{n h}(\traceF(L))$ by comparing the successive minima with respect to the trace-ball and using that~$F$ has an integral basis.

If~$L$ is free, we identify it with~$\cOF^h$ and let~$\langle x, y \rangle_L = 2 \lT{x} m y$ for suitable symmetric matrix~$m = m_L \in \MatT{h}(F)$.
To express the bilinear form of general~$L$, once and for all fix integral ideals~$\fraka$ of minimal norm representing each element of the class group of~$F$.
Using the Steinitz form as in~\cite[Proposition~81:5]{omeara-2000} and finiteness of the class group, we can identify~$L$ with~$\fraka^{-1} \oplus \cOF \oplus \cdots \oplus \cOF$ for one such~$\fraka$.
Then~$\langle x, y \rangle_L = 2 \lT{x} m y$ for suitable~$m$ as before.
These identifications also allow us to write~$L_\RR = F_\RR^h$ and~$L_\CC = F_\CC^h$.

The Jacobi upper half space associated with~$L$ is~$\HS_F \times L_\CC$.
We usually write~$(\tau,z)$ for its elements, and refer to~$\tau$ as the modular variable and~$z$ the elliptic variable.
The Jacobi upper half space carries an action of the Jacobi group, which we define next.
Its definition requires the Heisenberg group with underlying set
\begin{gather}
  \label{eq:def:heisenberg_group_set}
  \Hb{L}(\RR)
  \defeq
  \big\{
  ((\la, \mu), \ka)
  \in
  \big( L_\RR \oplus L_\RR \big) \times F_\RR
  \big\}
  \tx{.}
\end{gather}
The elements of~$\Hb{L}(\RR)$ are alternatively written as triples~$(\la,\mu,\ka)$.
The product in~$\Hb{L}(\RR)$ is given by
\begin{gather}
  \label{eq:def:heisenberg_group_product}
  \big( (\la_1, \mu_1), \ka_1 \big)
  \cdot
  \big( (\la_2, \mu_2), \ka_2 \big)
  \defeq
  \big(
  (\la_1 + \la_2, \mu_1 + \mu_2),\,
  \ka_1 + \ka_2
  +
  \tfrac{1}{2}
  \langle \la_1, \mu_2 \rangle_L
  -
  \tfrac{1}{2}
  \langle \mu_1, \la_2 \rangle_L
  \big)
  \tx{.}
\end{gather}

The Jacobi group is a semi-direct product
\begin{gather}
  \label{eq:def:jacobi_group}
  \Jac{L}(\RR)
  \defeq
  \SL{2}(F_\RR) \ltimes \Hb{L}(\RR)
  \tx{,}
\end{gather}
where the action of~$\SL{2}(F_\RR)$ is the right action defined by
\begin{gather}
  \label{eq:def:heisenberg_group_sl2_action}
  \Hb{L}(\RR) \circlearrowleft \SL{2}(F_\RR)
  \defcol
  \big( \la, \mu, \ka \big)\, \ga
  =
  \big( (\la, \mu), \ka \big)\, \ga
  \defeq
  \big( (\la, \mu) \ga,\, \ka \big)
  \tx{,}
\end{gather}
and~$\ga$ acts on~$(\la, \mu)$ by vector-matrix multiplication.
In formulas, multiplication in the Jacobi group is given by
\begin{gather}
  \label{eq:def:jacobi_group_product}
  \big( \ga_1, (\la_1, \mu_1, \ka_1 ) \big)
  \cdot
  \big( \ga_2, (\la_2, \mu_2, \ka_2 ) \big)
  =
  \big(
  \ga_1 \ga_2,\,
  (\la_1, \mu_1, \ka_1 ) \ga_2
  \cdot
  (\la_2, \mu_2, \ka_2 )
  \big)
  \tx{.}
\end{gather}
For~$\la, \mu \in L_\RR$, we define the element
\begin{gather}
  \label{eq:def:jacobi_group_translation}
  \transJ(\la, \mu)
  \defeq
  \bigl(
  \idmat{2},\,
  (\la, \mu, \tfrac{1}{2} \langle \la, \mu \rangle_L)
  \bigr)
  \in
  \Jac{L}(\RR)
  \tx{.}
\end{gather}
These elements will appear in our treatment of specializations to torsion points.

We have the integral points of the Heisenberg-type group and those of the Jacobi group, which are defined via the semi-direct product decomposition in~\eqref{eq:def:jacobi_group}:
\begin{gather*}
  \Hb{L}(\ZZ)
  \defeq
  \big\{
  (\la, \mu, \ka) \in \Hb{L}(\RR)
  \condsep
  \la, \mu \in L,\,
  \ka \in \tfrac{1}{2} \cOF
  \big\}
  \quad\tx{and}\quad
  \Jac{L}(\cOF)
  \defeq
  \SL{2}(\cOF) \ltimes \Hb{L}(\ZZ)
  \tx{.}
\end{gather*}

The actions of~$\SL{2}(F_\RR)$ and~$\Hb{L}(\RR)$ on the Jacobi upper half space~$\HS_F \times L_\CC$ are defined by
\begin{gather}
  \label{eq:def:jacobi_group_action_upper_half_space}
  \begin{psmatrix}a & b \\ c & d\end{psmatrix}\,
  (\tau, z)
  \defeq
  \bigl( \begin{psmatrix}a & b \\ c & d\end{psmatrix} \tau, z (c \tau + d)^{-1} \bigr)
  \quad\tx{and}\quad
  (\la, \mu, \kappa)\,
  (\tau, z)
  \defeq
  \big( \tau, z + \la \tau + \mu \big)
  \tx{.}
\end{gather}
The definition of Jacobi forms requires also a slash action, which we associate to a weight~$k \in \ZZ$.
For consistency with the classical setup, we include the Jacobi index~$L$ in our notation.
For a given function~$\phi$ on~$\HS_F \times L_\CC$, we set
\begin{gather}
  \label{eq:def:jacobi_slash_action_classical}
  \begin{aligned}
  \Bigl(
  \phi |_{k,L}\,
  \begin{psmatrix} a & b \\ c & d \end{psmatrix}
  \Bigr)(\tau, z)
   & \defeq
  \normF(c \tau + d)^{-k}\,
  e\big(
  - \tfrac{1}{2} \langle z \rangle\, c (c \tau + d)^{-1}
  \big)\,
  \phi\bigl(
  \begin{psmatrix} a & b \\ c & d \end{psmatrix} (\tau, z)
  \bigr)
  \tx{,}
  \\
  \big(
  \phi |_{k,L}\,
  (\la, \mu, \ka)
  \big)(\tau, z)
   & \defeq
  e\bigl(
  \tfrac{1}{2} \langle \la \rangle \tau
  + \langle \la, z \rangle
  + \tfrac{1}{2} \langle \la, \mu \rangle + \ka
  \bigr)\,
  \phi\big(
  \tau,
  z + \la \tau + \mu
  \big)
  \tx{.}
  \end{aligned}
\end{gather}
This defines an action of~$\Jac{L}(\RR)$ on holomorphic functions on~$\HS_F \times L_\CC$.
In the same way as we allow ourselves to identify~$L$ with a representing matrix~$m$, we may write~$|_{k,m}$ if~$L$ is free.

A holomorphic function~$\phi \defcol \HS_F \times L_\CC \ra \CC$ that is invariant under the slash action of~$\begin{psmatrix} 1 & b \\ 0 & 1 \end{psmatrix} \in \SL{2}(\cOF)$ and under that of~$(0,\mu,0) \in \Hb{L}(\ZZ)$ admits a Fourier series expansion
\begin{gather}
  \label{eq:def:jacobi_forms_fourier_expansion}
  \phi (\tau, z)
  =
  \sum_{\substack{n \in \cOF^\vee \\ r \in L^\vee}}
  c(\phi; n, r)\,
  e(n \tau + \langle r, z \rangle)
  \tx{,}\quad
  c(\phi; n, r) \in \CC
  \tx{.}
\end{gather}
Throughout, we set~$c(\phi; n, r) = 0$ if~$(n,r)$ does not contribute to this expansion.

\begin{definition}%
\label{def:jacobi_forms_hilbert}
Given~$k \in \ZZ$ and an integral lattice~$L$ over~$F$, a holomorphic function~$\phi \defcol \HS_F \times L_\CC \ra \CC$ is called a \emph{Hilbert--Jacobi form} of genus~$1$, weight~$k$, and Jacobi index~$L$ if:
\begin{enumerateromanseries}%
{en:def:jacobi_forms_hilbert}
\item
For all~$\ga \in \Jac{L}(\ZZ)$ we have~$\phi |_{k,L}\,\ga = \phi$.

\item
\label{it:def:jacobi_forms_hilbert:fourier_support}
If~$F = \QQ$, under the identification of~$L$ with~$m \in \MatT{h}(\ZZ)^\vee$, the Fourier coefficients in~\eqref{eq:def:jacobi_forms_fourier_expansion} satisfy
\begin{gather*}
  c(\phi; n, r) = 0
  \quad\tx{for all }
  n \in \ZZ, r \in \ZZ^h
  \tx{ such that }
  \begin{psmatrix} 2 n & \lT{r} \\ r & 2 m \end{psmatrix} \not\ge 0
  \tx{, i.e., \emph{not} positive semi-definite}
  \tx{.}
\end{gather*}
\end{enumerateromanseries}
\end{definition}
The space of Jacobi forms is written as~$\rmJHilb{k,L}$, where the cogenus can be inferred from the dimension of~$L_\QQ$.

As in the case of Hilbert modular forms, for~$F \ne \QQ$ the growth condition is automatic by the Koecher principle, if~$L$ is totally positive definite.
We will treat the case of general~$L$ in \fref{Section}{ssec:hilbert_jacobi_forms:positive_definite_jacobi_indices}.

\begin{theorem}%
\label{thm:hilbert_jacobi_koecher_principle}
Given a Hilbert--Jacobi form~$\phi$ of weight~$k \in \ZZ$ and totally positive definite Jacobi index~$L$, for all~$\ga \in \SL{2}(F)$ we have~$c(\phi |_{k,L}\, \ga; n, r) = 0$ if~$n \not\ge \frac{1}{2} \langle r \rangle_L$.
\end{theorem}

\begin{proof}
If~$F = \QQ$, this is part of the definition, observing that~$\SL{2}(\ZZ)$ has only one cusp.
Otherwise this is stated in Theorem~3.2 of~\cite{boylan-2015}.
\end{proof}

The \emph{vanishing order} at infinity of a Hilbert--Jacobi form~$\phi \in \rmJHilb{k,L}$ is defined as
\begin{gather}
  \label{eq:def:vanishing_order_jacobi_form}
  \ord(\phi)
  \defeq
  \sup\big\{
  \nu \in \RR
  \condsep
  c(\phi; n, r) = 0
  \tx{ for all } n \in F, r \in L_\QQ
  \tx{ with } n < \nu
  \big\}
  \tx{.}
\end{gather}
As in the case of Hilbert modular forms, there are two further notions of vanishing order at infinity, which by the Koecher principle take non-negative values.
\begin{align*}
  \ordtrace(\phi)
   & \defeq
  \sup\big\{
  \nu \in \RR
  \condsep
  c(\phi; n, r) = 0
  \tx{ for all } n \in F, r \in L_\QQ
  \tx{ with } \traceF(n) < \nu
  \big\}
  \tx{,}
  \\
  \ordnorm(\phi)
   & \defeq
  \sup\big\{
  \nu \in \RR
  \condsep
  c(\phi; n, r) = 0
  \tx{ for all } n \in F, r \in L_\QQ
  \tx{ with } \normF(n) < \nu^{n_F}
  \big\}
  \tx{.}
\end{align*}
We let~$\rmJHilb{k,L}[\nu]$ denote the space of Hilbert--Jacobi forms of vanishing order~$\ordnorm$ at least~$\nu$.
We will use this notation later in \fref{Section}{ssec:hermitian_modular_jacobi_forms:connection_hilbert_jacobi_forms}.

\begin{lemma}%
\label{la:vanishing_order_equivalence_jacobi_form}
Given a Hilbert--Jacobi form~$\phi$ of weight~$k \in \ZZ$ and totally positive definite Jacobi index~$L$, we have
\begin{gather*}
  \ord(\phi)
  \asymp_F
  \ordtrace(\phi)
  \asymp_F
  \ordnorm(\phi)
  \tx{.}
\end{gather*}
\end{lemma}

\begin{proof}
The proof is essentially the same as the one of \fref{Lemma}{la:vanishing_order_equivalence_hilbert_modular_form} except for one adjustment:
Invariance of~$\phi$ with respect to the action of~$\diag(\epsilon, \epsilon^{-1})$ yields~$|c(\phi;\, n, r)| = |c(\phi;\, \eps^2 n, \eps r)|$.
Since the definition of vanishing orders quantifies over all~$r$, that proof extends.
\end{proof}

We define the specialization to the torsion point~$z = \alpha \tau + \beta$, $\alpha, \beta \in L_\QQ$ by
\begin{gather}
  \label{eq:def:jacobi_forms_torsion_points}
  \phi[\alpha, \beta](\tau)
  \defeq
  \big( \phi \big|_{k,m}\, \transJ(\alpha, \beta) \big)(\tau, 0)
  \tx{.}
\end{gather}
Expanding the action of~$\transJ(\alpha, \beta)$, we find that
\begin{gather}
  \label{eq:jacobi_forms_torsion_points_explicit}
  \phi[\alpha, \beta](\tau)
  =
  e\bigl( \tfrac{1}{2} \langle \alpha \rangle_L \tau + \langle \alpha, \beta \rangle_L \bigr)\,
  \phi(\tau, \alpha \tau + \beta)
  \tx{.}
\end{gather}
If~$L$ is totally positive definite, the Koecher principle implies that if~$\phi$ is a Hilbert--Jacobi form, then~$\phi[\alpha,\beta]$ has Fourier series expansion supported on zero and on totally positive indices.

\subsection{Reduction to positive definite Jacobi indices}%
\label{ssec:hilbert_jacobi_forms:positive_definite_jacobi_indices}

The goal of this section is to reduce the theory of Hilbert--Jacobi forms to the case of positive definite Jacobi indices.
We employ the theory of elliptic functions, which does not need the Hilbert modular transformation behavior.

\begin{proposition}%
\label{prop:jacobi_forms_semi_definite_index}
Let~$\phi$ be a Hilbert--Jacobi form of weight~$k \in \ZZ$ and Jacobi index~$L$.
If\/~$L$ is \emph{not} totally positive semi-definite, then~$\phi = 0$.
If\/~$L$ is totally positive semi-definite and~$s \in L_\RR$ satisfies~$\langle s \rangle_L = 0$, then~$\phi(\tau, z + s z') = \phi(\tau, z)$ for all~$(\tau,z) \in \HS_F \times L_\CC$ and~$z' \in \CC$.
\end{proposition}

\begin{proof}
Throughout, we let~$(\tau', z') \in \HS \times \CC$.
We consider~$\xi \in \cOF$ totally positive, $s \in L$, and~$z_0 \in L_\CC$ with~$\langle s, z_0 \rangle_L = 0$.
We set~$m' = \traceF( \xi\, \frac{1}{2} \langle s \rangle_L )$ and examine~$\psi(z') = \phi(\xi^{-1} \tau', z_0 + s z')$ as a function in~$z' \in \CC$.
By modular invariance of~$\phi$ under~$\Hb{L}(\ZZ)$, it satisfies~$\psi(z' + 1) = \psi(z')$ and
\begin{gather}
  \label{eq:prf:prop:jacobi_forms_semi_definite_index:elliptic_function}
  \begin{aligned}
  \psi(z' + \tau)
   & =
  \phi(\xi^{-1} \tau', z_0 + s z' + s \xi \xi^{-1} \tau')
  =
  e\bigl( - \tfrac{1}{2} \langle s \xi \rangle_L\, \xi^{-1} \tau' - \langle s \xi, z_0 + s z' \rangle_L \bigr)\,
  \phi(\xi^{-1} \tau', z_0 + s z')
  \\
   & =
  e\bigl( - m' \tau' - 2 m' z' \bigr)\,
  \psi(z')
  \tx{.}
  \end{aligned}
\end{gather}
That is, $\psi$ is an elliptic function on~$\CC \slash \ZZ + \tau' \ZZ$ of index~$m'$.
If~$\langle s \rangle_L < 0$ then~$m' < 0$ and~$\psi$ vanishes by~\cite[Theorem~1.2]{eichler-zagier-1985}, and if~$\langle s \rangle_L = 0$ then~$m' = 0$ and~$\psi$ is constant by the same theorem.

Assume that~$L$ is not positive semi-definite and fix~$s_- \in L_\RR$ with~$\langle s_- \rangle_L < 0$.
Then there is a neighborhood~$S$ of~$s_-$ such that~$\langle s \rangle_L < 0$ for all~$s \in S$, and we can assume that~$S$ is an open $F_\RR$\nbd{}cone in~$L_\RR$.
By~\eqref{eq:prf:prop:jacobi_forms_semi_definite_index:elliptic_function} we have~$\phi(\xi^{-1} \tau', s z') = 0$ for~$s \in L \cap S$.
Since~$S$ is open, rational points are dense in it.
Since further~$S$ is a cone, it is contained in the closure of the subset~$\RR \cdot (S \cap L) \subseteq L_\RR$.
We conclude that~$\phi(\xi^{-1} \tau', z)$ vanishes on the open subset~$\xi^{-1} \HS \times S \subseteq \xi^{-1} \HS \times L_\CC$ and therefore on~$\xi^{-1} \HS \times L_\CC$.
Now we observe that~$\RR_{>0} \cdot \{ \ga^{-1} \condsep \ga \in \cOF \tx{ tot.\@ pos.\@} \} \subseteq F_\RR$ is dense.
In particular, the set of~$\xi^{-1} i y$ for~$\xi$ as before and~$y \in \RR_{>0}$ is dense in the purely imaginary elements of~$\HS_F$.
Thus~$\phi$ vanishes on~$\{ \tau \in \HS_F \condsep \Re(\tau) = 0 \} \times L_\CC$.
Since~$\phi$ is holomorphic, we conclude~$\phi = 0$ as desired, finishing the proof in case that~$L$ is not totally positive semi-definite.

Assume~$L$ is totally positive semi-definite.
Fix~$s \in L_\RR$ with~$\langle s \rangle_L = 0$, which implies that~$\langle s, L_\CC \rangle = 0$.
For~$s \in L$, by~\eqref{eq:prf:prop:jacobi_forms_semi_definite_index:elliptic_function}, the function~$\phi(\xi^{-1} \tau', z + s z')$ is constant in~$z'$ for all~$\xi$ and~$z \in L_\CC$.
In particular, we have~$\phi(\xi^{-1} \tau', z + s z') = \phi(\xi^{-1} \tau', z)$.
Thus by the same density argument as before, we have~$\phi(\tau, z + s z') = \phi(\tau, z)$.

Since~$L$ is positive semi-definite and integral, the maximal isotropic subspace~$(L_\RR)_0 \subseteq L_\RR$ is rational, and we have~$(L_\RR)_0 = L_0 \otimes \RR$ where~$L_0 = (L_\RR)_0 \cap L$.
In other words, every~$s \in L_\RR$ with~$\langle s \rangle_L = 0$ as in the statement can be expressed as a linear combination~$s = c_1 s_1 + \cdots + c_h s_h$ for~$s_i \in L_0$ and~$c_i \in \RR$.
By what we have proved already, we have
\begin{gather*}
  \phi(\tau, z + s z')
  =
  \phi\bigl( \tau, z + (c_1 s_1 + \cdots + c_h s_h) z' \bigr)
  =
  \phi\bigl( \tau, z + (c_1 s_1 + \cdots + c_{h-1} s_{h-1}) z' \bigr)
  =
  \cdots
  =
  \phi(\tau, z)
  \tx{.}
\end{gather*}
This shows the claimed independence of~$z'$.
\end{proof}

In light of \fref{Proposition}{prop:jacobi_forms_semi_definite_index} we can assume in the remainder that~$L$ is totally positive semi-definite.
The next corollary allows us to reduce all considerations to the case of totally positive definite indices.
Recall from p.~225 of~\cite{omeara-2000} that the radical~$L_0$ of~$L$ splits: we have~$L \cong L_0 \oplus L_+$ for a totally positive definite lattice~$L_+$.

\begin{corollary}%
\label{cor:jacobi_forms_semi_definite_index}
Under the identification~$L = L_0 \oplus L_+$ for a totally positive definite lattice~$L_+$, we have
\begin{align*}
  \rmJHilb{k,L}
  =
  \big\{
  (\tau,(z_0,z_+)) \mto \phi(\tau, z_+)
  \condsep
  \phi \in \rmJHilb{k,L_+}
  \big\}
  \tx{.}
\end{align*}
\end{corollary}

\begin{proof}
Since~$\langle \la, (z_0, z_+) \rangle_L = \langle \la, z_+ \rangle_{L_+}$ for all~$\la \in L$, it follows directly from the definition of the Jacobi slash action that~$(\tau,(z_0,z_+)) \mto \phi(\tau, z_+)$ is a Hilbert--Jacobi form of index~$L$.
Conversely, \fref{Proposition}{prop:jacobi_forms_semi_definite_index} shows that a Hilbert--Jacobi form of index~$L$ is independent of~$z_0$.
\end{proof}

\subsection{Dimension of spaces of Hilbert--Jacobi forms}%
\label{ssec:hilbert_jacobi_forms:dimensions}

We next record a convenient dimension bound for spaces of Hilbert--Jacobi forms.
Recall the Weil representation~$\rho_L$ associated in~\cite[Section~2.2]{boylan-2015} with the discriminant form~$\disc(L) = L^\vee \slash L$ of an integral, totally positive definite lattice~$L$ over~$F$.
It is the main example of arithmetic types, which we define as finite dimension complex representations of~$\SL{2}(\cOF)$ whose kernel is a congruence subgroup.
We write~$\rho^\vee$ for the dual of an arithmetic type~$\rho$ and~$V(\rho)$ for its representation space.

\begin{definition}%
\label{def:hilbert_modular_form_arithmetic_type}
Let~$k \in \ZZ$ and~$\rho$ an arithmetic type.
Then a holomorphic function~$f \defcol \HS_F \ra V(\rho)$ is called a \emph{Hilbert modular form} of weight~$k$ and type~$\rho$ if:
\begin{enumerateroman}%
\item
For all~$\ga \in \SL{2}(\cOF)$ we have $f |_k\,\ga = \rho(\ga)\, f$.

\item
If~$F = \QQ$, for all~$\ga \in \SL{2}(\ZZ)$ and all~$n < 0$ the Fourier coefficients in~\eqref{eq:hilbert_modular_fourier_expansion} satisfy~$c(f |_k\,\ga; n) = 0$.
\end{enumerateroman}
\end{definition}

The space of Hilbert modular forms of weight~$k$ and type~$\rho$ is written as~$\rmM_k(\rho)$.
The definition of vanishing orders in~\eqref{eq:def:vanishing_order_hilbert_modular_form} extends verbatim to~$f \in \rmM_k(\rho)$.

\begin{proposition}%
\label{prop:dimension_bound_hilbert_jacobi_forms}
For~$k \in \ZZ_{> 0}$ and all integral, positive semi-definite lattices~$L \cong L_0 \oplus L_+$ over~$F$, we have
\begin{gather*}
  \dim\, \rmJHilb{k,L}
  \ll_F
  k^{n_F}\, \#\disc(L_+)
  \tx{.}
\end{gather*}
\end{proposition}

\begin{proof}
We can assume that~$L$ is totally positive semi-definite by \fref{Proposition}{prop:jacobi_forms_semi_definite_index}.
By~\cite[Theorem~3.5]{boylan-2015} we have an isomorphism between~$\rmJHilb{k,L}$ and~$\rmM_{k-\frac{h}{2}}(\rho_L^\vee)$.
We estimate the dimension of the latter space of Hilbert modular forms.
If~$h$ is odd, then its weight is half-integral and this step requires the metaplectic cover~$\Mp{1}(\cOF)$ of~$\SL{2}(\cOF)$, which is described in~\cite[Section~3.3]{boylan-2015}.
The notions of arithmetic types, Fourier coefficients, and vanishing orders extend to Hilbert modular forms for the metaplectic cover.

We consider an arithmetic type~$\rho$ for~$\Mp{1}(\cOF)$ and fix~$v \in V(\rho)^\vee$.
We can and will assume that the center of~$\Mp{1}(\cOF)$ acts by scalars by decomposing~$\rho$ suitably.
Let~$N$ be such that~$\SL{2}(\cOF,N)$ is contained in the image of the kernel of~$\rho$ under the projection~$\Mp{1}(\cOF) \thra \SL{2}(\cOF)$.
The set
\begin{gather*}
  f_\ga
  \defeq
  v(f) |_k\, \td\ga
  =
  v\bigl( f |_k\, \td\ga \bigr)
  =
  \bigl( \rho^\vee(\td\ga) v \bigr) (f)
\end{gather*}
for~$\ga \in \SL{2}(\cOF,N) \backslash \SL{2}(\cOF)$ is finite of size~$M$, say, where~$\td\ga$ is any fixed lift of~$\ga$ to~$\Mp{1}(\cOF)$.
The functions~$f_\ga$ depend on the choice of lifts only up to nonzero scalar multiples.
From the definition of vanishing orders, we conclude that~$\ord(f_\ga) \ge \ord(f)$.
Further, we observe that~$f_\ga^2$ is a Hilbert modular form of integral weight and level\nbd{}$N$.
\fref{Lemma}{la:accumulated_hecke_bound_twisted_permutation_type} shows that~$v(f)^2 = 0$, if~$\sum_\ga \ord(f_\ga) \gg_F k\, M$.

The left hand side of the previous inequality is at least~$M \ord(f)$.
That is, we have~$v(f)^2 = 0 = v(f)$ if~$\ord(f) \gg_F k$.
Since~$v$ was arbitrary, this shows that there is a constant~$\nu_F$ such that the map
\begin{gather*}
  \rmM_k(\rho)
  \lra
  V(\rho)^B
  \tx{,}\quad
  f
  \lmto
  \bigl( c(f; n) \bigr)_{n \in B}
  \quad\tx{with }
  B
  \defeq
  \bigl\{ n \in \cOF^\vee \condsep n \tx{ tot.\@ pos.\@}, n < k \nu_F \bigr\}
\end{gather*}
is injective.
We have~$\# B \ll_F k^{n_F}$ and thus
\begin{gather*}
  \dim\, \rmM_k(\rho)
  \ll_F
  k^{n_F}\,
  \dim\, \rho
  \tx{.}
\end{gather*}

Applying this to the special case of~$\rho = \rho_L^\vee$, we obtain the proposition, since~$\dim\, \rho_L = \#\disc(L)$ by the definition of~$\rho_L$.
\end{proof}

\subsection{Torsion points of prime denominator}%
\label{ssec:hilbert_jacobi_forms:torsion_point_prime_denominator}

In this section, we start working towards the vanishing criterion in \fref{Theorem}{thm:hilbert_jacobi_forms_hecke_bound}.
We first show that vanishing of Hilbert--Jacobi forms can be detected at sufficiently many torsion points, whose denominators we restrict to be rational primes.

The closed Voronoi cell of a totally positive definite lattice~$L$ over~$F$ is
\begin{gather}
  \label{eq:def:voronoi_cell}
  \rmVclosed_L
  =
  \rmVclosed_{\traceF(L)}
  \defeq
  \big\{
  s \in L_\RR
  \condsep
  \traceF(\langle s \rangle_L)
  \le
  \min_{\la \in L}\,
  \traceF( \langle s + \la \rangle_L )
  \big\}
  \tx{.}
\end{gather}
It is a polytopal complex in a natural way.
We fix once and for all a subcomplex that is an exact fundamental domain for the action of~$L$ on~$L_\RR$, and call it the Voronoi cell~$\rmV_L$ of~$L$.
Its interior is the open Voronoi cell~$\rmVopen_L$.

We assume that~$L$ is free, and use the coordinate representation~$\cOF^h$ for the module underlying~$L$.
For \emph{distinct rational} primes~$p_i$, $1 \le i \le h$, we set
\begin{multline}
  \label{eq:def:torsion_points_distinct_primes}
  \rmTP_L(p_1,\ldots,p_h)
  \\
  \defeq
  \bigl\{
  (\alpha, \beta)
  \in
  (L_\QQ \cap \rmV_L)^2
  \condsep
  \alpha_i, \beta_i \in \tfrac{1}{p_i} \cOF,\,
  \alpha_i \cOF + \beta_i \cOF + \cOF = \tfrac{1}{p_i} \cOF
  \tx{ for all } 1 \le i \le h
  \bigr\}
  \tx{.}
\end{multline}

\begin{lemma}%
\label{la:torsion_points_vanishing}
Fix~$M \in \RR_{> 0}$ and a Hilbert--Jacobi form~$\phi \in \rmJHilb{k,L}$ of free, totally positive definite index~$L$.
Then~$\phi = 0$, if
\begin{gather*}
  \phi[\alpha,\beta] = 0
  \quad\tx{for all }
  (\alpha,\beta) \in \rmTP_L(p_1, \ldots, p_h)
  \tx{ and all distinct primes }
  p_i > M
  \tx{.}
\end{gather*}
\end{lemma}

\begin{proof}
This follows from holomorphicity of~$\phi$ and the observation that
\begin{gather*}
  \bigcup_{\substack{
      (\alpha,\beta) \in \rmTP_L(p_1, \ldots, p_h)\tx{,} \\
      p_i > M \tx{ distinct primes}
    }}
  \mspace{-24mu}
  \big\{
  (\tau, \alpha \tau + \beta)
  \condsep
  \tau \in \HS_F
  \big\}
  \subset
  \HS_F \times \bigl( \rmVclosed_L \tau + \rmVclosed_L \bigr)
\end{gather*}
is dense, and the right hand side contains an open subset of~$\HS_F \times L_\CC$.
\end{proof}

The condition on~$\alpha_i \cOF + \beta_i \cOF$ in the definition of~$\rmTP_L(p_1, \ldots, p_h)$ is difficult to work with.
The next statement allows us to work with a subset of torsion points for which~$\alpha$ is unrestricted.

\begin{lemma}%
\label{la:torsion_points_coprime_beta}
Given a free, totally positive definite lattice~$L$ and distinct rational primes~$p_i$ that are unramified in~$F$, there is a set~$B \subset L_\QQ \cap \rmV_L$ such that
\begin{gather*}
  \bigl\{
  \alpha
  \in
  L_\QQ \cap \rmV_L
  \condsep
  \alpha_i \in \tfrac{1}{p_i} \cOF
  \bigr\}
  \times
  B
  \subset
  \rmTP_L(p_1,\ldots,p_h)
  \quad\tx{and}\quad
  \# B
  \ge
  \prod_{i = 1}^h
  (p_i - 1)^{n_F}
  \tx{.}
\end{gather*}
\end{lemma}

\begin{proof}
The condition on~$\alpha_i \cOF + \beta_i \cOF$ translates to co-prime~$p_i \alpha_i$ and~$p_i \beta_i$ in~$\cOF \slash p_i \cOF$.
Since~$p_i$ is unramified, the quotient~$\cOF \slash p_i \cOF$ is a direct sum of fields.
The product of their unit groups, which yields~$(\cOF \slash p_i \cOF)^\times$, is of size at least~$(p_i - 1)^{n_F}$.
Hence we can set
\begin{gather*}
  B
  \defeq
  \bigl\{
  \beta
  \in
  L_\QQ \cap \rmV_L
  \condsep
  \beta_i \in \tfrac{1}{p_i} \cOF
  \tx{ such that\/ }
  p_i \beta_i \in ( \cOF \slash p_i \cOF )^\times
  \bigr\}
  \tx{.}
\end{gather*}
\end{proof}

We next describe the action of~$\SL{2}(\cOF)$ on specializations to torsion points.

\begin{lemma}%
\label{la:torsion_points_sl2_action_transitive}
Given~$(\alpha,\beta) \in \rmTP_L(p_1,\ldots,p_h)$ for a free lattice~$L$ over~$F$ and~$\ga \in \SL{2}(\cOF)$, there are unique~$\la, \mu \in L$ such that
\begin{align*}
  (\alpha,\beta) \ga - (\la,\mu)
  \in
  \rmTP_L(p_1,\ldots,p_h)
  \tx{.}
\end{align*}
With these~$\la, \mu$, we have a right action
\begin{gather*}
  \rmTP_L(p_1,\ldots,p_h) \circlearrowleft \SL{2}(\cOF)
  \tx{,}\quad
  \bigl( (\alpha,\beta), \ga \bigr)
  \lmto
  (\alpha,\beta) \ga - (\la,\mu)
  \tx{,}
\end{gather*}
which is transitive if all~$p_i$ are unramified in~$F$.
\end{lemma}

\begin{proof}
All conditions on~$(\alpha,\beta) \in L_\QQ^2$ in the definition~\eqref{eq:def:torsion_points_distinct_primes} of~$\rmTP_L(p_1, \ldots, p_h)$ are preserved under the action of~$\SL{2}(\cOF)$ on~$L_\QQ^2$.
Since~$\rmV_L$ is defined as a strict fundamental domain, existence and uniqueness of~$\la, \mu$ as in the statement follow.

The action in the statement arises from the natural action of~$\SL{2}(\cOF)$ on~$L^2 \backslash L_\QQ^2$.
More specifically, for fixed~$1 \le i \le h$ let~$p_i = \frakp_{i\!1} \cdots \frakp_{i\!m}$ be the prime ideal factorization of~$p_i$.
Then the action on the~$i$\thdash{} component is given by the direct sum of actions of~$\SL{2}(\cOF \slash \frakp_{i\!j})$ acting on nonzero elements of the vector space~$( \cOF \slash \frakp_{i\!j} )^2$.
The direct sum of all these actions over~$i$ and~$j$ is transitive, since the primes~$p_i$ are mutually distinct, finishing the proof.
\end{proof}

The next two lemmas help to employ the previous one in the context of Hilbert--Jacobi forms.

\begin{lemma}%
\label{la:jacobi_forms_torsion_points_sl2_shift_action}
Given~$\phi \in \rmJHilb{k,L}$ and~$\alpha,\beta \in L_\QQ$, for~$\ga \in \SL{2}(\cOF)$ set~$(\alpha', \beta') = (\alpha, \beta)\, \ga$.
Then we have
\begin{align*}
  \phi[\alpha, \beta] \big|_k\, \ga
   & =
  e\bigl(
  \tfrac{1}{2} \langle \alpha, \beta \rangle_L
  -
  \tfrac{1}{2} \langle \alpha', \beta' \rangle_L
  \bigr)\,
  \phi[\alpha', \beta']
  \intertext{and for~$\la, \mu \in L$}
  \phi[\alpha + \la, \beta + \mu]
   & =
  e\bigl(
  \tfrac{1}{2} \langle \la, \mu \rangle_L
  +
  \langle \alpha, \mu \rangle_L
  \bigr)\,
  \phi[\alpha, \beta]
  \tx{.}
\end{align*}
\end{lemma}

\begin{proof}
Recall that the action of~$\SL{2}(\cOF)$ on~$\Hb{L}(\RR)$, given in~\eqref{eq:def:heisenberg_group_sl2_action}, is based on vector-matrix multiplication of~$L \times L$ and~$\SL{2}(\cOF)$.
The definition of specializations to torsion points in~\eqref{eq:def:jacobi_forms_torsion_points} yields
\begin{gather*}
  \phi[\alpha, \beta] \big|_k\, \ga
  =
  \Bigl( \phi \big|_{k,m}\, \transJ(\alpha, \beta)\, \ga \Bigr)(\tau, 0)
  =
  \Big(
  \phi \big|_{k,m}\, \ga\,
  \transJ(\alpha', \beta')\,
  \bigl(0,0,
  \tfrac{1}{2} \langle \alpha, \beta \rangle_L
  -
  \tfrac{1}{2} \langle \alpha', \beta' \rangle_L
  \bigr)
  \Big)(\tau, 0)
  \tx{.}
\end{gather*}
Modular invariance of~$\phi$ allows us to discard~$\ga$ on the right hand side.
The exponential factor in the first equality of the lemma arises from the action of the central element acting on the right hand side.

Similarly, the second equality follows from the factorization
\begin{gather*}
  \transJ(\alpha + \la, \beta + \mu)
  =
  (\la,\mu,0)\,
  \transJ(\alpha, \beta)\,
  \bigl(0,0,
  \tfrac{1}{2} \langle \la, \mu \rangle_L
  +
  \langle \alpha, \mu \rangle_L
  \bigr)
  \tx{.}
\end{gather*}
\end{proof}

\begin{lemma}%
\label{la:torsion_points_modularity}
Given a Hilbert--Jacobi form~$\phi \in \rmJHilb{k,L}$ of free index~$L$ and~$(\alpha,\beta) \in \rmTP_L(p_1,\ldots,p_h)$ for distinct rational primes~$p_1, \ldots, p_h$, we have
\begin{gather*}
  \phi[\alpha,\beta]
  \in
  \rmM_k\bigl( \SL{2}\bigl( \cOF, N_F \prod p_i^2 \bigr) \bigr)
  \tx{,}
\end{gather*}
where~$N_F^{-1} \ZZ = \traceF( \frac{1}{2} \cOF^\vee )$.
\end{lemma}

\begin{proof}
We fix~$\ga \in \SL{2}( \cOF, N_F \prod p_i^2 )$.
As in \fref{Lemma}{la:jacobi_forms_torsion_points_sl2_shift_action}, set~$(\alpha',\beta') = (\alpha, \beta) \ga$.
Further, we let~$(\la, \mu) = (\alpha',\beta') - (\alpha,\beta)$.
From the definition of~$\rmTP_L(p_1,\ldots,p_h)$ in~\eqref{eq:def:torsion_points_distinct_primes}, which uses that~$L$ is free, we infer that~$\la_i, \mu_i \in N_F p_i \cOF$ holds for each coordinate.
Then combining both transformation formulas in \fref{Lemma}{la:jacobi_forms_torsion_points_sl2_shift_action} leaves us with
\begin{gather*}
  \phi[\alpha,\beta]
  \big|_k\, \ga
  =
  e\bigl(
  \tfrac{1}{2} \langle \alpha, \beta \rangle_L
  -
  \tfrac{1}{2} \langle \alpha', \beta' \rangle_L
  \bigr)\,
  e\bigl(
  \tfrac{1}{2} \langle \la, \mu \rangle_L
  +
  \langle \alpha, \mu \rangle_L
  \bigr)\,
  \phi[\alpha,\beta]
  \tx{.}
\end{gather*}
Substituting~$(\alpha',\beta') = (\alpha,\beta) + (\la, \mu)$, the exponential factors yield~$e( \frac{1}{2} \langle \alpha, \mu \rangle - \frac{1}{2} \langle \beta, \la \rangle )$.
By our assumptions we have~$\alpha_i \mu_i, \beta_i \la_i \in N_F \cOF$.
Thus~$\frac{1}{2} \langle \alpha, \mu \rangle \in \frac{N_F}{2} \cOF^\vee$, and similarly for the pairing of~$\beta$ and~$\la$.
By definition of~$N_F$, the exponential factor in the transformation formula for~$\phi[\alpha,\beta]$ thus equals~$1$.

If~$F = \QQ$, we need to verify the growth condition for~$\phi[\alpha,\beta] |_k\, \ga$ for any~$\ga \in \SL{2}(\ZZ)$.
By \fref{Lemma}{la:jacobi_forms_torsion_points_sl2_shift_action}, it suffices to consider~$\ga = 1$ after replacing~$\alpha$ and~$\beta$ suitably.
The Fourier coefficient of~$\phi$ of index~$(n,r)$ contributes to the coefficient of~$\phi[\alpha,\beta]$ of index
\begin{gather*}
  n + \lT{r} \alpha + m[\alpha]
  =
  \mfrac{1}{2}
  \begin{psmatrix} 2 n & \lT{r} \\ r & 2 m \end{psmatrix}
  \Bigl[\begin{psmatrix} 1 \\ \alpha \end{psmatrix}\Bigr]
  \ge
  0
  \tx{,}
\end{gather*}
by the growth condition for~$\phi$ in \fref{Definition}{def:jacobi_forms_hilbert}.
\end{proof}

\subsection{Vanishing order at torsion points}%
\label{ssec:hilbert_jacobi_forms:vanishing_order_torsion_points}

To apply the vanishing statement in \fref{Lemma}{la:accumulated_hecke_bound_twisted_permutation_type} we also need to bound the vanishing order of specializations to torsion points from below.

Recall the closed Voronoi cell from~\eqref{eq:def:voronoi_cell}.
By the symmetry of the lattice, we have the point symmetry~$\rmVclosed_L = - \rmVclosed_L$.
The closed Voronoi cell is a polytope and in particular has finitely many extreme points, i.e.\@ points that are not a proper convex combination of two distinct points in~$\rmVclosed_L$, for which we introduce the notation
\begin{gather}
  \label{eq:def:voronoi_cell_extreme_points}
  \rmEV_L
  \defeq
  \big\{
  s \in \rmVclosed_L
  \condsep
  s \tx{ an extreme point of\/ } \rmVclosed_L
  \big\}
  \tx{.}
\end{gather}

\begin{proposition}%
\label{prop:jacobi_forms_torsion_points_vanishing_order}
Given a Hilbert--Jacobi form~$\phi \in \rmJHilb{k,L}$ of totally positive definite index~$L$, the vanishing order of the specialization~$\phi[\alpha,\beta]$ for~$\alpha, \beta \in L_\QQ$ admits the following lower bound, expressed in terms of the extreme points~\eqref{eq:def:voronoi_cell_extreme_points} of the closed Voronoi cell:
\begin{gather*}
  \ord\big( \phi[\alpha, \beta] \big)
  \gg_F
  \ord(\phi)
  -
  \max \big\{
  \traceF\bigl(
  \tfrac{1}{2} \langle r_0 \rangle_L - \tfrac{1}{2} \langle r_0 + \alpha + \la \rangle_L
  \bigr)
  \condsep
  r_0 \in \rmEV_L,
  \la \in L
  \bigr\}
  \tx{.}
\end{gather*}
\end{proposition}

\begin{proof}
For convenience set~$\nu = \ord(\phi)$.
We start by examining the support of the Fourier expansion of~$\phi$.
Modular invariance of~$\phi$ implies that
\begin{gather}
  \label{eq:jacobi_lambda_invariance_fourier_coefficients}
  c(\phi; n, r)
  =
  c(\phi; n', r')
  \tx{,}\quad
  \tx{where }
  r'
  \in
  r + L
  \tx{ and }
  n - \tfrac{1}{2} \langle r \rangle
  =
  n' - \tfrac{1}{2} \langle r' \rangle
  \tx{.}
\end{gather}
We can apply the specialization to torsion points defined in~\eqref{eq:def:jacobi_forms_torsion_points} to each term in the Fourier expansion of~$\phi$ separately.
The contribution of the term of index~$(n,r)$ equals, up to non-zero scalar multiples,~$e( \frac{1}{2} \langle \alpha \rangle \tau )\, e( n \tau + \langle r, \alpha \rangle \tau )$.
As a consequence, we have the lower bound
\begin{gather}
  \label{eq:prop:jacobi_forms_torsion_points_vanishing_order:first_estimate}
  \ordtrace\bigl( \phi[\alpha, \beta] \bigr)
  \ge
  \inf \big\{ \traceF\bigl(
  n + \tfrac{1}{2} \langle \alpha \rangle + \langle r, \alpha \rangle
  \bigr)
  \condsep
  n \in \cOF^\vee, r \in L^\vee,\;
  c( \phi; n, r ) \ne 0
  \bigr\}
  \tx{.}
\end{gather}

Every index~$r \in L^\vee \subset L_\RR$ can be decomposed as~$r = r_0 + \la$ for~$r_0 \in \rmV_L$ in the Voronoi cell and~$\la \in L$.
We define an auxiliary variable~$n_0$ by~$n_0 - \frac{1}{2} \langle r_0 \rangle = n - \frac{1}{2} \langle r \rangle$, and insert the resulting expression for~$n$ into~\eqref{eq:prop:jacobi_forms_torsion_points_vanishing_order:first_estimate}.
Further, the condition~$c(\phi; n, r) \ne 0$ that appears in~\eqref{eq:prop:jacobi_forms_torsion_points_vanishing_order:first_estimate} is equivalent to~$c(\phi; n_0, r_0) \ne 0$ by~\eqref{eq:jacobi_lambda_invariance_fourier_coefficients}.
By \fref{Lemma}{la:vanishing_order_equivalence_jacobi_form}, we have~$\ordtrace(\phi) \asymp_F \ord(\phi) = \nu$, and can therefore replace it by the weaker condition~$\traceF(n_0) \gg_F \nu$.
We obtain the estimate
\begin{alignat*}{2}
   &         &  &
  \ordtrace\bigl( \phi[\alpha, \beta] \bigr)
  \\
   & \ge{}   &  &
  \inf \big\{
  \traceF\bigl(
  n_0
  -
  \tfrac{1}{2} \langle r_0 \rangle
  +
  \tfrac{1}{2} \langle r \rangle
  +
  \tfrac{1}{2} \langle \alpha \rangle
  +
  \langle r, \alpha \rangle
  \bigr)
  \condsep
  n_0 \in F_\RR,
  \traceF(n_0) \gg_F \nu,\;
  r_0 \in \rmV_L,
  r \in r_0 + L
  \bigr\}
  \\
   & \gg_F{} &  &
  \nu
  -
  \sup \big\{
  \traceF\bigl(
  \tfrac{1}{2} \langle r_0 \rangle
  -
  \tfrac{1}{2} \langle r \rangle
  -
  \tfrac{1}{2} \langle \alpha \rangle
  -
  \langle r, \alpha \rangle
  \bigr)
  \condsep
  r_0 \in \rmV_L,
  r \in r_0 + L
  \bigr\}
  \\
   & ={}     &  &
  \nu
  -
  \sup \big\{
  \traceF\bigl(
  \tfrac{1}{2} \langle r_0 \rangle
  -
  \tfrac{1}{2} \langle r_0 + \alpha + \la \rangle
  \bigr)
  \condsep
  r_0 \in \rmV_L,
  \la \in L
  \bigr\}
  \tx{.}
\end{alignat*}

For fixed~$\la$ the function~$\traceF( \frac{1}{2} \langle r_0 \rangle - \frac{1}{2} \langle r_0 + \alpha + \la \rangle)$ is linear in~$r_0$, and thus convex.
By Bauer's Maximum Principle it attains its maximum on the compact, convex set~$\rmVclosed_L$ at one of the extreme points in~$\rmEV_L$.
The supremum over~$\la$ is a maximum, since~$L$ is positive definite.
By \fref{Lemma}{la:vanishing_order_equivalence_hilbert_modular_form} we have~$\ord(\phi[\alpha, \beta]) \asymp_F \ordtrace(\phi[\alpha, \beta])$, finishing the proof.
\end{proof}

When employing \fref{Proposition}{prop:jacobi_forms_torsion_points_vanishing_order}, we will reduce ourselves to the case of lattices over~$\ZZ$, for which the next lemma provides a simple but crucial estimate.

\begin{lemma}%
\label{la:jacobi_forms_torsion_points_vanishing_order_alpha_linear_estimate}
Consider an integral, positive definite lattice~$L$ over~$\QQ$.
Let~$\| \,\cdot\, \|_1$ be the~$1$\nbd{}norm associated with a basis of~$L$ for which~$m_L$ is Minkowski reduced.
Given~$\alpha \in \RR^h$ with~$\|\alpha\|_1 \ll_h 1$, we have
\begin{gather*}
  \max \big\{
  \tfrac{1}{2} \langle r_0 \rangle_L - \tfrac{1}{2} \langle r_0 + \la + \alpha \rangle_L
  \condsep
  r_0 \in \rmEV_L,
  \la \in L
  \bigr\}
  \ll_h
  \la_h(L)\,
  \| \alpha \|_1
  \tx{.}
\end{gather*}
\end{lemma}

\begin{proof}
We may work with coordinates such that~$m = m_L$ is Minkowski reduced.
By~\cite[XI.3~(2)]{cassels-1997}, the extreme points~$r_0$ of~$\rmEV_L$ satisfy~$\| r_0 \|_1 \ll_h 1$.
Writing~$m' \in \MatT{h}(\ZZ)$ for the diagonal matrix with entries~$m'_{i\!i} = m_{i\!i}$, since~$m$ is Minkowski reduced, we have
\begin{gather*}
  \tfrac{1}{2} \langle r_0 \rangle_L - \tfrac{1}{2} \langle r_0 + \la + \alpha \rangle_L
  =
  m[r_0] - m[r_0 + \alpha + \la]
  \asymp_h
  m'[r_0] - m'[r_0 + \alpha + \la]
  \tx{.}
\end{gather*}
Since~$m_{i\!i} \le \la_h(L)$ and since there is only finitely many~$r_0 \in \rmEV_L$, which all have bounded components~$|r_{0,i}| \ll_h 1$, it suffices to bound
\begin{gather*}
  \max \big\{
  s^2 - (s + \alpha + \la)^2
  \condsep
  \la \in \ZZ
  \bigr\}
\end{gather*}
for arbitrary~$s, \alpha \in \RR$, $s, \alpha \ll_h 1$ to prove the lemma.

The maximum of $s^2 - (s + \alpha + \la)^2$ with respect to~$\la \in \RR$ is attained at~$\la = - s - \alpha$.
Since it is a concave expression with respect to~$\la$, its maximum with respect to~$\la \in \ZZ$ is therefore attained at some~$- s - \alpha - 1 < \la < - s - \alpha + 1$.
Since~$s, \alpha \ll_h 1$ we conclude~$\la \ll_h 1$ for such~$\la$.
In particular, we have
\begin{gather*}
  s^2 - (s + \alpha + \la)^2
  =
  - \alpha^2 - (s + \la)^2 - 2 (s + \la) \alpha
  \le
  - 2 (s + \la) \alpha
  \ll_h
  |\alpha|
  \tx{.}
\end{gather*}
\end{proof}

We finally estimate the sum through which \fref{Proposition}{prop:jacobi_forms_torsion_points_vanishing_order} appears in the next section.

\begin{lemma}%
\label{la:jacobi_forms_torsion_points_vanishing_order_sum}
Consider an integral, free, totally positive definite lattice~$L$ over~$F$ and~$\nu \in \RR$ satisfying~$\la_h(L) \ll_{F,h} \nu$.
Then for integers~$n_i \gg_{F,h} 1$ we have
\begin{gather*}
  \sum_{\substack{
      \alpha \in \rmV_L \\
      \alpha_i \in \frac{1}{n_i} \cOF \tx{ for all } i
    }}
  \mspace{-36mu}
  \max\Big\{ 0,\,
  \nu
  -
  \max \big\{
  \traceF\bigl(
  \tfrac{1}{2} \langle r_0 \rangle_L - \tfrac{1}{2} \langle r_0 + \alpha + \la \rangle_L
  \bigr)
  \condsep
  r_0 \in \rmEV_L,
  \la \in L
  \big\}
  \Bigr\}
  \gg_{F,h}
  \nu
  \prod_{i = 1}^h n_i^{n_F}
  \tx{.}
\end{gather*}
\end{lemma}

\begin{proof}
First, we note that the sum does not depend on the~$\cOF$ structure of~$L$.
Further, recall that the successive minima of~$L$ and its trace lattice satisfy~$\la_h(L) \asymp_F \la_{nh}(\traceF(L))$.
We can thus replace~$L$ by~$\traceF(L)$ throughout the proof, that is, we work with~$F = \QQ$.
In particular, $L$ is free and we fix an isomorphism of the underlying~$\ZZ$\nbd{}module with~$\ZZ^h$ in such a way that~$m_L$ is Minkowski reduced.

For brevity, we write~$f(\alpha)$ for the summand that appears in the lemma.
Since~$f(\alpha) \ge 0$, we can estimate the sum by
\begin{gather*}
  \sum_{\substack{
      \alpha \in \rmV_L \\
      \alpha_i \in \frac{1}{n_i} \ZZ \tx{ for all } i
    }}
  \mspace{-30mu}
  f(\alpha)
  \ge
  \sum_{\substack{
      \alpha \in \rmV_L \\
      \| \alpha \|_1 \ll_h 1 \\
      \alpha_i \in \frac{1}{n_i} \ZZ \tx{ for all } i
    }}
  \mspace{-30mu}
  f(\alpha)
  \tx{,}
\end{gather*}
where the implicit condition~$\| \alpha \|_1 \ll_h 1$ is chosen in such a way that \fref{Lemma}{la:jacobi_forms_torsion_points_vanishing_order_alpha_linear_estimate} applies.
This lemma yields~$f(\alpha) \gg_h \max\{ 0, \nu - \la_h(L) \|\alpha\|_1 \}$.
Since~$\la_h(L) \ll_h \nu$, for~$\| \alpha \|_1 \ll_h 1$ (with possibly a smaller implicit constant) we have~$f(\alpha) \gg_h \nu\, (1 - \|\alpha\|_1)$.

There is~$h^{-1} > \eps_h > 0$, which depends only on~$h$, such that~$\alpha \in [-\eps_h,\eps_h]^h$ satisfies the previous implicit bound~$\| \alpha \|_1 \ll_h 1$.
In particular, for such~$\alpha$ we have~$f(\alpha) \gg_h \nu\, ( 1 - h \eps_h) \gg_h \nu$.
Writing~$m'_L$ for the diagonal of~$m_L$, we have~$m_L \asymp_h m'_L$ by Minkowski reduction.
Let~$L'$ be the corresponding lattice on~$\ZZ^h$.
We have
\begin{gather*}
  \big| \langle \la, s \rangle_L \big|
  \ll_h
  \big| \langle \la, s \rangle_{L'} \big|
  \le
  \tfrac{1}{2} \langle \la \rangle_{L'}
  \ll_h
  \tfrac{1}{2} \langle \la \rangle_{L}
  \quad
  \tx{for all }
  s \in \rmV_{\rmL'} \tx{ and } \la \in \ZZ^h
  \tx{.}
\end{gather*}
This implies that~$2 \eps_h \rmVopen_{L'} \subset \rmVopen_L$ after shrinking~$\eps_h$ if needed.
Since the Voronoi cell of~$L'$ equals~$[-\frac{1}{2}, \frac{1}{2}]^h$, we conclude that~$[-\eps_h,\eps_h]^h \subset \rmVopen_L$.
Note that here~$\eps_h$ still only depends on~$h$.
Thus if~$n_i \gg_h 1$ we have the estimate
\begin{gather*}
  \#\bigl\{
  \alpha \in \rmV_L
  \condsep
  -\eps_h \le \alpha_i \le \eps_h,
  \alpha_i \in \tfrac{1}{n_i} \ZZ \tx{ for all } i
  \bigr\}
  \gg_h
  \prod_{i = 1}^h
  n_i
  \tx{.}
\end{gather*}
When combining this with our estimate for~$f(\alpha)$, we obtain the desired lower bound
\begin{gather*}
  \sum_{\substack{
      \alpha \in \rmV_L \\
      \alpha_i \in \frac{1}{n_i} \ZZ \tx{ for all } i
    }}
  \mspace{-30mu}
  f(\alpha)
  \ge
  \sum_{\substack{
      \alpha \in \rmV_L \\
      -\eps_h \le \alpha_i \le \eps_h \\
      \alpha_i \in \frac{1}{n_i} \ZZ \tx{ for all } i
    }}
  \mspace{-30mu}
  f(\alpha)
  \gg_h
  \sum_{\substack{
      \alpha \in \rmV_L\\
      -\eps_h \le \alpha_i \le \eps_h \\
      \alpha_i \in \frac{1}{n_i} \ZZ \tx{ for all } i
    }}
  \mspace{-30mu}
  \nu
  \gg_h
  \nu\,
  \prod_{i = 1}^h
  n_i
  \tx{.}
\end{gather*}
\end{proof}

\subsection{Hecke bound}%
\label{ssec:hilbert_jacobi_forms:hecke_bound}

We are now in position to prove the desired vanishing bound for Jacobi forms, which is a Hecke bound for Hilbert--Jacobi forms.

\begin{theorem}%
\label{thm:hilbert_jacobi_forms_hecke_bound}
There is\/~$\rho_{F,h} > 0$ with the following property:
Consider a Hilbert--Jacobi form~$\phi \in \rmJHilb{k,L}$ of totally positive definite index~$L$.
If~$k < \rho_{F,h} \ord(\phi)$ and~$\la_{h}(L) < \rho_{F,h} \ord(\phi)$, then~$\phi = 0$.
\end{theorem}

\begin{proof}%
We set~$\nu = \ordtrace(\phi)$ and observe that~$\nu \asymp_F \ord(\phi)$ by \fref{Lemma}{la:vanishing_order_equivalence_jacobi_form}.
Given vectors~$v_1, \ldots, v_h$ in~$L$ that realize its successive minima, we can replace the~$\cOF$\nbd{}module underlying~$L$ by the span of~$v_1, \ldots, v_h$, which is free.
In particular, we can employ the notion of torsion points in~\eqref{eq:def:torsion_points_distinct_primes}.
To show that~$\phi = 0$, we exhibit tuples~$(p_1,\ldots,p_h)$ of mutually distinct rational primes so that~$\phi$ vanishes at the torsion points of denominator~$(p_1,\ldots,p_h)$ in~\eqref{eq:def:torsion_points_distinct_primes}.

Given any~$(p_1,\ldots,p_h)$ as before, our goal is to apply \fref{Lemma}{la:accumulated_hecke_bound_twisted_permutation_type} to the functions~$f_{(\alpha,\beta)} \defeq \phi[\alpha,\beta]$ for~$(\alpha,\beta) \in \rmTP_L(p_1,\ldots,p_h)$.
By \fref{Lemma}{la:torsion_points_modularity}, $f_{(\alpha,\beta)}$ is a modular form of level~$N_F p_1^2 \cdots p_h^2$, where~$N_F$ is a positive integer defined in that lemma.

Write~$(\alpha,\beta) \bullet \ga$ for the transitive action of~$\SL{2}(\cOF)$ on~$\rmTP_L(p_1,\ldots,p_h)$ in \fref{Lemma}{la:torsion_points_sl2_action_transitive}.
We have the relation~$f_{(\alpha,\beta)} |_k\, \ga = c_{(\alpha,\beta), \ga}\, f_{(\alpha,\beta) \bullet \ga}$ for a suitable root of unity~$c_{(\alpha,\beta), \ga}$ by \fref{Lemma}{la:jacobi_forms_torsion_points_sl2_shift_action}.
This verifies all but the last assumption in \fref{Lemma}{la:accumulated_hecke_bound_twisted_permutation_type}.

To verify the last assumption in \fref{Lemma}{la:accumulated_hecke_bound_twisted_permutation_type}, we estimate the vanishing orders of~$f_{(\alpha,\beta)}$ as follows.
\fref{Proposition}{prop:jacobi_forms_torsion_points_vanishing_order} and \fref{Lemma}{la:torsion_points_modularity} show that
\begin{gather*}
  \ordtrace\big( \phi[\alpha,\beta] \big)
  \ge
  \max\Big\{ 0,\,
  \nu
  -
  \max \big\{
  \traceF\bigl(
  \tfrac{1}{2} \langle r_0 \rangle_L - \tfrac{1}{2} \langle r_0 + \alpha + \la \rangle_L
  \bigr)
  \condsep
  r_0 \in \rmEV_L,
  \la \in L
  \Bigr\}
  \tx{.}
\end{gather*}
We fix a set~$B$ as in \fref{Lemma}{la:torsion_points_coprime_beta}, and we write~$\rmTP'_L(p_1,\ldots,p_h)$ for the set~$\{ \alpha \in L_\QQ \cap \rmV_L \condsep \alpha_i \in \frac{1}{p_i} \cOF \} \times B$, which appears there.
Since~$\ordtrace(\phi[\alpha,\beta]) \ge 0$ we can use \fref{Lemma}{la:torsion_points_coprime_beta} to estimate the next sum over torsion points by the sum over this subset.
Since our estimate for~$\ordtrace(\phi[\alpha,\beta])$ is independent of~$\beta$, this allows us to remove~$\beta$ from the summation range.
\begin{align*}
        &
  \sum_{(\alpha,\beta) \in \rmTP_L(p_1,\ldots,p_h)}
  \mspace{-30mu}
  \ordtrace\big( f_{(\alpha,\beta)} \big)
  =
  \sum_{(\alpha,\beta) \in \rmTP_L(p_1,\ldots,p_h)}
  \mspace{-30mu}
  \ordtrace\big( \phi[\alpha,\beta] \big)
  \ge
  \sum_{(\alpha,\beta) \in \rmTP'_L(p_1,\ldots,p_h)}
  \mspace{-30mu}
  \ordtrace\big( \phi[\alpha,\beta] \big)
  \\
  \ge{} &
  \Bigl( \prod_{i = 1}^h (p_i-1)^{n_F} \Bigr)
  \mspace{-18mu}
  \sum_{\substack{
  \alpha \in \rmV_L \\
      \alpha_i \in \frac{1}{p_i} \cOF \tx{ for all } i
    }}
  \mspace{-30mu}
  \max\Big\{ 0,\,
  \nu
  -
  \max \big\{
  \traceF\bigl(
  \tfrac{1}{2} \langle r_0 \rangle_L - \tfrac{1}{2} \langle r_0 + \alpha + \la \rangle_L
  \bigr)
  \condsep
  r_0 \in \rmEV_L,
  \la \in L
  \Bigr\}
  \tx{.}
\end{align*}

The sum on the right hand side is estimated in \fref{Lemma}{la:jacobi_forms_torsion_points_vanishing_order_sum}, which together with \fref{Lemma}{la:vanishing_order_equivalence_jacobi_form} yields for~$p_i \gg_{F,h} 1$ that
\begin{gather*}
  \sum_{(\alpha,\beta) \in \rmT_L(p_1,\ldots,p_h)}
  \mspace{-32mu}
  \ord\big( f_{(\alpha,\beta)} \big)
  \asymp_{F,h}
  \sum_{(\alpha,\beta) \in \rmT_L(p_1,\ldots,p_h)}
  \mspace{-32mu}
  \ordtrace\big( f_{(\alpha,\beta)} \big)
  \gg_{F,h}
  \Bigl( \prod_{i = 1}^h p_i (p_i-1) \Bigr)^{n_F}\,
  \nu
  \gg_{F,h}
  \nu\, \# \rmT_L(p_1,\ldots,p_h)
  \tx{.}
\end{gather*}
This shows that the last assumption in \fref{Lemma}{la:accumulated_hecke_bound_twisted_permutation_type} follows from~$k \ll_{F,h} \nu$, which holds by the assumption of the theorem.
In summary, for~$p_i \gg_{F,h} 1$ we have~$\phi[\alpha,\beta] = f_{(\alpha,\beta)} = 0$.
We can now apply \fref{Lemma}{la:torsion_points_vanishing} to finish the proof.
\end{proof}

%% file: sections/02_hermitian_modular_forms.tex
\section{Hermitian Hilbert modular and Jacobi forms}%
\label{sec:hermitian_modular_jacobi_forms}

For simplicity, we refer to Hermitian Hilbert modular forms as Hermitian modular forms.
Classical references for Hermitian modular forms are the papers by Braun~\cite{braun-1949,braun-1950,braun-1951}.
Shimura rediscovered and generalized her theory~\cite{shimura-1964,shimura-1978}.
Haverkamp's thesis~\cite{haverkamp-1995,haverkamp-1996} contains a theory of Hermitian Jacobi forms in genus and cogenus~$1$, and Haight's thesis~\cite{haight-2024} covers some aspects in higher degree.
All these references but Shimura's work on the unitary group restrict to the case of imaginary quadratic fields.

We extend the setup from \fref{Section}{sec:hilbert_jacobi_forms}.
We write~$\lD{x}$ for the transpose conjugate of a complex matrix~$x$.
The space~$\MatD{g}(\CC)$ of Hermitian matrices is the set of fixed points of~$\Mat{g}(\CC)$ under the map~$x \mto \lD{x}$.
We set~$x[y] = \lD{y} x y$ for~$x \in \Mat{g}(\CC)$ and~$y \in \Mat{g,h}(\CC)$, extending our previous notation for~$y \in \Mat{g,h}(\RR)$.
For a Hermitian matrix~$x$ and~$\alpha \in \RR$, we write~$x > \alpha$ or~$x \ge \alpha$ if~$x - \alpha$ is positive definite or positive semi-definite, respectively.

Throughout we fix a CM field~$E / F$ with ring of integers~$\cOE / \cOF$.
We let~$\ov{a}$ denote complex conjugation in~$E$ and write~$\traceE(a) = a + \ov{a}$.
We write~$\cOE^\vee$ for the dual of~$\cOE$ with respect to the trace of~$E / \QQ$.
We set~$E_\RR \defeq E \otimes_\QQ \RR$.
After fixing an order for the real embeddings of~$F$ and representatives for the complex embeddings of~$E$, we have~$E_\RR \cong F_\CC \cong \CC^{n_F}$ and~$F_\RR \cong \RR^{n_F}$.
This also allows us to extend~$\traceF$ and~$\normF$ to maps~$E \ra \CC$ in a non-canonical way.
Notions of symmetric and Hermitian matrices extend component-wise to~$E_\RR$ and~$F_\RR$.
Further, we write~$\detF(a) = \normF(\det(a))$ for~$a \in \Mat{g}(E_\RR)$.

We follow Koecher~\cite[\S10]{koecher-1960} and extend the trace form on~$\Mat{g}(\RR)$, which we defined in \fref{Section}{sec:hilbert_jacobi_forms}, to~$\Mat{g}(E_\RR)$ by~$(x,y) \mto \traceF \circ \trace(x y)$, where~$\trace$ is the matrix trace as before.
Note that this is a symmetric, as opposed to Hermitian, bilinear form, but takes real values if~$x, y \in \MatD{g}(E_\RR)$.
The dual~$\MatD{g}(\cOE)^\vee \subset \MatD{g}(E)$ of~$\MatD{g}(\cOE)$ with respect to the trace form consists of Hermitian matrices in~$\MatD{g}(\cOE^\vee)$ with diagonal entries in~$\cOF^\vee$.
There is a right action of~$\GL{g}(\cOE)$ on~$\MatD{g}(\cOE)^\vee$ defined by~$(t,u) \mto t[u]$, since~$\GL{g}(\cOE)$ acts on~$\MatD{g}(\cOE)$.

\subsection{Hermitian Hilbert modular forms}%
\label{ssec:hermitian_modular_jacobi_forms:modular_forms}

Given a complex, square matrix~$\tau$, we define its ``real and imaginary part'' as
\begin{gather*}
  \Re(\tau) \defeq \tfrac{1}{2}( \tau + \lT{\ovsmash\tau} )
  ,\quad
  \Im(\tau) \defeq \tfrac{1}{2 i}( \tau - \lT{\ovsmash\tau} )
  \in
  \MatD{g}(\CC)
  \tx{,}
\end{gather*}
thus maintaining the usual decomposition~$\tau = \Re(\tau) + i \Im(\tau)$.
This notion extends component-wise to matrices with values in~$E_\RR \cong \CC^{n_F}$.
The Hermitian Hilbert upper half space is
\begin{gather}
  \label{eq:def:HSg}
  \HSEg
  \defeq
  \big\{
  \tau \in \Mat{g}(E_\RR) \condsep
  \Im(\tau) > 0
  \big\}
  \tx{.}
\end{gather}
Given~$0 \le h \le g$, we decompose~$\tau \in \HSEg$ as
\begin{gather}
  \label{eq:HSg_decomposition}
  \tau
  =
  \begin{psmatrix} \tau_1 & \lT{z} \\ w & \tau_2 \end{psmatrix}
  \quad\tx{with }
  \tau_1 \in \HS_{E,g-h},\;
  \tau_2 \in \HS_{E,h},\;
  z, w \in \Mat{h,g-h}(E_\RR)
  \tx{.}
\end{gather}
We apply the transpose to~$z$ as opposed to the transpose conjugate to ensure that holomorphic functions in~$\tau$ are also holomorphic in~$z$.

The special unitary group for the extension~$E \slash F$ is defined by
\begin{gather}
  \label{eq:def:SUgg}
  \SU{g,g}(F)
  \defeq
  \Big\{
  \ga \in \SL{2g}(E) \condsep
  \lD{\ga} \begin{psmatrix} 0_g & -1_g \\ 1_g & 0_g \end{psmatrix} \ga
  =
  \begin{psmatrix} 0_g & -1_g \\ 1_g & 0_g \end{psmatrix}
  \Big\}
  \tx{,}
\end{gather}
where~$1_g$ and~$0_g$ denote the $g \times g$ identity and zero matrix.
We view~$\SU{g,g}$ as an algebraic group over~$F$ with integral structure~$\SU{g,g}(\cOF) = \SU{g,g}(F) \cap \Mat{2g}(\cOE)$.
Throughout, we decompose the elements of~$\SU{g,g}(F_\RR)$ into four block entries as~$\ga = \begin{psmatrix} a & b \\ c & d \end{psmatrix}$ for $g \times g$ matrices~$a, b, c, d$ without further mentioning the sizes of~$a, b, c, d$.
We use the following shorthand notation for translations and coordinate rotations on~$\HSEg$:
\begin{gather*}
  \trans(b)
  \defeq
  \begin{psmatrix} 1_g & b \\ 0_g & 1_g \end{psmatrix}
  \tx{,}\quad
  b \in \MatD{g}(E_\RR)
  \tx{,}
  \quad\tx{and}\quad
  \rot(a)
  \defeq
  \begin{psmatrix} a & 0_g \\ 0_g & \lDr{a}{^{-1}} \end{psmatrix}
  \tx{,}\quad
  a \in \GL{g}(E_\RR),
  \det(a) = \det(\ov{a})
  \tx{,}
\end{gather*}
and write the standard involution as
\begin{gather*}
  \sinv
  \defeq
  \begin{psmatrix}
    & -1_g \\
    1_g &
  \end{psmatrix}
  \tx{,}\quad
  \sinv(1)
  \defeq
  \begin{psmatrix}
    0 &         & -1 &         \\
    & 1_{g-1} &    &   \\
    1 &         & 0  &         \\
    &         &    & 1_{g-1}
  \end{psmatrix}
  \tx{.}
\end{gather*}

For a positive integer~$N$ we let~$\GL{g}(\cOE, N)$ and~$\SU{g,g}(\cOF, N)$ be the principal congruence subgroups of~$\GL{g}(\cOE)$ and~$\SU{g,g}(\cOF)$ of level~$N \cOE$ and~$N \cOF$, respectively.

An arithmetic type is a finite dimensional, complex representation~$\rho$ of~$\SU{g,g}(\cOF)$.
We assume throughout that the kernel of~$\rho$ is a congruence subgroup.
The level of~$\rho$ is defined as the level of its kernel.
We write~$V(\rho)$ for the representation space of~$\rho$.
For an even, positive definite Hermitian lattice~$L$ over~$E / F$ with bilinear form~$\langle \cdot\,,\,\cdot \rangle$, rank~$\rk(L) \in \ZZ$, and discriminant form~$\disc(L) = L^\vee \slash L$, we define the Weil representation on~$\CC[ \disc(L)^g ]$ with basis elements~$\frake_{\mu}$ by
\begin{gather}
  \label{eq:def:weil_representation}
  \begin{aligned}
  \rho^{(g)}_L\bigl( \trans(b) \bigr)\,
  \frake_\mu
   & =
  e\bigl( b\, \langle \mu, \mu \rangle \slash 2 \bigr)\,
  \frake_{\mu}
  \tx{,}\qquad
  \rho^{(g)}_L\bigl( \rot(a) \bigr)\,
  \frake_\mu
  =
  \detF(\ov{a})^{\rk(L)}\, \frake_{\mu a^{-1}}
  \tx{,}
  \\
  \rho^{(g)}_L\bigl( \sinv(1) \bigr)\,
  \frake_\mu
   & =
  \frac{e(-n_F \rk(L) \slash 4 )}{\sqrt{\# \disc(L)}}
  \sum_{\mu'_1 \in \disc(L)}
  \mspace{-8mu}
  e\bigl( - \langle \mu_1, \mu'_1 \rangle \bigr)\,
  \frake_{(\mu'_1,\mu_2,\ldots,\mu_g)}
  \tx{.}
  \end{aligned}
\end{gather}
By \fref{Proposition}{prop:siegel_and_klingen_parabolics_generate} this determines the representation uniquely.

The action of~$\SU{g,g}(F_\RR)$ on~$\HSEg$ is defined by
\begin{gather*}
  \begin{psmatrix} a & b \\ c & d \end{psmatrix}\,
  \tau
  \defeq
  (a \tau + b) (c \tau + d)^{-1}
  \tx{.}
\end{gather*}
It yields slash actions on functions~$\HSEg \ra \CC$ associated with weights~$k \in \ZZ$
\begin{gather*}
  \big( f \big|_k\,\ga \big) (\tau)
  \defeq
  \detF(c \tau + d)^{-k}\,
  f(\ga \tau)
  \tx{.}
\end{gather*}
This action extends linearly to functions~$\HSEg \ra V$ for complex vector spaces~$V$.

We extend the notation~$e(x)$ to matrices over~$E_\RR$ by~$e(x) = e(\traceF(x))$, and maintain the convention~$e(x + y) = e(x) e(y)$ for matrices~$x$ and~$y$ of possibly different sizes.
By this definition~$e(t \tau)$ is compatible with the notion of duals in~$\MatD{g}(E_\RR)$ in the following sense: given a complex vector space~$V$, any holomorphic function~$f \defcol \HSEg \ra V$ invariant under the action of~$\trans(N \MatD{g}(\cOE))$ admits a Fourier series expansion
\begin{gather}
  \label{eq:hermitian_fourier_expansion}
  f(\tau)
  =
  \sum_{t \in \frac{1}{N} \MatD{g}(\cOE)^\vee}
  c(f; t)\, e(t \tau)
  \tx{,}\quad
  c(f; t) \in V
  \tx{.}
\end{gather}
For convenience, if~$t$ does not contribute to this expansion we set~$c(f; t) = 0$.

We need two notions of Hermitian modular forms: one for arithmetic types and one for congruence subgroups.
For clarity we separate their definitions.
\begin{definition}%
\label{def:hermitian_modular_form_arithmetic_type}
Given~$k \in \ZZ$ and an arithmetic type~$\rho$, a holomorphic function~$f \defcol \HSEg \ra V(\rho)$ is called a \emph{Hermitian modular form} of weight~$k$ and type~$\rho$ if:
\begin{enumerateroman}
\item
For all~$\ga \in \SU{g,g}(\cOF)$ we have~$f |_k\,\ga = \rho(\ga) \circ f$.
\item
In case of~$F = \QQ$ and~$g = 1$, for all~$n \not\ge 0$ the Fourier coefficients in~\eqref{eq:hermitian_fourier_expansion} satisfy~$c(f; n) = 0$.
\end{enumerateroman}
\end{definition}

\begin{definition}%
\label{def:hermitian_modular_form_level}
Given~$k \in \ZZ$ and~$N \in \ZZ_{\ge 1}$, a holomorphic function~$f \defcol \HSEg \ra \CC$ is called a \emph{Hermitian modular form} of weight~$k$ and level~$N$ if:
\begin{enumerateroman}
\item
For all~$\ga \in \SU{g,g}(\cOF, N)$ we have~$f |_k\,\ga = f$.
\item
\label{it:def:hermitian_modular_form_level:growth}
In case of~$F = \QQ$ and~$g = 1$, for all~$\ga \in \SU{1,1}(\ZZ)$ and all~$n \not\ge 0$ we have~$c(f|_k\, \ga; n) = 0$.
\end{enumerateroman}
\end{definition}

A Hermitian modular form is called a \emph{Hermitian cusp form} if for all~$\ga \in \SU{g,g}(F)$ its Fourier coefficients satisfy~$c(f |_k\, \ga; t) = 0$ if~$t \not> 0$, that is, if~$t$ is not positive definite.

We write~$\rmM^{(g)}_k(\rho) \supseteq \rmS^{(g)}_k(\rho)$ and~$\rmM^{(g)}_k(N) \supseteq \rmS^{(g)}_k(N)$ for the spaces of Hermitian modular forms of weight~$k$ and type~$\rho$ or level~$N$ and their subspaces of cusp forms.
We may omit~$\rho$ and~$N$ from our notation, if~$\rho$ is the trivial representation and~$N = 1$.
The graded algebra of Hermitian modular forms for the full modular group is
\begin{gather*}
  \rmM^{(g)}_\bullet
  \defeq
  \bigoplus_{k \in \ZZ}
  \rmM^{(g)}_k
  \tx{.}
\end{gather*}

For a Hermitian modular form~$f$ of weight~$k$ and type~$\rho$ of level~$N$, the action of~$\rot(\GL{g}(\cOE))$ yields the symmetry relation
\begin{gather}
  \label{eq:hermitian_fourier_coefficient_symmetry}
  \begin{multlined}[0.9\displaywidth]
  c(f; t[a])
  =
  \detF(\ov{a})^k\,
  \rho(\rot(a))^{-1}\,
  c(f; t)
  \\
  \tx{for all }
  t \in \tfrac{1}{N} \MatD{g}(\cOE)^\vee
  \tx{ and }
  a \in \GL{g}(\cOE), \det(a) = \det(\ov{a})
  \end{multlined}
\end{gather}
among the Fourier coefficients of~$f$.
Similar to~\eqref{eq:HSg_decomposition}, we decompose the indices in the Fourier series~\eqref{eq:hermitian_fourier_expansion} as
\begin{gather}
  \label{eq:hermitian_fourier_index_decomposition}
  t
  =
  \begin{psmatrix} n & \lT{\ovsmash{r}} \\ r & m \end{psmatrix}
  \quad\tx{with }
  n \in \tfrac{1}{N} \MatD{g-h}(\cOE)^\vee,\;
  m \in \tfrac{1}{N} \MatD{h}(\cOE)^\vee,\;
  r \in \tfrac{1}{N} \Mat{h,g-h}(\cOE^\vee)
  \tx{.}
\end{gather}

\begin{remark}%
\label{rm:hermitian_fourier_index_decomposition}
Note the absence of the factor~$\frac{1}{2}$ in front of~$r$ compared to the Siegel and Hilbert-Siegel modular case.
The difference stems from the contribution to the exponential, which here is~$e(t \tau) = e(n \tau_1 + \lT{r} z + \lT{\ovsmash{r}}w + m \tau_2)$ and in the case of Siegel modular forms is~$e(t \tau) = e(n \tau_1 + 2 \lT{r} z + m \tau_2)$.
\end{remark}

The Koecher Principle guarantees that the analogue of the growth condition is satisfied automatically for Hermitian modular forms of genus greater than~$1$ or over base fields~$F \ne \QQ$.
We also need the asymptotic dimension of spaces of Hermitian modular forms.
Both follow from geometric arguments.

\begin{theorem}%
\label{thm:hermitian_koecher_principle}
Given a Hermitian modular form~$f$ of genus~$g$, weight~$k \in \ZZ$, and type~$\rho$ or level~$N$, for all~$\ga \in \SU{g,g}(\cOF)$ we have~$c(f |_k\, \ga; t) = 0$ if~$t \not\ge 0$.
\end{theorem}

\begin{proof}
If~$F = \QQ$ and $g = 1$, this is part of the definition.
In all other cases, by Baily--Borel~\cite{baily-borel-1966} the modular orbifold~$\SU{g,g}(\cOF, N) \backslash \HSEg$ admits a normal compactification with boundary of codimension at least~$2$.
\end{proof}

\begin{remark}
For an analytic proof, if~$F \ne \QQ$ and~$g = 1$, the Koecher Principle is the statement of Theorem~1.4 of~\cite{garrett-1990}.
If~$g > 1$, the proof for Siegel modular forms (see Hilfssatz~3.5 of~\cite{freitag-1983}) extends to the Hermitian case.
\end{remark}

\begin{proposition}%
\label{prop:hermitian_modular_forms_dimension_asymptotic}
For fixed genus~$g$ we have
\begin{gather*}
  \dim\, \rmM^{(g)}_k
  \asymp_{E}
  k^{n_F g^2}
  \tx{ as }
  k \ra \infty
  \tx{.}
\end{gather*}
\end{proposition}

\begin{proof}
Hermitian modular forms of weight~$k$ correspond to sections of the~$k$\thdash{} power of a suitable line bundle~$\cM$ on the orbifold~$\SU{g,g}(\cOF) \backslash \HSEg$ of dimension~$n_F g^2$, which is very ample by Baily--Borel~\cite{baily-borel-1966} for sufficiently large~$k$.
\end{proof}

For integers~$0 \le h \le g$ we have the following parabolic subgroups of~$\SU{g,g}(F_\RR)$:
\begin{align*}
  P_h(F_\RR)
   & \defeq
  \Bigg\{
  \begin{psmatrix}
    a &   & b &                \\
    & u   &   &                \\
    c &   & d &                \\
    &   &   & \lDr{u}{^{-1}}
  \end{psmatrix}\;
  \begin{psmatrix}
    1_{g-h} &     &         & \lD{\mu}      \\
    \lambda & 1_h & \mu     & \kappa        \\
    &     & 1_{g-h} & -\lD{\lambda} \\
    &     &         & 1_h
  \end{psmatrix}
  \condsep
  \begin{split}
   &
  u \in \GL{h}(E_\RR), \det(u) = \det(\ov{u}),\,
  \begin{psmatrix} a & b \\ c & d \end{psmatrix} \in \SU{g-h,g-h}(F_\RR),
  \\
   &
  \mu, \lambda \in \Mat{h,g-h}(E_\RR),\,
  \kappa + \lD{\la} \mu \in \MatD{h}(E_\RR)
  \end{split}
  \Bigg\}
  \tx{,}
  \\
  Q_h(F_\RR)
   & \defeq
  \Bigg\{
  \begin{psmatrix}
    a &   & b &                \\
    & 1_{g-h} &   &                \\
    c &   & d &                \\
    &   &   & 1_{g-h}
  \end{psmatrix}\;
  \begin{psmatrix}
    1_{g-h} &     &         & \lD{\mu}      \\
    \lambda & 1_h & \mu     & \kappa        \\
    &     & 1_{g-h} & -\lD{\lambda} \\
    &     &         & 1_h
  \end{psmatrix}
  \condsep
  \begin{split}
   &
  \begin{psmatrix} a & b \\ c & d \end{psmatrix} \in \SU{g-h,g-h}(F_\RR),
  \\
   &
  \mu, \lambda \in \Mat{h,g-h}(E_\RR),\,
  \kappa + \lD{\la} \mu \in \MatD{h}(E_\RR)
  \end{split}
  \Bigg\}
  \tx{,}
\end{align*}
where all omitted matrix entries are zero and we suppress~$g$ from our notation~$P_h$.
These groups correspond to algebraic subgroups of~$\SU{g,g}$.
We write~$P_h(\cOF)$ and~$Q_h(\cOF)$ for their integral structures.
As a special case we have the Siegel-type parabolic~$P_g(F_\RR)$, whose unipotent radical and Levi factor consist of all~$\trans(b)$, $b \in \MatD{g}(E_\RR)$ and~$\rot(u)$, $u \in \GL{g}(E_\RR)$ with~$\det(u) = \det(\ov{u})$, respectively.

We conclude with a statement on generators of the special unitary group.
Later, we need the case~$h = 1$, but the proof remains the same for all~$h$ in the proposition.

\begin{proposition}%
\label{prop:siegel_and_klingen_parabolics_generate}
For\/~$1 \le h < g$, the special unitary group is generated as follows:
\begin{gather}
  \label{eq:siegel_and_klingen_parabolics_generate}
  \SU{g,g}(\cOF)
  =
  \big\langle P_g(\cOF),\, Q_h(\cOF) \big\rangle
  \tx{.}
\end{gather}
\end{proposition}

\begin{proof}
For all~$1 \le h < g$ the definition of~$Q_h(\cOF)$ shows that~$\SU{1,1}(\cOF) \subset Q_h(\cOF)$ viewed as the set of matrices
\begin{gather*}
  \begin{psmatrix}
    a &         & b &         \\
    & 1_{g-1} &   &   \\
    c &         & d &         \\
    &         &   & 1_{g-1}
  \end{psmatrix}
  \quad\tx{with }
  \begin{psmatrix} a & b \\ c & d \end{psmatrix} \in \SU{1,1}(\cOF)
  \tx{,}
\end{gather*}
where zero entries are omitted.
The proposition thus follows if we show that
\begin{gather}
  \label{eq:prf:prop:siegel_and_klingen_parabolics_generate:unitary_rhs}
  \SU{g,g}(\cOF)
  =
  \big\langle P_g(\cOF),\, \SU{1,1}(\cOF) \big\rangle
  \tx{.}
\end{gather}
This is a consequence of Theorem~9.2.6, p.~534 by Hahn--O'Meara~\cite{hahn-omeara-1989}.
We have to translate the statement~$\rmS\rmU_{2n}(R) = \rmE\rmU_{2n}(R)$ in Hahn--O'Meara's theorem and verify their condition~(iii).

We will need the notion of form rings~$(R,\Lambda)$, including the notion of form parameters~$\Lambda$, that appears in~9.1.4 of~\cite{hahn-omeara-1989}.
They are defined in Section~5.1C.
Form ideals~$(\fraka,\Ga)$ are defined in Section~5.2D; we will only need the trivial one~$(\fraka, \Ga) = (R,\Lambda)$.
Specifically, in our case we consider the ring~$R = \cOE$ with complex conjugation, denoted~$J$ by Hahn--O'Meara, and hence~$\epsilon = -1$.
We have~$\Lambda_{\max} = \cOF$, the ring of invariants in~$\cOE$ under complex conjugation.
This allows us to choose~$\Lambda = \cOF$ as the form parameter.
In notation of Theorem~9.2.6 of~\cite{hahn-omeara-1989}, we have~$R_0 = \cOF$.
Both~$\cOE \subset E$ and~$\cOF \subset F$ are Hasse domains, whose definition Hahn--O'Meara give in their Section~2.2E.
This confirms condition~(iii) of Theorem~9.2.6.

Next we translate the notion of unitary groups.
Our goal is to show that~$\rmS\rmU_{2g}(R)$ in Hahn--O'Meara coincide with our groups~$\SU{g,g}(\cOF)$.
The unitary group~$\U{g,g}(\cOF)$ (defined as our~$\SU{g,g}(\cOF)$ except for the determinant condition) is called a hyperbolic unitary group by Hahn--O'Meara.
The underlying quadratic module is defined in Section~5.3 of~\cite{hahn-omeara-1989}.
Note that the matrix~$F$ at the beginning of Section~5.3 is printed incorrectly.
Its bottom left block entry needs to be~$\epsilon$ times the identity, thus matching the form that appears in~\eqref{eq:def:SUgg} of the present work.
Statement~5.3.1~of~\cite{hahn-omeara-1989} asserts that~$\rmU_{2n}(R, \Lambda)$ coincides with our~$\U{n,n}(\cOF)$ as desired.
Section~5.3A, p.~226, contains the definition of~$\rmU_{2n}(R)$ and in the text following afterwards the definition of~$\rmS\rmU_{2n}(R)$, which coincides with~$\SU{n,n}(\cOF)$ in this work.

It remains to translate the elementary subgroup~$\rmE\rmU_{2g}(R)$.
It is defined in terms of generators, as we will see.
We have to show that the right hand side of~\eqref{eq:prf:prop:siegel_and_klingen_parabolics_generate:unitary_rhs} contains all of them.
The definition of~$\rmE\rmU_{2g}(R,\Lambda)$ is given in the discussion after Statement~5.3.1~of~\cite{hahn-omeara-1989}.
It is the group generated by the elementary unitary matrices~$E_{i\!j}(r)$ defined there as well.
We have~$\rmE\rmU_{2g}(R) = \rmE\rmU_{2g}(R, \Lambda)$ by the definition on p.~226.
For convenience, we reproduce the definition of~$E_{i\!j}(r)$ and adopt our notation.
All entries of the following matrices~$e_{m,n}(r)$ and~$s_{m,n}(r)$ that we do not mention are zero:
For~$1 \le m \ne n \le g$ and~$r \in \cOE$, we let~$e_{m,n}(r) \in \SL{g}(\cOE)$ be the usual elementary matrix with~$r$ in position~$(m,n)$ and diagonal entries equal to~$1$.
For~$1 \le m, n \le g$ and~$r \in \cOE$ if~$m \ne n$ or~$r \in \cOF$ if~$m = n$, we let~$s_{m,n}(r)$ be the matrix with~$r$ in position~$(m,n)$ and~$\ov{r}$ in position~$(n,m)$.
Using~$e_{m,n}(r)$ and~$s_{m,n}(r)$, we build the unitary elementary matrices
\begin{gather}
  \label{eq:prf:prop:siegel_and_klingen_parabolics_generate:unitary_elementary_matrices}
  \trans\big( s_{m,n}(r) \big),\;
  \lD{\trans\big(s_{m,n}(r) \big)},\;
  \rot\big( e_{n,m}(r) \big)
  \tx{.}
\end{gather}
Hahn--O'Meara define~$\rmE\rmU_{2g}(R, \Lambda)$ as the group generated by~\eqref{eq:prf:prop:siegel_and_klingen_parabolics_generate:unitary_elementary_matrices} as~$m$, $n$, and~$r$ run through all possible choices.

Thus to finish the proof, we have to show that the right hand side of~\eqref{eq:prf:prop:siegel_and_klingen_parabolics_generate:unitary_rhs} contains all matrices in~\eqref{eq:prf:prop:siegel_and_klingen_parabolics_generate:unitary_elementary_matrices}.
From the definition of the parabolic group, we see directly that
\begin{gather*}
  \trans\big( s_{m,n}(r) \big),\,
  \rot\big( e_{n,m}(r) \big)
  \in
  P_g(\cOF)
  \quad\tx{ for all } m, n, r
  \tx{.}
\end{gather*}
Let~$u_i$ be the permutation matrix that swaps the~$1$\stdash{} and the~$i$\thdash{} component of~$\cOE^g$.
Since we have~$\rot(u_i) \in P_g(\cOF)$ and~$\sinv(1) \in \SU{1,1}(\cOF) \subset \SU{g,g}(\cOF)$, we also have
\begin{gather*}
  \sinv
  =
  \sinv(1) \cdot
  \rot(u_2) \sinv(1) \rot(u_2) \cdots
  \rot(u_g) \sinv(1) \rot(u_g)
  \in
  \big\langle P_g(\cOF),\, \SU{1,1}(\cOF) \big\rangle
  \tx{.}
\end{gather*}
Now the equality
\begin{gather*}
  \lD{\trans\big( s_{m,n}(r) \big)}
  =
  \sinv^{-1}\,
  \trans\big( s_{m,n}(-\ov{r}) \big)\,
  \sinv
  \in
  \big\langle P_g(\cOF),\, \SU{1,1}(\cOF) \big\rangle
\end{gather*}
finishes the proof.
\end{proof}

\subsection{Hermitian Hilbert--Jacobi forms}%
\label{ssec:hermitian_modular_jacobi_forms:jacobi_forms}

As for Hermitian Hilbert modular forms, we allow ourselves to refer to Hermitian Hilbert--Jacobi forms as Hermitian Jacobi forms.
Recall the setup from \fref{Section}{ssec:hilbert_jacobi_forms:jacobi_forms_definition}.
We assume~$1 \le h < g$ throughout this section.
We restrict to the case of Hermitian Jacobi forms whose Jacobi index is a free lattice, as these are the ones that arise from the Fourier--Jacobi expansion of Hermitian modular forms.
The Hermitian Jacobi upper half space is defined by
\begin{gather}%
  \label{eq:def:HJgh}
  \HJ{E,g-h,h}
  \defeq
  \HS_{E,g-h} \times \Mat{h,g-h}(E_\RR) \times \Mat{h,g-h}(E_\RR)
  \tx{,}
\end{gather}
with typical elements denoted~$(\tau_1,z,w)$.
Given~$(\tau_1,z,w) \in \HJ{E,g-h,h}$ all~$\tau_2 \in \HS_{E,h}$ such that~$\Im(\tau_2)$ has sufficiently large eigenvalues yield~$\tau \in \HSEg$ as in~\eqref{eq:HSg_decomposition}.
For~$\la, \mu \in \Mat{h,g-h}(E_\RR)$, we set
\begin{gather}
  \label{eq:def:hermitian_jacobi_group_translation}
  \transJU(\la, \mu)
  \defeq
  \rot\bigl( \begin{psmatrix} 1 & 0 \\ \la & 1 \end{psmatrix} \bigr)\,
  \trans\Bigl( \begin{psmatrix} 0 & \lD{\mu} \\ \mu & 0 \end{psmatrix} \Bigr)
  \in
  Q_h(F_\RR)
  \tx{.}
\end{gather}
The Jacobi slash action of~$\ga \in Q_h(F_\RR)$ on functions~$\HJ{E,g-h,h} \ra \CC$ with~$k \in \ZZ$ and~$m \in \MatD{g}(E_\RR)$ is
\begin{gather}
  \label{eq:def:hermitian_jacobi_slash_action}
  \big( \phi \big|_{k,m}\, \ga \big) (\tau_1, z, w)
  \defeq
  e(-m \tau_2)\,
  \bigl(
  \phi(\tau_1, z, w)\, e(m \tau_2) \big|_k\, \ga
  \bigr)
  \tx{.}
\end{gather}

Given a holomorphic~$\phi \defcol \smash{\HJ{E,g-h,h}} \ra V$ for a complex vector space~$V$, suppose that~$\phi$ is invariant under the slash action of~$\trans(N \smash{\MatD{g}}(\cOE)) \subset Q_h(\cOF)$.
In analogy with~\eqref{eq:hermitian_fourier_expansion}, we have~$\phi = 0$ if~$m \not\in \tfrac{1}{N} \MatD{h}(\cOE)^\vee$.
Further,~$\phi$ admits the Fourier series expansion
\begin{gather}
  \label{eq:def:hermitian_jacobi_forms_fourier_expansion}
  \phi(\tau_1, z, w)
  =
  \sum_{\substack{
      n \in \frac{1}{N} \MatD{g-h}(\cOE)^\vee \\
      r \in \frac{1}{N} \Mat{h,g-h}(\cOE^\vee)
    }}
  \mspace{-24mu}
  c(\phi; n, r)\,
  e\bigl( n \tau_1 + \lT{r} z + \lT{\ovsmash{r}} w \bigr)
  \tx{,}
  \quad
  c(\phi; n, r)
  \in
  V
  \tx{.}
\end{gather}
We set~$c(\phi; n, r) = 0$ if~$(n,r)$ does not contribute to this expansion.

In the next definitions we employ the block decomposition of~$\tau$ and~$t$ in~\eqref{eq:HSg_decomposition} and~\eqref{eq:hermitian_fourier_index_decomposition}.
The notion of Hermitian Jacobi forms of level~$N$ is merely required in the proof of \fref{Proposition}{prop:hermitian_jacobi_forms_torsion_points_modularity}.

\begin{definition}%
\label{def:hermitian_jacobi_form}
Given~$k \in \ZZ$, $m \in \MatD{h}(E)$, and an arithmetic type~$\rho$ of genus~$g$, we call a holomorphic function~$\HJ{E,g-h,h} \ra V(\rho)$ a \emph{Hermitian Jacobi form} of genus~$g-h$, cogenus~$h$, weight~$k$, index~$m$, and type~$\rho$ if:
\begin{enumerateroman}
\item
For all~$\ga \in Q_h(\cOF) \subset \SU{g,g}(\cOF)$, we have~$\phi \big|_{k,m}\, \ga = \rho(\ga) \circ \phi$.

\item
If~$F = \QQ$ and~$h = g-1$, for all~$n, r$ the Fourier coefficients in~\eqref{eq:def:hermitian_jacobi_forms_fourier_expansion} satisfy~$c(\phi; n, r) = 0$ if~$t \not\ge 0$.
\end{enumerateroman}
\end{definition}

\begin{definition}%
\label{def:hermitian_jacobi_form_level}
Given~$k \in \ZZ$, $m \in \MatD{h}(E)$, and~$N \in \ZZ_{\ge 1}$, we call a holomorphic function~$\HJ{E,g-h,h} \ra \CC$ a \emph{Hermitian Jacobi form} of genus~$g-h$, cogenus~$h$, weight~$k$, index~$m$, and level~$N$ if:
\begin{enumerateroman}
\item
For all~$\ga \in Q_h(\cOF) \cap \SU{g,g}(\cOF, N)$, we have~$\phi \big|_{k,m}\, \ga = \phi$.

\item
If~$F = \QQ$ and~$h = g-1$, for all~$\ga \in \SU{1,1}(\ZZ)$ and for all~$n, r$ we have~$c(\phi |_{k,m}\, \ga; n, r) = 0$ if~$t \not\ge 0$.
\end{enumerateroman}
\end{definition}

A Hermitian Jacobi form of genus~$g-h$ and cogenus~$h$ is called a \emph{Hermitian Jacobi cusp form} if for all~$\ga \in Q_h(F)$ its Fourier coefficients satisfy~$c(\phi |_{k,m}\, \ga; n, r) = 0$ for~$t \not> 0$.
Note that by this definition Hermitian Jacobi cusp forms vanish if~$m \not> 0$.

The spaces of Hermitian Jacobi forms of weight~$k$, index~$m$, and type~$\rho$ or level~$N$ are written as~$\rmJ^{(g-h)}_{k,m}(\rho)$ and~$\rmJ^{(g-h)}_{k,m}(N)$.
As in the case of Hermitian modular forms, we may omit~$\rho$ or~$N$ from our notation, if~$\rho$ is the trivial representation or~$N = 1$.
We will not need notation for the space of Jacobi cusp forms.

The proof of the Koecher Principle in \fref{Theorem}{thm:hermitian_koecher_principle} extends to the case of Hermitian Jacobi forms, for which we record without further argument:

\begin{theorem}%
\label{thm:hermitian_jacobi_forms_koecher_principle}
Given a Hermitian Jacobi form~$\phi$ of genus~$g-h$, cogenus~$h$, weight~$k \in \ZZ$, and type~$\rho$ or level~$N$, for all~$\ga \in Q_h(\cOF) \subset \SU{g,g}(\cOF)$ we have~$c(\phi |_{k,m}\, \ga; n, r) = 0$ if~$t \not\ge 0$, where~$t$ is as in~\eqref{eq:hermitian_fourier_index_decomposition}.
Further, the Jacobi index~$m$ is totally positive semi-definite if~$\phi \ne 0$.
\end{theorem}

Generalizing the specialization of elliptic Jacobi forms to torsion points in~\eqref{eq:def:jacobi_forms_torsion_points}, for~$\phi \in \rmJ^{(g-h)}_{k,m}(N)$ and~$\alpha, \beta \in \Mat{h,g-h}(E)$ we define:
\begin{gather}
  \label{eq:def:hermitian_jacobi_forms_torsion_points}
  \phi[\alpha, \beta](\tau_1)
  \defeq
  e(-m \tau_2)\;
  \Bigl(
  \bigl( \phi(\tau_1, z, w) e(m \tau_2) \bigr) \big|_k\,
  \transJU(\ov\alpha, \ov\beta)
  \Bigr)_{z = w = 0}
  \tx{.}
\end{gather}
Note that the complex conjugates of~$\alpha$ and~$\beta$ appear on the right hand side.
We record the explicit expansion in terms of~$\phi$ as a separate lemma for later use.

\begin{lemma}%
\label{la:hermitian_jacobi_forms_torsion_points}
We have
\begin{gather*}
  \phi[\alpha, \beta](\tau_1)
  =
  e\bigl( m (\tau_1 [\lT{\alpha}] + \ov{\beta} \lT{\alpha} ) \bigr)\,
  \phi\bigl( \tau_1,\, \alpha\, \lTr{\tau}{_1} + \beta,\, \ovsmash{\alpha}\, \tau_1 + \ovsmash{\beta} \bigr)
  \tx{.}
\end{gather*}
\end{lemma}

\begin{proof}
We compute
\begin{align*}
  \transJU(\ov\alpha, \ov\beta)\,
  \begin{psmatrix} \tau_1 & \lT{z} \\ w & \tau_2 \end{psmatrix}
  =
  \begin{psmatrix}
    \tau_1 & \lT{(}z + \alpha\, \lTr{\tau}{_1} + \beta) \\
    w + \ovsmash\alpha \tau_1 + \ovsmash\beta & \tau_2 + w \lT{\alpha} + \ovsmash\alpha \lT{z} + \ovsmash\alpha \tau_1 \lT{\alpha} + \ov{\beta} \lT{\alpha}
  \end{psmatrix}
  \tx{,}
\end{align*}
and insert this into~\eqref{eq:def:hermitian_jacobi_forms_torsion_points} to confirm the statement.
\end{proof}

\begin{definition}%
\label{def:hermitian_jacobi_forms_admissible_torsion_point}
We call
\begin{gather*}
  \big(
  \tau_1,\,
  \alpha\, \lTr{\tau}{_1} + \beta,\,
  \ov{\alpha}\, \tau_1 + \ov{\beta}
  \big)
  \in
  \HJ{E,g-1,1}
\end{gather*}
an \emph{admissible~$N$\nbd{}torsion point} if~$\alpha, \beta \in \frac{1}{N} \Mat{1,g-1}(\cOE)$ and there exists~$M \in \ZZ_{>0}$ such that the fractional ideal generated by~$1$ and the entries of~$\alpha$ equals~$M^{-1} \cOE$.
\end{definition}

We extend this notion to elements~$\tau \in \HSEg$ with arbitrary~$\tau_2$ via the decomposition in~\eqref{eq:HSg_decomposition}.
To simplify the presentation, we will also refer to the pair~$(\alpha,\beta)$ as a torsion point.
We employ the following density statement several times in \fref{Section}{sec:convergence}.

\begin{lemma}%
\label{la:admissible_torsion_points_dense}
The set of admissible torsion points is dense in~$\HJ{E,g-1,1}$.
\end{lemma}

\begin{proof}
By existence of complex conjugation of~$E / F$, there are infinitely many rational primes~$p$ that are totally split in~$F$, that is, $p \cOF = \frakp_1 \cdots \frakp_{n_F}$, and every~$\frakp_i$ is inert in~$E / F$.
For such~$p$, a $p$\nbd{}torsion point~$(\alpha,\beta)$ is admissible with~$M = p$ in \fref{Definition}{def:hermitian_jacobi_forms_admissible_torsion_point} if~$p \alpha_j + \frakp_i \cOE = \cOE$ for all~$1 \le i \le n_F$ and all~$1 \le j \le g-1$.
Since~$\frakp_i$ is inert, we have~$p \alpha_j + \frakp_i \cOE = \frakp_i \cOE$ or~$p \alpha_j + \frakp_i \cOE = \cOE$.
We let~$\frac{1}{p} \frakt_i \subset \frac{1}{p} \cOE$ be the~$\ZZ$\nbd{}module of~$\alpha_j$ for which the former equality holds, and observe that it has index~$p$.
The set of admissible $p$\nbd{}torsion points contains
\begin{gather*}
  \tfrac{1}{p} \Mat{1,g-1}\bigl( \cOE \setminus \bigcup_{i = 1}^{n_F} \frakt_i \bigr)
  \times
  \tfrac{1}{p} \Mat{1,g-1}(\cOE)
  \tx{.}
\end{gather*}
We view this coordinate-wise in~$E_\RR \cong \RR^{2 n_F}$, where we can identify~$\cOE$ with~$\ZZ^{2 n_F}$ after a suitable change of basis, which preserves density.
In~$\RR^{2 n_F}$ the lattices~$\frac{1}{p} \ZZ^{2 n_F}$ with any union of~$n_F$ index~$p$ sublattices removed are dense as~$p \ra \infty$, finishing our proof.
\end{proof}

Specializing Hermitian Jacobi forms to torsion points, we obtain Hermitian modular forms, which we will also employ in \fref{Section}{sec:convergence}.

\begin{proposition}%
\label{prop:hermitian_jacobi_forms_torsion_points_modularity}
Given~$\phi \in \rmJ^{(g-1)}_{k,m}$ and~$\alpha, \beta \in \frac{1}{N} \Mat{1,g-1}(\cOE)$, we have
\begin{gather}
  \label{eq:hermitian_jacobi_forms_torsion_points_modularity}
  \phi[\alpha, \beta]
  \in
  \rmM^{(g-1)}_k(N_F N^2)
  \tx{,}
\end{gather}
where~$N_F^{-1} \ZZ = \traceF( \frac{1}{2} \cOF^\vee )$.
Further, if~$\phi$ is a cusp form, then so is~$\phi[\alpha,\beta]$.
\end{proposition}

\begin{proof}
We define
\begin{gather*}
  \iota
  \defcol
  \SU{g-1,g-1}(F_\RR)
  \lhra
  \SU{g,g}(F_\RR)
  \tx{,}\quad
  \begin{psmatrix} a & b \\ c & d \end{psmatrix}
  \lmto
  \begin{psmatrix}
    a &   & b &   \\
    & 1 &   &   \\
    c &   & d &   \\
    &   &   & 1
  \end{psmatrix}
  \tx{.}
\end{gather*}
Then the modular invariance of~$\phi[\alpha,\beta]$ follows from the inclusion
\begin{gather*}
  \transJU\bigl( \ov\alpha, \ov\beta \bigr)\,
  \iota\bigl( \SU{g-1,g-1}(\cOF, N_F N^2) \bigr)\,
  \transJU\bigl( \ov\alpha, \ov\beta \bigr)^{-1}
  \subseteq
  \iota\bigl( \SU{g-1,g-1}(\cOF) \bigr)
  \tx{,}
\end{gather*}
which follows analogously to \fref{Lemma}{la:torsion_points_modularity}.

To handle the growth condition and cusp form condition, we apply the analogue of~\fref{Lemma}{la:jacobi_forms_torsion_points_sl2_shift_action}.
For~$\ga \in \SU{g-1,g-1}(F)$ there exist~$\alpha', \beta' \in \frac{1}{N'} \Mat{1,g-1}(\cOE)$, for some positive integer~$N'$, and a root of unity~$c$ such that~$\phi[\alpha,\beta] |_k\, \ga = c\, (\phi |_{k,m}\, \ga)[\alpha',\beta']$.
In particular, it suffices to examine the Fourier coefficients of~$\phi[\alpha,\beta]$ for Hermitian Jacobi forms of general level.

We consider the growth condition if~$F = \QQ$ and~$g-1 = 1$.
Any Fourier coefficient~$c(\phi; n, r) \ne 0$ satisfies~$t \ge 0$ using the decomposition of~$t$ in~\eqref{eq:hermitian_fourier_index_decomposition}.
It contributes to the Fourier coefficient of index~$t[\lT{(}1,\lT{\ovsmash\alpha})] \ge 0$ of~$\phi[\alpha,\beta]$.
In other words, Fourier coefficients of~$\phi[\alpha,\beta]$ of negative index receive no contribution and therefore vanish as required.

Similarly, for general~$F$ and~$g$ if~$\phi$ is a cusp form then~$c(\phi |_{k,m} \ga; n, r) \ne 0$ implies~$t > 0$.
This Fourier coefficient contributes to the coefficient of index~$t[\lT{(}1_{g-1}, \lT{\ovsmash\alpha})] > 0$, meaning that only Fourier coefficients of~$\phi[\alpha,\beta]$ of definite index may be nonzero.
\end{proof}

For positive definite Jacobi index~$m$ we set
\begin{gather*}
  \disc(m)^{g-h}
  \defeq
  \Mat{h,g-h}(\cOE^\vee)
  \bigslash
  m \Mat{h,g-h}(\cOE)
  \tx{.}
\end{gather*}
Note that in contrast to the case of Siegel--Jacobi forms, we here have the quotient of matrices with entries in the inverse different as opposed to~$\ZZ$ by multiples of~$m$ as opposed to~$2m$.
This stems from the different normalization of Fourier indices discussed in \fref{Remark}{rm:hermitian_fourier_index_decomposition}.
Hermitian Jacobi theta series of positive definite index~$m$, for~$\mu \in \disc(m)^{g-h}$, are defined as
\begin{gather}
  \label{eq:def:hermitian_jacobi_theta}
  \theta_{m,\mu}(\tau_1, z, w)
  \defeq
  \sum_{r \in \mu + m \Mat{h,g-h}(\cOE)}
  e\big( m^{-1}[r] \tau_1 + \lT{r} z + \lT{\ovsmash{r}} w \big)
  \tx{,}\qquad
  \theta_m
  \defeq
  \bigl( \theta_{m,\mu} \bigr)_{\mu \in \disc(m)^{g-h}}
  \tx{.}
\end{gather}
The vector-valued Hermitian Jacobi form~$\theta_m$ has weight~$h$ for the Weil representation~$\rho^{(g-h)}_m$ in~\eqref{eq:def:weil_representation}, where we identify the index~$m$ with the Hermitian lattice that has bilinear form~$(x,y) \mto \lD{y} m x$.
In the special case~$g = 2, h = 1$, Haverkamp established their transformation behavior~\cite{haverkamp-1995}.
The general case follows by inspection of the Fourier expansion to derive the behavior under~$\trans$ and~$\rot$ and by restriction to the Jacobi upper half space~$\HS \times \CC^{2hn_F}$ along~$\HS \hra \HS^{n_F} = \HS_F$, $\tau \mto \xi \tau$ for totally positive~$\xi \in F$ to derive the behavior under~$\sinv(1)$.

\subsection{Connection to Hilbert--Jacobi forms}%
\label{ssec:hermitian_modular_jacobi_forms:connection_hilbert_jacobi_forms}

In this section, we establish a connection between Hermitian Jacobi forms of genus~$1$ and Hilbert--Jacobi forms.
We assume throughout that~$h = g-1$.
To make this connection, we need to translate Jacobi indices, we need to relate the variable~$z$ on the Hilbert--Jacobi upper half space to the variables~$z,w$ on the Hermitian upper half space, and we need a notion of vanishing orders.

We write~$L_m$ for the free Hermitian lattice over~$E / F$ associated with~$m$, whose underlying module is~$\cOE^h$ and whose Hermitian form is given by~$(x,y) \mto 2 \lT{\ovsmash{x}} m y$.
Its reduced trace lattice relative to~$E / F$ is
\begin{gather*}
  \traceE(L_m)
  \defeq
  \bigl( \cOE^h,\, (x,y) \mto \tfrac{1}{2} \traceE( 2 \lT{\ovsmash{x}} m y ) \bigr)
  \tx{,}
\end{gather*}
where we consider~$\cOE^h$ as an~$\cOF$\nbd{}module.
Note that~$\langle x \rangle_{\traceE(L_m)} = \langle x \rangle_{L_m}$, since~$L_m$ is Hermitian.
This also shows that~$\traceE(L_m)$ is integral.
The upper half spaces are related by the map
\begin{gather}
  \label{eq:hermitian_jacobi_upper_half_space_to_hilbert}
  \HS_F \times \traceE(L_m)_\CC
  =
  \HS_F \times \cOE^h \otimes_\ZZ \CC
  \lra
  \HS_F \times E_\RR^h \times E_\RR^h,\quad
  (\tau, z \otimes c)
  \lmto
  (\tau, \ov{z} \cdot c, z \cdot c)
  \tx{,}
\end{gather}
which extends linearly in the second component.
The definition of vanishing orders leverages the fact that the Fourier index~$n$ lies in the totally real subfield~$F \subset E$:
\begin{gather}
  \label{eq:hermitian_jacobi_form_vanishing_order}
  \ordnorm(\phi)
  \defeq
  \sup\big\{
  \nu \in \RR
  \condsep
  c(\phi; n, r) = 0
  \tx{ for all } n \in F, r \in E^h
  \tx{ with } \normF(n) < \nu^{n_F}
  \big\}
  \tx{.}
\end{gather}
We let~$\rmJ^{(1)}_{k,m}[\nu]$ denote the spaces of such Jacobi forms of vanishing order~$\ordnorm$ at least~$\nu$.

\begin{proposition}%
\label{prop:hermitian_jacobi_form_to_hilbert}
Given a weight~$k \in \ZZ$, a Hermitian Jacobi index~$m \in \MatD{h}(\cOE)^\vee$, and a vanishing order~$\nu \in \RR$, we have a natural inclusion
\begin{gather*}
  \rmJ^{(1)}_{k,m}[\nu]
  \lhra
  \rmJHilb{k,\traceE(L_m)}[\nu]
\end{gather*}
under pullback along~\eqref{eq:hermitian_jacobi_upper_half_space_to_hilbert}.
\end{proposition}

\begin{proof}
The map is clearly injective, preserves holomorphicity and vanishing order.

For the time being, we assume that~$F = \QQ$, for which we have to check the growth condition.
We set~$\omega_E \defeq (D_E + \sqrt{D_E}) \slash 2$, where~$D_E$ is the discriminant of~$E$.
This yields an isomorphism between~$\cOE^h \otimes_\ZZ \CC$ and~$\CC^h \oplus \omega_E \CC^h$ with elements~$z_1 + \omega_E z_2$.
We exhibit individual Fourier terms to deduce the growth condition.
The term~$e(n \tau_1 + \lT{r} z + \lT{\ovsmash{r}} w)$ yields
\begin{gather*}
  e\bigl( n \tau_1 + \lT{r} (z_1 + \ov{\omega_E} z_2) + \lT{\ovsmash{r}} (z_1 + \omega_E z_2) \bigr)
  =
  e\bigl( n \tau_1 + \trace_{E \slash \QQ}(\lT{r}) z_1 + \trace_{E \slash \QQ}(\lT{\ovsmash{r}} \omega_E) z_2 \bigr)
  \tx{.}
\end{gather*}
Thus we have to check that for~$t \ge 0$ the matrix
\begin{gather*}
  t_E
  \defeq
  \mfrac{1}{2}
  \begin{psmatrix}
    2n                                   & \trace_{E \slash \QQ}(r)                    & \trace_{E \slash \QQ}(\ovsmash{r} \omega_E) \\
    \trace_{E \slash \QQ}(\lT{r})     & \trace_{E \slash \QQ}(m)                    & \trace_{E \slash \QQ}(m \omega_E) \\
    \trace_{E \slash \QQ}(\lT{\ovsmash{r}} \omega_E) & \trace_{E \slash \QQ}(\lT{m} \omega_E) & \trace_{E \slash \QQ}(m |\omega_E|^2)
  \end{psmatrix}
\end{gather*}
is positive semi-definite.
We have
\begin{gather*}
  t_E\Bigl[ \begin{psmatrix} u \\ x \\ y \end{psmatrix} \Bigr]
  =
  t\bigl[ \begin{psmatrix} u \\ x + \omega_E y \end{psmatrix} \bigr]
  \quad\tx{for }
  u \in \RR,
  x, y \in \RR^h
  \tx{,}
\end{gather*}
which implies~$t_E \ge 0$ as desired, since~$t \ge 0$.

We return to the case of general~$F$.
It remains to verify modular invariance with respect to the slash action in~\eqref{eq:def:jacobi_slash_action_classical}.
It suffices to consider generators of~$\Jac{\traceE(L_m)}(\ZZ)$.
Invariance with respect to the center of~$\Hb{\traceE(L_m)}(\ZZ)$ is clear.
Invariance under translations~$\begin{psmatrix} 1 & b \\ 0 & 1 \end{psmatrix} \in \SL{2}(\cOF)$ and~$(0, \mu, 0) \in \Hb{\traceE(L_m)}(\ZZ)$ follows by inspection of the Fourier series and the map in~\eqref{eq:hermitian_jacobi_upper_half_space_to_hilbert}.

Since~$\SL{2}(\cOF)$ is elementary generated by~\cite[Theorem~4.3.10]{hahn-omeara-1989}, it suffices to further check invariance under~$\begin{psmatrix} 1 & 0 \\ c & 1 \end{psmatrix} \in \SL{2}(\cOF)$ and~$(\la, 0, 0) \in \Hb{\traceE(L_m)}(\ZZ)$.
We inspect the behavior under~$\begin{psmatrix} 0 & -1 \\ 1 & 0 \end{psmatrix} \in \SL{2}(\cOF)$ instead, and combine it with already considered translations to obtain the required invariance.
We inspect the Hilbert--Jacobi slash action~\eqref{eq:def:jacobi_slash_action_classical} of this element and the contribution of the factor~$e(m \tau_2)$ to the slash action of~$\sinv(1) \in Q_h(\cOF)$ on Hermitian Jacobi forms.
This shows that we have to compare~$e(\frac{1}{2} \langle z \rangle_L)$ with~$L = \traceE(L_m)$ in the Hilbert--Jacobi setting with~$e(\lT{z} m w)$ in the Hermitian Jacobi setting.
It suffices to compare them for elementarily tensors in~$\cOE^h \otimes_\ZZ \CC \cong E_\RR \otimes_\RR \CC$.
Recall that, since~$L_m$ is Hermitian, we can discard the reduced trace from~$E$ to~$F$ in the next expression:
\begin{gather*}
  \langle z \otimes c \rangle
  =
  \lT{\ovsmash{z}} m z\,
  c^2
  =
  ( \lT{\ovsmash{z}} \cdot c) m (z \cdot c)
  \tx{,}
\end{gather*}
where the right hand side coincides with~$\lT{z} m w$ after applying~\eqref{eq:hermitian_jacobi_upper_half_space_to_hilbert}.
\end{proof}

The connection between Hermitian Jacobi forms and Hilbert--Jacobi forms allows us to extend two key results from \fref{Section}{sec:hilbert_jacobi_forms}: The dimension estimates and the Hecke bound.

\begin{corollary}%
\label{cor:dimension_bound_hermitian_jacobi_forms}
For~$k \in \ZZ_{> 0}$ and totally positive semi-definite~$m \in \MatD{h}(\cOE)^\vee$, we have
\begin{gather*}
  \dim\,\rmJ^{(1)}_{k,m}
  \ll_{E,h}
  k^{n_F}\,
  \prod_{i = 1}^h \max\{1, \normF(m_{i\!i})^2 \}
  \tx{.}
\end{gather*}
\end{corollary}

\begin{proof}
By \fref{Proposition}{prop:hermitian_jacobi_form_to_hilbert} it suffices to bound the dimension of~$\rmJHilb{k,L}$ with~$L \defeq \traceE(L_m)$.
\fref{Proposition}{prop:dimension_bound_hilbert_jacobi_forms} provides such a dimension bound in terms of~$\#\disc(L_+)$, where~$L = L_0 \oplus L_+$ is a splitting of the radical~$L_0 \subseteq L$.
That is, the statement follows if we show for a suitable~$L_+$ that
\begin{gather*}
  \#\disc(L_+)
  \ll_{E,h}
  \prod_{i = 1}^h \max\bigl\{ 1, \normF(m_{i\!i})^2 \bigr\}
  \tx{.}
\end{gather*}

Let~$m'_+$ be the submatrix of~$m$ formed from entries~$m_{i\!j}$, $i, j \in I$ for a maximal subset of indices~$I \subseteq \{ 1, \ldots, h \}$ such that~$m'_+$ is totally positive definite.
It corresponds to a totally positive sublattice~$L'_+$ of~$L$, and we can choose~$L_+$ to be a superlattice of~$L'_+$.
In particular, we have~$\# \disc(L_+) \le \# \disc(L'_+)$.
We can replace~$m$ by~$m'_+$ and~$L$ by~$L'_+$ in the remainder of the proof, and assume that~$m$ and~$L$ are totally positive definite.
This also allows us to discard the maximum with~$1$ on the right hand side of the stated estimate.
While preserving~$L$ and hence~$\#\disc(L)$ as well as~$\normF(m_{i\!i})$, we can achieve~$m_{i\!i} \asymp_F \traceF(m_{i\!i})$ after multiplying the basis elements of~$L$ with suitable units by \fref{Lemma}{la:unit_balancing} applied to~$U = (\cOF^\times)^2$.

Comparing the definitions of dual lattices
\begin{align*}
  \traceF(L)^\vee
   & =
  \bigl\{
  x \in L_\QQ
  \condsep
  \traceF \bigl( \tfrac{1}{2} \traceE \langle x, y \rangle_L \bigr) \in \ZZ
  \tx{ for all } y \in L
  \bigr\}
  \tx{,}
  \\
  L^\vee
   & =
  \bigl\{
  x \in L_\QQ
  \condsep
  \tfrac{1}{2} \traceE \langle x, y \rangle_L \in \cOF^\vee
  \tx{ for all } y \in L
  \bigr\}
  \tx{,}
\end{align*}
we see that~$\traceF(L)^\vee = L^\vee$.
In particular, we have~$\#\disc(L) \asymp_F \#\disc(\traceF(L))$.
This reduces us to showing that
\begin{gather*}
  \#\disc\bigl( \traceF(L) \bigr)
  \ll_{E,h}
  \prod_{i = 1}^h \normF(m_{i\!i})^2
  \tx{.}
\end{gather*}

Let~$\omega_1, \ldots, \omega_{2n}$ be an integral basis of~$E / \QQ$ with minimal successive~$\normF(\omega_j \ov{\omega_j})$, which only depend on~$E$.
Let~$m'$ be the positive definite, symmetric matrix representing~$\traceF(L)$ with respect to the basis consisting of products of the basis elements of~$L$ with the elements~$\omega_j$.
We have~$\#\disc(\traceF(L)) = \det(m')$.
To estimate~$\det(m')$, we note that the diagonal entries of~$m'$ equal
\begin{gather*}
  m'_{i\!i}
  =
  \traceF\bigl( \tfrac{1}{2} \traceE( \ovsmash{\omega_j} \omega_j m_{i\!i} ) \bigr)
  =
  \traceF( \ovsmash{\omega_j} \omega_j m_{i\!i} )
  \tx{.}
\end{gather*}
We use~$m_{i\!i} \asymp_F \traceF(m_{i\!i})$ in the first and last estimate of the following chain of inequalities:
\begin{gather*}
  m'_{i\!i}
  \ll_E
  \traceF( \ovsmash{\omega_j} \omega_j)\,
  \traceF( m_{i\!i} )
  \asymp_E
  \traceF( m_{i\!i} )
  \ll_E
  \normF(m_{i\!i})^{\frac{1}{n}}
  \tx{.}
\end{gather*}
By Hadamard's inequality~$\det(m')$ is bounded by the product of its diagonal entries.
Each~$m_{i\!i}$ contributes to~$2n$ of them, and thus the above estimate finishes the proof.
\end{proof}

\begin{corollary}%
\label{cor:hecke_bound_hermitian_jacobi_forms}
For~$k \in \ZZ_{>0}$, totally positive semi-definite~$m \in \MatD{h}(\cOE)^\vee$, and $\nu \in \RR_{> 0}$, we have
\begin{gather*}
  \dim\,\rmJ^{(1)}_{k,m} [\nu]
  =
  0
  \quad
  \tx{if\/ }
  k \ll_{E,h} \nu
  \tx{ and }
  \normF(m_{i\!i}) \ll_{E,h} \nu^{n_F}
  \tx{ for all\/ }1 \le i \le h
  \tx{.}
\end{gather*}
\end{corollary}

\begin{proof}
As in the proof of \fref{Corollary}{cor:dimension_bound_hermitian_jacobi_forms}, by \fref{Proposition}{prop:hermitian_jacobi_form_to_hilbert}, it suffices to show that under the given assumptions~$\dim\,\rmJHilb{k,L} [\nu] = 0$ with~$L \defeq \traceE(L_m)$.
We can also reduce ourselves to the case of totally positive definite~$m$ and~$L$ by the same mechanism as in the proof of \fref{Corollary}{cor:dimension_bound_hermitian_jacobi_forms}.

We want to show vanishing of every Hilbert--Jacobi form~$\phi$ of weight~$k \ll_{E,h} \ordnorm(\phi)$ and index~$m$ with~$\normF(m_{i\!i}) \ll_{E,h} \ordnorm(\phi)^{n_F}$.
Our goal is to invoke the Hecke bound in \fref{Theorem}{thm:hilbert_jacobi_forms_hecke_bound}, which requires us to establish the bounds~$k \ll_{F,h} \ord(\phi)$ and~$\la_{2h}(L) \ll_{F,h} \ord(\phi)$.
Using \fref{Lemma}{la:vanishing_order_equivalence_jacobi_form}, we obtain~$k \ll_{E,h} \ordnorm(\phi) \asymp_F \ord(\phi)$.
The successive minimum can be estimated by the same argument as given at the end of the proof of \fref{Corollary}{cor:dimension_bound_hermitian_jacobi_forms}.
With~$m'$ as in that proof, we have
\begin{gather*}
  \la_{2 h}(L)
  \asymp_F
  \la_{2 n_F h}\bigl( \traceF(L) \bigr)
  \le
  \max\bigl\{ m'_{i\!i} \condsep 1 \le i \le 2 n_F h \bigr\}
  \ll_E
  \normF(m_{i\!i})^{1 \slash n_F}
  \tx{.}
\end{gather*}
Combining the assumptions of the corollary with \fref{Lemma}{la:vanishing_order_equivalence_jacobi_form}, the right hand side can be estimated in terms of~$\ord(\phi)$, finishing the proof.
\end{proof}

\subsection{Symmetric formal Fourier--Jacobi series}%
\label{ssec:hermitian_modular_jacobi_forms:symmetric_formal_fourier_jacobi_series}

We assume~$1 \le h < g$ throughout this section.
Fix a weight~$k \in \ZZ$ and an arithmetic type~$\rho$ of level~$N$.
With decomposition~\eqref{eq:HSg_decomposition} in mind, a formal series
\begin{gather}
  \label{eq:def:formal_fourier_jacobi_series}
  f(\tau)
  =
  \sum_{m \in \frac{1}{N} \MatD{h}(\cOE)^\vee}
  \phi_m(\tau_1, z, w)\,
  e(m \tau_2)
  \quad\tx{with }
  \phi_m \in \rmJ^{(g-h)}_{k,m}(\rho)
  \tx{ for all } m
\end{gather}
is called a formal Fourier--Jacobi series of genus~$g$, cogenus~$h$, weight~$k$, and type~$\rho$.
Note that we omit the dependence on~$f$ of the Fourier--Jacobi coefficients~$\phi_m$ from our notation.
Further, by slight abuse of notation, we write~\eqref{eq:def:formal_fourier_jacobi_series} as~$f(\tau)$, but do not mean that~\eqref{eq:def:formal_fourier_jacobi_series} defines a function of~$\tau$.
Based on the Fourier expansion of Hermitian Jacobi forms in~\eqref{eq:def:hermitian_jacobi_forms_fourier_expansion} and the decomposition in~\eqref{eq:hermitian_fourier_index_decomposition} of the Fourier indices~$t$ into matrices~$n,r,m$, the Fourier coefficients of a formal Fourier--Jacobi series~$f$ are defined as
\begin{gather}
  \label{eq:def:formal_fourier_jacobi_series_fourier_expansion}
  c(f;\, t)
  \defeq
  c(\phi_m; n, r)
  \tx{.}
\end{gather}

\begin{definition}%
\label{def:symmetric_formal_fourier_jacobi_series}
We call a formal series as in~\eqref{eq:def:formal_fourier_jacobi_series} \emph{symmetric} if its Fourier coefficients~$c(f; t)$ satisfy relation~\eqref{eq:hermitian_fourier_coefficient_symmetry}:
\begin{multline*}
  c(f; t[a])
  =
  \detF(\ovsmash{a})^k\,
  \rho\bigl( \rot(a) \bigr)^{-1}\,
  c(f; t)
  \\
  \tx{for all }
  t \in \tfrac{1}{N} \MatD{g}(\cOE)^\vee
  \tx{ and all } a \in \GL{g}(\cOE) \tx{ with } \det(a) = \det(\ov{a})
  \tx{.}
\end{multline*}
We call such a series \emph{cuspidal} if all~$\phi_m$ are Hermitian Jacobi cusp forms.
\end{definition}

We let~$\rmFM^{(g,h)}_k(\rho)$ and~$\rmFS^{(g,h)}_k(\rho)$ denote the space of symmetric Fourier--Jacobi series of genus~$g$, cogenus~$h$, weight~$k$, and type~$\rho$, and its subspace of cusp forms.
As in the case of Hermitian modular forms, we may omit~$\rho$ from our notation if it is the trivial representation.
In this case the relation in \fref{Definition}{def:symmetric_formal_fourier_jacobi_series} simplifies to~$c(f; t[a]) = \detF(\ovsmash{a})^k\, c(f; t)$.

Observe that the relation in \fref{Definition}{def:symmetric_formal_fourier_jacobi_series} is compatible with products of formal Fourier--Jacobi series, that is, products of symmetric formal Fourier--Jacobi series are again symmetric.
The graded algebra of symmetric formal Fourier--Jacobi series is
\begin{gather*}
  \rmFM^{(g,h)}_\bullet
  \defeq
  \bigoplus_{k \in \ZZ}
  \rmFM^{(g,h)}_k
  \tx{.}
\end{gather*}

Relation~\eqref{eq:hermitian_fourier_coefficient_symmetry} among Fourier coefficients of Hermitian modular forms via their Fourier--Jacobi expansion yields the inclusion
\begin{gather}
  \label{eq:symmetric_formal_fourier_jacobi_series_from_modular_forms}
  \rmM^{(g)}_k(\rho)
  \lhra
  \rmFM^{(g,h)}_k(\rho)
  \tx{,}\quad
  f
  \lmto
  \sum_{m \in \frac{1}{N} \MatD{h}(\cOE)^\vee}\mspace{-12mu}
  \phi_m(\tau_1, z, w)\, e(m \tau_2)
  \tx{.}
\end{gather}

In order to relate symmetric Fourier--Jacobi expansions of different cogenus to each other we refine the decomposition of~$\tau \in \HSEg$ and Fourier indices~$t \in \frac{1}{N} \MatD{g}(\cOE)^\vee$ given in~\eqref{eq:HSg_decomposition} and~\eqref{eq:hermitian_fourier_index_decomposition} as follows.
For~$1 < h < g$, we decompose~$\tau$ in two different ways as follows:
\begin{alignat}{4}
  \label{eq:def:variable_subdivision}
  \left(
  \begin{array}{c|cc}
      \tau_1 & \lTr{z}{_{11}} & \lTr{z}{_{12}} \\
      \hline
      w_{11} & \tau_{21}      & \lTr{z}{_{2}}  \\
      w_{12} & w_{2}          & \tau_{22}
    \end{array}
  \right)
   & =
  \left(
  \begin{array}{cc}
      \tau_1 & \lT{z} \\ w & \tau_2
    \end{array}
  \right)
   &   & =
  \tau
   &   & =
  \left(
  \begin{array}{cc}
      \tau'_1 & \lTr{z}{^\prime} \\ w' & \tau'_2
    \end{array}
  \right)
   &   & =
  \left(
  \begin{array}{cc|c}
      \tau'_{11} & \lTr{z}{^\prime_1} & \lTr{z}{^\prime_{21}} \\
      w'_{1}     & \tau'_{12}         & \lTr{z}{^\prime_{22}} \\
      \hline
      w'_{21}    & w'_{22}            & \tau'_2
    \end{array}
  \right)
  \tx{,}
  \intertext{%
  where $\tau_1 = \tau'_{11} \in \HS_{g-h}$, $\tau_{21} = \tau'_{12} \in \HS_1$, $\tau_{22} = \tau'_2 \in \HS_{h-1}$, etc.
  Correspondingly, we decompose the Fourier indices as follows:%
  }
  \label{eq:def:fourier_index_subdivision}
  \left(
  \begin{array}{c|cc}
      n   & \lTr{\ovsmash{r}}{_1} & \lTr{\ovsmash{r}}{_2}    \\
      \hline
      r_1 & m_{11}                & \lTr{\ovsmash{m}}{_{21}} \\
      r_2 & m_{21}                & m_{22}
    \end{array}
  \right)
   & =
  \left(
  \begin{array}{cc}
      n & \lT{\ovsmash{r}} \\ r & m
    \end{array}
  \right)
   &   & =
  t
   &   & =
  \left(
  \begin{array}{cc}
      n' & \lTr{\ovsmash{r}}{^\prime} \\ r' & m'
    \end{array}
  \right)
   &   & =
  \left(
  \begin{array}{cc|c}
      n'_{11} & \lTr{\ovsmash{n}}{^\prime_{21}} & \lTr{\ovsmash{r}}{^\prime_1} \\
      n'_{21} & n'_{22}                         & \lTr{\ovsmash{r}}{^\prime_2} \\
      \hline
      r'_1    & r'_2                            & m'
    \end{array}
  \right)
  \tx{,}
\end{alignat}
where $n = n'_{11} \in \frac{1}{N} \MatD{g-h}(\cOE)^\vee$, $m_{11} = n'_{22} \in \frac{1}{N} \cOF^\vee$, $m_{22} = m' \in \frac{1}{N} \MatD{h}(\cOE)^\vee$, etc.
For elements~$\tau_{12}$, $z_{11}$ etc.\@ with two indices the first one references the diagonal block of~\eqref{eq:def:variable_subdivision} in whose subdivision the matrix lies.

It is easy to increase the cogenus of a symmetric formal Fourier--Jacobi series, as we impose fewer conditions.
For a statement that allows us to decrease the cogenus see \fref{Proposition}{prop:reduction_to_lower_cogenus}.

\begin{proposition}%
\label{prop:inclusion_into_higher_cogenus}
Given~$1 \le h' < h < g$ and an arithmetic type~$\rho$, we have inclusions
\begin{gather*}
  \rmFM_\bullet^{(g,h')}(\rho)
  \lhra
  \rmFM_\bullet^{(g,h)}(\rho)
  \quad\tx{and}\quad
  \rmFS_\bullet^{(g,h')}(\rho)
  \lhra
  \rmFS_\bullet^{(g,h)}(\rho)
\end{gather*}
that are the identity on Fourier coefficients.
\end{proposition}

\begin{proof}
The second inclusion follows from the first one and the definition of cuspidal symmetric formal Fourier--Jacobi series via their Fourier expansion support.

To establish the first inclusion, it suffices to consider the case~$h' = h - 1$, and deduce the case of general~$h'$ by applying the resulting inclusion~$h - h'$ times.
Throughout the proof we use notation in~\eqref{eq:def:variable_subdivision} and~\eqref{eq:def:fourier_index_subdivision}.
Given a formal Fourier--Jacobi series~$f'(\tau) = \sum \psi_{m'}(\tau'_1, z', w')\, e(m' \tau'_2)$ of cogenus~$h-1$, we define Hermitian Jacobi forms of cogenus~$h$ by
\begin{gather*}
  \phi_m(\tau_1, z, w)
  =
  \sum_{n, r} c(f'; t)\, e(n \tau_1 + \lT{r} z + \lT{\ovsmash{r}} w)
  \tx{.}
\end{gather*}
This right hand side converges absolutely, since it is a subseries of the absolutely convergent Fourier expansion of~$\psi_{m'}$ by virtue of
\begin{gather*}
  \psi_{m'}(\tau'_1, z', w')
  =
  \sum_{n'_{22}, r'_2}
  \phi_m(\tau_1, z, w)\,
  e\big( n'_{22} \tau'_{12} + \lTr{r}{^\prime_2} z'_{22} + \lTr{\ovsmash{r}}{^\prime_2} w'_{22} \big)
  \tx{.}
\end{gather*}
Further, $\phi_m$ satisfies the modular covariance condition, since~$Q_h(\cOF)$ consists of transformations that preserve this decomposition of~$\psi_{m'}$ term by term.
If~$g-h = 1$, the growth condition for~$\phi_m$ follows from the Koecher Principle in \fref{Theorem}{thm:hermitian_jacobi_forms_koecher_principle}.

The formal Fourier--Jacobi series~$f(\tau) = \sum \phi_m(\tau_1, z, w)\, e(m \tau_2)$ has the same Fourier coefficients as~$f'$ per the definition of~$\phi_m$.
In particular, it is symmetric if and only if~$f'$ is, finishing the proof.
\end{proof}

%% file: sections/03_convergence.tex
\section{Convergence of formal Fourier--Jacobi series}%
\label{sec:convergence}

In this section, we demonstrate the convergence of symmetric formal Fourier--Jacobi series of cogenus~$1$ and trivial type.
As a first step in \fref{Section}{ssec:convergence:algebraic}, we show that they yield an algebraic extension of the graded ring of Hermitian modular forms, which remains true even for arbitrary cogenus.
Then we study their growth at torsion points in \fref{Section}{ssec:convergence:torsion_points}, which we use to deduce a restricted convergence statement in \fref{Section}{ssec:convergence:torsion_point_subvarieties}.
This allows us to invoke Arzelà--Ascoli in the final \fref{Section}{ssec:convergence:meromorphic_and_regular} to show that they yield meromorphic Hermitian modular forms.
In \fref{Section}{ssec:convergence:meromorphic_and_regular}, we also show that they are holomorphic by employing properties of possible polar divisors that we deduce in \fref{Section}{ssec:convergence:torsion_point_divisors} via a fibration argument.

\subsection{Algebraicity over the ring of Hermitian modular forms}%
\label{ssec:convergence:algebraic}

The inclusion~\eqref{eq:symmetric_formal_fourier_jacobi_series_from_modular_forms} of Hermitian modular forms into the graded ring of symmetric formal Fourier--Jacobi series yields an~$\rmM^{(g)}_\bullet$\nbd{}algebra structure on~$\rmFM^{(g,h)}_\bullet$.
The goal of this section is to show that this extension is algebraic.

\begin{theorem}%
\label{thm:symmetric_formal_fourier_jacobi_series_dimension}
For every cogenus~$1 \le h < g$ and weight~$k \in \ZZ$, we have
\begin{gather*}
  \dim\, \rmFM^{(g,h)}_k
  \ll_{E,g}
  k^{n_F g^2}
  \quad\tx{if\/ } k > 0
  \tx{.}
\end{gather*}
\end{theorem}

We defer the proof of \fref{Theorem}{thm:symmetric_formal_fourier_jacobi_series_dimension} and state directly its main consequence.

\begin{corollary}%
\label{cor:symmetric_formal_fourier_jacobi_series_algebraic}
With assumptions as in \fref{Theorem}{thm:symmetric_formal_fourier_jacobi_series_dimension} the extension
\begin{gather*}
  \rmM^{(g)}_\bullet
  \subseteq
  \rmFM^{(g,h)}_\bullet
\end{gather*}
is algebraic.
That is, every symmetric formal Fourier--Jacobi series of scalar weight and trivial arithmetic type of~$\SU{g,g}(\cOF)$ is algebraic over the graded ring of Hermitian modular forms.
\end{corollary}

\begin{proof}
If there were a symmetric formal Fourier--Jacobi series~$f \in \rmFM^{(g,h)}_l$ that were transcendental over Hermitian modular forms, we would have a direct sum submodule
\begin{gather*}
  \bigoplus_{d = 0}^\infty
  f^d\, \rmM^{(g)}_\bullet
  \subseteq
  \rmFM^{(g,h)}_\bullet
  \tx{.}
\end{gather*}

We combine the asymptotic dimension for Hermitian modular forms in \fref{Proposition}{prop:hermitian_modular_forms_dimension_asymptotic} and the asymptotic dimension bound in \fref{Theorem}{thm:symmetric_formal_fourier_jacobi_series_dimension} to arrive at the contradiction
\begin{gather*}
  k^{n_F g^2+1}
  \asymp_g
  \sum_{d = 0}^{k \slash l}
  \big( k - dl \big)^{n_F g^2}
  \asymp_{E,g}
  \dim\,
  \Big(
  \bigoplus_{d = 0}^\infty
  f^d\, \rmM^{(g)}_{k - dl}
  \Big)
  \le
  \dim\,
  \rmFM^{(g,h)}_k
  \ll_{E,g}
  k^{n_F g^2}
\end{gather*}
as~$k \ra \infty$.
\end{proof}

For the proof of \fref{Theorem}{thm:symmetric_formal_fourier_jacobi_series_dimension}, we employ the strict lexicographic ordering, which we denote~$\prec$, on ascending~$h$\nbd{}tuples of non-negative rationals:
\begin{gather*}
  (m_1, \ldots, m_h)
  \prec
  (m'_1, \ldots, m'_h)
  \quad\tx{if and only if}\quad
  \exists 1 \le i \le h
  \quantsep
  \bigl(
  \forall 1 \le j < i
  \quantsep
  m_j = m'_j
  \bigr)
  \wedge
  m_i < m'_i
  \tx{,}
\end{gather*}
and let~$\preceq$ be the associated non-strict ordering.
Under this ordering the set of ascending~$h$\nbd{}tuples of non-negative rationals is an ordered monoid.
We extend~$\prec$ to positive semi-definite~$m \in \MatD{h}(E)$ via the sorted tuple of the norms relative to~$F / \QQ$ of the diagonal entries of~$m$.
We write~$m \sim (m_1,\ldots,m_h)$, if the sorted tuple of these norms agrees with~$(m_1,\ldots,m_h)$.

\begin{lemma}%
\label{la:symmetric_formal_fourier_jacobi_series_index_reduction}
Let~$m \in \MatD{h}(\cOE)^\vee$ be totally positive semi-definite.
Then there is~$u \in \GL{h}(\cOE)$ such that
\begin{gather*}
  \normF\bigl( m[u]_{i\!i} \bigr)
  \le
  \normF\bigl( m[u]_{j\!j} \bigr)
  \quad\tx{and}\quad
  \normF\bigl( m[u]_{i\!j}\, \ov{m[u]_{i\!j}} \bigr)
  \ll_E
  \normF\bigl( m[u]_{i\!i} \bigr)
\end{gather*}
for all~$1 \le i < j \le h$.
\end{lemma}

\begin{proof}
We choose~$u$ in such a way that~$(\normF(m[u])_{11}, \ldots, \normF(m[u])_{h\!h})$ is minimal with respect to the partial order~$\prec$.
Since~$E / F$ is CM, for any~$a \in \cOF \setminus \{0\}$, we can use Minkowski theory first for~$\omega a \cOF$ with some fixed~$\omega \in \cOE \setminus \cOF$ and then for~$a \cOF$ to obtain a constant~$c_E$ independent of~$a$ such that~$\cOE \slash a \cOE$ has a set of representatives whose norms are bounded by~$c_E\, \normF(a)$.
Now if~$\normF(m[u]_{i\!j}\, \ov{m[u]_{i\!j}}) > c_E\, \normF(m[u]_{i\!i})$, then by composing~$u$ with a suitable elementary matrix~$u'$, we can achieve $\normF(m[u u']_{i\!j} \ov{m[u u']_{i\!j}}) \le c_E\, \normF(m[u u']_{i\!i})$.
Comparing determinants of the submatrices composed of the~$i$\thdash{} and~$j$\thdash{} entries of~$m[u]$ and~$m[u u']$, we find that~$\normF(m[u u'])_{j\!j} < \normF(m[u])_{j\!j}$.
This contradicts the choice of~$u$, and thus shows that the bounds stated in the lemma hold for~$m[u]$.
\end{proof}

In this subsection, we refer to indices that satisfy the properties of~$m[u]$ stated in \fref{Lemma}{la:symmetric_formal_fourier_jacobi_series_index_reduction} as reduced.
For non-negative~$m_1 \le \cdots \le m_h \in \QQ_{\ge 0}$, we set
\begin{align*}
  M_{m_1,\ldots, m_h}
  \defeq
  \big\{
  m \in \MatD{h}(\cOE)^\vee
  \condsep
  m \sim (m_1,\ldots,m_h)
  \tx{ totally positive semi-definite and reduced}
  \big\}
  \tx{.}
\end{align*}
The next lemma bounds the number of relevant indices in a symmetric formal Fourier--Jacobi series.
The main role of the reduction condition is to allow for the exponent~$n_F (h-1) h$ as opposed to~$2 n_F (h-1) h$.

\begin{lemma}%
\label{la:symmetric_formal_fourier_jacobi_series_index_set}
We have the bound
\begin{gather*}
  \# M_{m_1,\ldots,m_h}
  \ll_E
  \prod_{i = 1}^h
  \max \{ 1, m_i \}^{2 n_F (h-i)}
  \le
  \max \{ 1, m_h \}^{n_F (h-1) h}
  \tx{.}
\end{gather*}
\end{lemma}

\begin{proof}
Given~$1 \le i < j \le h$, we estimate the number of possible entries in position~$(i,j)$ of elements of~$m \in M_{m_1,\ldots,m_h}$.
The reduction condition yields~$\normF(m_{i\!j} \ov{m_{i\!j}}) \ll_E \normF(m_{i\!i}) = m_i$.
We conclude by Minkowski theory that
\begin{gather*}
  \# \big\{
  a \in \cOE^\vee
  \condsep
  \tx{there is } m \in M_{m_1,\ldots,m_h}
  \tx{ such that } a = m_{i\!j}
  \big\}
  \ll_E
  \max\{ 1, m_i \}^{2 n_F}
  \tx{.}
\end{gather*}
Since~$M_{m_1,\ldots,m_h}$ consists of Hermitian matrices, we have~$m_{ji} = \ov{m_{i\!j}}$.
In summary we have
\begin{gather*}
  \# M_{m_1,\ldots,m_h}
  \ll_E
  \prod_{i = 1}^{h-1}
  \prod_{j = i+1}^h
  \max\{ 1, m_i \}^{2 n_F}
  \tx{,}
\end{gather*}
which yields the lemma after simplification.
\end{proof}

\begin{proof}%
[Proof of \fref{Theorem}{thm:symmetric_formal_fourier_jacobi_series_dimension}]
The inclusion in \fref{Proposition}{prop:inclusion_into_higher_cogenus} allows us to restrict to the case~$h = g - 1$.

For clarity, we recall that the ordering~$\prec$ on Hermitian matrices is defined using the~\emph{sorted} tuple of the norms of their diagonal entries.
For the purpose of this proof, we define~$D_F > 0$ by~$D_F \ZZ = \normF(\cDF)$.
Throughout the proof, we restrict to non-negative~$m_1 \le \cdots \le m_h$ with~$m_i \in D_F^{-1} \ZZ$, that is, we assume that the denominators of~$m_i$ are bounded.

We consider
\begin{gather*}
  \rmFM^{(g,h)}_k [m_1, \ldots, m_h]
  \defeq
  \Big\{
  \sum_m \phi_m(\tau_1, z, w)\, e(m \tau_2)
  \in
  \rmFM^{(g,h)}_k
  \condsep
  \phi_m = 0 \tx{ for all } m
  \tx{ with } m \prec (m_1,\ldots,m_h)
  \Big\}
  \tx{.}
\end{gather*}
We have~$\rmFM^{(g,h)}_k = \rmFM^{(g,h)}_k[0,\ldots,0]$, since the index~$m$ of any nonzero coefficient~$\phi_m$ is positive semi-definite by \fref{Theorem}{thm:hermitian_jacobi_forms_koecher_principle}.
The inclusion
\begin{gather*}
  \rmFM^{(g,h)}_k [m_1, \ldots, m_h] \supseteq \rmFM^{(g,h)}_k [m'_1, \ldots, m'_h]
  \quad\tx{if}\quad
  (m_1, \ldots, m_h) \prec (m'_1, \ldots, m'_h)
\end{gather*}
follows directly from the defining vanishing condition on the Fourier--Jacobi coefficients~$\phi_m$ and the fact that~$\prec$ is a strict partial order.
Finally, we have
\begin{gather}
  \label{eq:prf:thm:symmetric_formal_fourier_jacobi_series_dimension:filtration_intersection}
  \bigcap_{0 \le m_1 \le \cdots \le m_h}
  \mspace{-18mu}
  \rmFM^{(g,h)}_k [m_1, \ldots, m_h]
  =
  \{0\}
  \tx{,}
\end{gather}
since any nonzero element~$f \in \rmFM^{(g,h)}_k$ has a nonzero Fourier--Jacobi coefficient~$\phi_m$ for some~$m$.
Then choosing~$(m'_1, \ldots, m'_h)$ the sorted tuple of~$\normF(m_{i\!i}) + 1$, $1 \le i \le h$, we find that
\begin{gather*}
  f
  \not\in
  \rmFM^{(g,h)}_k [m'_1, \ldots, m'_h]
  \supset
  \bigcap_{0 \le m_1 \le \cdots \le m_h}
  \mspace{-18mu}
  \rmFM^{(g,h)}_k [m_1, \ldots, m_h]
  \tx{.}
\end{gather*}

For brevity, we now write~$\ul{m} = (m_1, \ldots, m_h)$, and~$\ul{0} = (0, \ldots, 0)$.
We have canonical projection maps
\begin{gather*}
  \pi_{\ulsmash{m}}
  \defcol
  \rmFM^{(g,h)}_k[ \ul{m} ]
  \lthra
  \rmFM^{(g,h)}_k[ \ul{m} ]
  \Bigslash
  \bigcup_{\ul{m} \prec \ulsmash{m}'}
  \rmFM^{(g,h)}_k[ \ulsmash{m}' ]
  \tx{,}
\end{gather*}
where the union on the right hand side is a subspace, since~$\prec$ is a total ordering.
We let~$\pi_{\ulsmash{m}}^{-1}$ be a fixed linear right inverse, i.e.\@ a section.
We note that~$P(\ul{m}) \defeq \{ \ulsmash{m}' \condsep \ulsmash{m}' \preceq \ul{m} \}$ is finite for every~$\ul{m}$, since the entries of~$\ulsmash{m}'$ are nonnegative with bounded denominator.
For a symmetric formal Fourier--Jacobi series~$f$, we set~$f_{\ul{m}, \ulsmash{0}} \defeq f$ and if~$\ulsmash{m}' \ne 0$ then
\begin{gather*}
  f_{\ulsmash{m}, \ulsmash{m}'}
  \defeq
  f_{\ulsmash{m}, \ulsmash{m}''}
  -
  \pi_{\ulsmash{m}''}^{-1}\bigl(
  \pi_{\ulsmash{m}''}\bigl(
  f_{\ulsmash{m}, \ulsmash{m}''}
  \bigr) \bigr)
  \in
  \rmFM^{(g,h)}_k[ \ulsmash{m}' ]
  \tx{,}
\end{gather*}
where~$\ulsmash{m}''$ is the predecessor of~$\ulsmash{m}'$ in~$P(\ul{m})$.
The resulting~$f_{\ulsmash{m}} \defeq f_{\ulsmash{m}, \ulsmash{m}}$ depends linearly on~$f$.
We conclude from~\eqref{eq:prf:thm:symmetric_formal_fourier_jacobi_series_dimension:filtration_intersection} that we have an injective linear map
\begin{gather*}
  \rmFM^{(g,h)}_k
  \lhra
  \bigoplus_{0 \le m_1 \le \cdots \le m_h}
  \rmFM^{(g,h)}_k [\ul{m}]
  \Bigslash
  \bigcup_{\ul{m} \prec \ulsmash{m}'}
  \rmFM^{(g,h)}_k [\ulsmash{m}']
  \tx{,}\quad
  f
  \lmto
  \bigl(
  \pi_{\ul{m}}( f_{\ul{m}} )
  \bigr)_{0 \le m_1 \le \cdots \le m_h}
  \tx{.}
\end{gather*}
We will employ the resulting dimension estimate
\begin{gather}
  \label{eq:prf:thm:symmetric_formal_fourier_jacobi_series_dimension:filtration_dimension_estimate}
  \dim\, \rmFM^{(g,h)}_k
  \le
  \sum_{0 \le m_1 \le \cdots \le m_h}
  \dim\,\Bigg(
  \rmFM^{(g,h)}_k [\ul{m}]
  \Bigslash
  \bigcup_{\ul{m} \prec \ulsmash{m}'}
  \rmFM^{(g,h)}_k [\ulsmash{m}']
  \Bigg)
  \tx{.}
\end{gather}

Given~$0 \le m_1 \le \cdots \le m_h$ as before, we have a linear map
\begin{gather*}
  \rmFM^{(g,h)}_k [m_1,\ldots,m_h]
  \lra
  \bigoplus_{m \in M_{m_1,\ldots,m_h}}
  \mspace{-18mu}
  \rmJ^{(1)}_{k,m}
  \tx{,}\quad
  f(\tau)
  =
  \sum_m \phi_m(\tau_1, z, w)\, e(m \tau_2)
  \lmto
  \big( \phi_m \big)_{m \in M_{m_1,\ldots,m_h}}
  \tx{.}
\end{gather*}
We consider a Fourier coefficient in the image with indices~$n$, $r$, and~$m$, where~$\normF(n) < m_h$.
Using notation from~\eqref{eq:hermitian_fourier_index_decomposition} we write~$t$ for the index composed of them.
By definition of the Fourier coefficients of~$f$ and the symmetry condition~\eqref{eq:hermitian_fourier_coefficient_symmetry} imposed in \fref{Definition}{def:symmetric_formal_fourier_jacobi_series}, we have
\begin{gather*}
  c(\phi_m; n, r)
  =
  c(f, t)
  =
  (-1)^k\,
  c(f; t')
  =
  (-1)^k\,
  c(\phi_{m'}; n', r')
  \tx{,}
\end{gather*}
where~$t' = t[s]$ with block decomposition as in~\eqref{eq:hermitian_fourier_index_decomposition} and~$s \in \GL{g}(\cOE)$ is the permutation matrix that swaps the first and last row.
We have~$m'_{h\!h} = n$ and since~$\normF(n) < m_h$, we conclude~$m' \prec m$, hence~$\phi_{m'} = 0$ and~$c(\phi_m; n, r) = 0$.
This proves that we have a map with smaller codomain
\begin{gather}
  \label{eq:prf:thm:symmetric_formal_fourier_jacobi_series_dimension:filtration_to_jacobi_forms}
  \rmFM^{(g,h)}_k [m_1,\ldots,m_h]
  \lra
  \bigoplus_{m \in M_{m_1,\ldots,m_h}}
  \mspace{-18mu}
  \rmJ^{(1)}_{k,m}[m_h^{1 \slash n_F}]
  \tx{,}\quad
  \sum_m \phi_m(\tau_1, z, w)\, e(m \tau_2)
  \lmto
  \big( \phi_m \big)_{m \in M_{m_1,\ldots,m_h}}
  \tx{,}
\end{gather}
where the codomain consists of spaces of Hermitian Jacobi forms with prescribed vanishing order, defined in \fref{Section}{ssec:hermitian_modular_jacobi_forms:jacobi_forms}.

To determine the kernel of~\eqref{eq:prf:thm:symmetric_formal_fourier_jacobi_series_dimension:filtration_to_jacobi_forms}, we note that~$\ul{m} = (m_1,\ldots,m_h)$ is the unique maximum element of the set~$\{ \ulsmash{m}'' \condsep \ulsmash{m}'' \prec \ulsmash{m}' \tx{ for all } \ulsmash{m} \prec \ulsmash{m}' \}$.
Combining this with the definition of the spaces~$\rmFM^{(g,h)}_k [\ul{m}]$, we conclude that the kernel of~\eqref{eq:prf:thm:symmetric_formal_fourier_jacobi_series_dimension:filtration_to_jacobi_forms} equals
\begin{gather*}
  \bigcup_{\ul{m} \prec \ulsmash{m}'}
  \rmFM^{(g,h)}_k [\ulsmash{m}']
  \tx{.}
\end{gather*}
In other words, we can use the dimension of the image of~\eqref{eq:prf:thm:symmetric_formal_fourier_jacobi_series_dimension:filtration_to_jacobi_forms} to estimate the terms on the right hand side of~\eqref{eq:prf:thm:symmetric_formal_fourier_jacobi_series_dimension:filtration_dimension_estimate}:
\begin{gather}
  \label{eq:prf:thm:symmetric_formal_fourier_jacobi_series_dimension:jacobi_dimension_estimate}
  \dim\, \rmFM^{(g,h)}_k
  \le
  \sum_{\substack{
      0 \le m_1 \le \cdots \le m_h \\
      m \in M_{m_1,\ldots,m_h}
    }}
  \mspace{-18mu}
  \dim\, \rmJ^{(1)}_{k,m} [m_h^{1 \slash n_F}]
  \tx{.}
\end{gather}

We next use \fref{Corollary}{cor:hecke_bound_hermitian_jacobi_forms} to show that the terms on the right hand side of~\eqref{eq:prf:thm:symmetric_formal_fourier_jacobi_series_dimension:jacobi_dimension_estimate} vanish if we have~$k \ll_{E,g} m_h$.
The remaining terms satisfy~$m_h \ll_{E,g} k$ and can be estimated using the dimension bound in \fref{Corollary}{cor:dimension_bound_hermitian_jacobi_forms}, which yields
\begin{gather*}
  \dim\, \rmFM_k^{(g,h)}
  \ll_{E,g}
  \sum_{0 \le m_1 \le \cdots \le m_h \ll_{E,g} k}
  \mspace{-32mu}
  \# M_{m_1,\ldots,m_h}
  \cdot
  k^{n_F}\, \prod_{i = 1}^h \max\{1, m_i\}^{2 n_F}
  \tx{.}
\end{gather*}
We insert the estimate for~$M_{m_1,\ldots,m_h}$ in \fref{Lemma}{la:symmetric_formal_fourier_jacobi_series_index_set} and replace~$m_i$ by~$m_h$ in the product on the right hand side, and then replace the summation range by~$m_1, \ldots, m_h \ll_{E,g} k$ to obtain
\begin{align*}
  \dim\, \rmFM_k^{(g,h)}
   & \ll_{E,g}
  \sum_{0 \le m_1, \ldots, m_h \ll_{E,g} k}
  \mspace{-32mu}
  \max\{1, m_h\}^{n_F (h-1) h}
  \cdot
  k^{n_F}\, \max\{1, m_h\}^{2 n_F h}
  =
  k^{n_F (h + (h-1) h + 1 + 2 h)}
  =
  k^{n_F g^2}
  \tx{.}
\end{align*}
\end{proof}

\subsection{Growth at torsion points}%
\label{ssec:convergence:torsion_points}

In this section, we show that cuspidal symmetric formal Fourier--Jacobi series specialized to torsion points yield coefficients of moderate growth.
In contrast to the previous subsection, the weight~$k$ will be fixed throughout the section.
We preface the main argument with several standard results.

We use two norms in the proof of \fref{Theorem}{thm:torsion_point_hecke_bound}.
The $\infty$\nbd{}norm of~$\td{f} \in \rmS^{(g-1)}_k(N)$ is defined as
\begin{gather*}
  \| \td{f} \|_\infty
  \defeq
  \sup_{\tau_1 \in \HS_{E,g-1}}
  \detF\bigl( \Im(\tau_1) \bigr)^{\frac{k}{2}}\,
  \big| \td{f}(\tau_1) \big|
  \tx{.}
\end{gather*}

\begin{lemma}%
\label{la:hermitian_modular_forms_infty_norm_bounded}
Given~$\td{f} \in \rmS^{(g-1)}_k(N)$ for~$g \ge 2$, we have~$\| \td{f} \|_\infty < \infty$.
\end{lemma}

\begin{proof}
For convenience, we write~$\tau_1 = x_1 + i y_1$ for Hermitian matrices~$x_1 = \Re(\tau_1)$ and~$y_1 = \Im(\tau_1)$.
We set
\begin{gather*}
  \hat{f}(\tau_1)
  \defeq
  \detF( y_1 )^{\frac{k}{2}}\,
  \big| \td{f}(\tau_1) \big|
  \tx{,}
\end{gather*}
which is invariant under the action of~$\SU{g-1,g-1}(\cOF,N)$.
By Borel--Harish-Chandra~\cite{borel-harish-chandra-1962} a fundamental domain for~$\SU{g-1,g-1}(\cOF,N)$ is contained in a suitable finite union of~$\ga\, \HS_{E,g-1}(K,\eps)$, where
\begin{gather*}
  \HS_{E,g-1}(K,\eps)
  \defeq
  \big\{
  \tau_1 \in \HS_{E,g-1}
  \condsep
  x_1 \in K,
  y_1 > \epsilon
  \big\}
\end{gather*}
for~$\ga \in \SU{g-1,g-1}(F)$, $K \subset \MatT{g-1}(F_\RR)$ compact, and~$\eps > 0$.
Now by~\cite{koecher-1960}, this fundamental domain is contained in a finite union of~$\ga\, \HS_{E,g-1}'(K,\eps)$, where the additional condition~$y_1' \ll_{E,g} y_1 \ll_{E,g} y_1'$ on the diagonal~$y_1'$ of~$y_1$ is imposed.

Thus it suffices to bound~$\hat{f} \circ \ga$ on~$\HS_{E,g-1}'(K,\eps)$.
Given~$\delta > 0$, using
\begin{gather*}
  \detF(\Im(\tau_1))
  \ll_{E,g}
  \detF(y_1')
  \ll_{E,g,\delta}
  e\bigl(- i y_1'\, \mfrac{2 \delta}{k} \bigr)
\end{gather*}
for~$\tau_1 \in \HS_{E,g-1}'(K,\eps)$, we find
\begin{gather*}
  \hat{f}(\tau_1)
  \ll_{E,g,\eps}
  \sum_{n \in \MatD{g-1}(\cOE)^\vee}
  |c(\td{f};n)|\,
  e\bigl( n i ( y_1' - \delta) \bigr)
  \tx{,}
\end{gather*}
which is bounded if~$\delta \ll_{E,g} \eps$, since~$y_1' \gg_{E,g} \eps$ and~$\td{f}$ has an absolutely convergent Fourier series on~$\HS_{E,g-1}$.
\end{proof}

We introduce a further norm derived from Fourier coefficients, for which we need the vanishing bound in the next statement.

\begin{lemma}%
\label{la:hermitian_modular_forms_diagonal_hecke_bound}
For~$g \ge 2$ and~$N \in \ZZ_{\ge 1}$, there is\/~$\nu_{E,g,k}(N) > 0$ such that for~$\td{f} \in \rmS^{(g-1)}_k(N)$, we have~$\td{f} = 0$ if\/~$c(\td{f}; n) = 0$ for all~$n \in B(N)$ with
\begin{multline*}
  B(N)
  =
  B_{E,g,k}(N)
  \defeq
  \big\{
  n \in \tfrac{1}{N}\MatD{g-1}(\cOE)^\vee
  \condsep
  n \tx{ totally positive definite},
  \\
  \traceF(n_{i\!i}) < \nu_{E,g,k}(N) \tx{ for all } 1 \le i \le g-1
  \big\}
  \tx{.}
\end{multline*}
\end{lemma}

\begin{proof}
Since~$\rmS^{(g-1)}_k$ is finite dimensional by \fref{Proposition}{prop:hermitian_modular_forms_dimension_asymptotic}, the desired constant exists for~$N = 1$.
We observe that by a naive estimate using~$\# (\cOE \slash N \cOE) = N^{2 n_F}$, we have
\begin{gather*}
  \#
  \Bigl(
  \SU{g-1,g-1}(\cOF, N) \big\backslash \SU{g-1,g-1}(\cOF)
  \Bigr)
  \le
  N^{8 n_F (g-1)^2}
  \tx{.}
\end{gather*}
Since~$\traceF(n_{i\!i})$ is linear in~$n_{i\!i}$, the statement for~$N > 1$ follows by a standard argument after applying the case of~$N = 1$ to the form~$\prod_\ga f|_k\, \ga$, where~$\ga$ runs through representatives of the quotient~$\SU{g-1,g-1}(\cOF, N) \backslash \SU{g-1,g-1}(\cOF)$; see the proof of \fref{Lemma}{la:accumulated_hecke_bound_twisted_permutation_type} for a similar argument.
\end{proof}

We fix~$\nu_{E,g,k}(N)$ as in \fref{Lemma}{la:hermitian_modular_forms_diagonal_hecke_bound} once and for all and continue to use the notation~$B(N)$.
Note that~$B(N)$ is a finite set.
Thus we can define the norm
\begin{gather*}
  \| \td{f} \|_{\rmF\rmE}
  \defeq
  \sum_{n \in B(N)} \big| c(\td{f};n) \big|
  \tx{.}
\end{gather*}

We also need the Hecke bound for Hermitian modular forms.

\begin{lemma}%
\label{la:hermitian_modular_forms_hecke_bound}
Given~$\td{f} \in \rmS^{(g-1)}_k(N)$ for~$g \ge 2$, we have the bound~$c(\td{f}; n) \ll_{\td{f}} \detF(n)^{\frac{k}{2}}$.
\end{lemma}

\begin{proof}
Siegel's original argument~\cite{siegel-1939} applies:
Since~$\td{f}$ is a cusp form, we may assume that~$n$ is totally positive definite.
As before, we write~$\tau_1 = x_1 + i y_1$ with~$x_1 = \Re(\tau_1)$ and~$y_1 = \Im(\tau_1)$.
We express the $n$\thdash{} Fourier coefficient of~$\td{f}$ as the integral
\begin{gather*}
  c(\td{f}; n)
  \asymp_N
  \int_{N \MatD{g-1}(\cOE) \backslash \MatD{g-1}(E)}
  \td{f}(x_1 + i y_1)\,
  e(-n \tau_1)
  \,\rmdx_1
  \tx{,}
\end{gather*}
whose absolute value we can estimate using \fref{Lemma}{la:hermitian_modular_forms_infty_norm_bounded} by
\begin{gather*}
  \detF(y_1)^{-\frac{k}{2}}
  e(- n i y_1)
  \int_{N \MatD{g-1}(\cOE) \backslash \MatD{g-1}(E)}
  \detF(y_1)^{\frac{k}{2}}
  \bigl| \td{f}(x_1 + i y_1) \bigr|
  \rmdx_1
  \ll_{\td{f}}
  \detF(y_1)^{-\frac{k}{2}}
  e(- n i y_1)
  \tx{.}
\end{gather*}
Inserting~$n^{-1}$ for~$y_1$ we obtain the statement.
\end{proof}

\begin{corollary}%
\label{cor:hermitian_jacobi_forms_hecke_bound}
Given a cogenus~$1$ cusp form~$\phi \in \rmJ^{(g-1)}_{k,m}$ for~$g \ge 2$, we have the bound~$c(\phi; n, r) \ll_{\phi} \detF(t)^{\smash{\frac{k-1}{2}}}$, where~$t$ is given by the decomposition in~\eqref{eq:hermitian_fourier_index_decomposition}.
\end{corollary}

\begin{proof}
Hermitian Jacobi theta series in~\eqref{eq:def:hermitian_jacobi_theta} yield the theta decomposition
\begin{gather*}
  \phi(\tau_1, z, w)
  =
  \sum_{\mu}
  f_\mu(\tau_1)\,
  \theta_{m,\mu}(\tau_1, z, w)
  \tx{,}
\end{gather*}
where~$\mu$ runs through the coset in~\eqref{eq:def:hermitian_jacobi_theta}.
In this decomposition all~$f_\mu$ are Hermitian cusp forms of weight~$k-1$ with Fourier coefficients~$c(f_\mu; n') = c(\phi; t)$ for suitable~$t$ that satisfies~$\detF(n')\, \det(m) = \detF(t)$.
Now the statement follows from \fref{Lemma}{la:hermitian_modular_forms_hecke_bound}.
\end{proof}

We are now in position to bound symmetric formal Fourier--Jacobi series at admissible torsion points.
Recall the specialization of Hermitian Jacobi forms to torsion points in~\eqref{eq:def:hermitian_jacobi_forms_torsion_points} and the notion of admissible torsion points in \fref{Definition}{def:hermitian_jacobi_forms_admissible_torsion_point}.

\begin{theorem}%
\label{thm:torsion_point_hecke_bound}
Given~$f = \sum_m \phi_m(\tau_1,z,w)\, e(m \tau_2) \in \rmFS^{(g,1)}_k$ and an admissible~$N$\nbd{}torsion point~$(\alpha,\beta)$, we have the bound
\begin{gather*}
  \big\|
  \phi_m [\alpha, \beta]
  \big\|_\infty
  \ll_{f}
  \normF(m)^{\frac{k-3}{2} + g}
  \quad\tx{for all totally positive } m \in \cOF^\vee
  \tx{.}
\end{gather*}
\end{theorem}

\begin{proof}
By \fref{Proposition}{prop:hermitian_jacobi_forms_torsion_points_modularity}, we have~$\phi[\alpha,\beta] \in \rmS^{(g-1)}_k(N_F N^2)$, where~$N_F^{-1} \ZZ = \traceF( \frac{1}{2} \cOF^\vee )$.
Since the space~$\rmS^{(g-1)}_k(N_F N^2)$ is finite dimensional, norm comparison shows that
\begin{gather*}
  \big\| \phi_m[\alpha,\beta] \big\|_\infty
  \ll_{E,g,k,N}
  \big\| \phi_m[\alpha,\beta] \big\|_{\rmF\rmE}
  \tx{.}
\end{gather*}
In the proof, we will bound the right hand side.

Employgin the symmetry of Fourier coefficients of~$f$ and using \fref{Lemma}{la:unit_balancing} with~$U = (\cOF^\times)^2$, we may replace~$m$ by~$\eps^2 m$ for a suitable~$\eps \in \cOF^\times$ and assume that~$m \asymp_F \traceF(m)$.
Expanding the definition of specializations to torsion points in~\eqref{eq:def:hermitian_jacobi_forms_torsion_points}, we find that
\begin{gather*}
  \big| c(\phi_m[\alpha,\beta]; n) \big|
  \le
  \sum_{r \in \frac{1}{N_F N^2} \Mat{1,g-1}(\cOE^\vee)}
  \Big|
  c\Bigl( f; t\Bigl[ \begin{psmatrix} 1 & 0 \\ -\alpha & 1 \end{psmatrix} \Bigr] \Bigr)
  \Big|
  \quad\tx{with }
  t
  =
  \begin{psmatrix} n & \lT{\ovsmash{r}} \\ r & m \end{psmatrix}
  \tx{.}
\end{gather*}
Since~$f$ and hence~$\phi_m$ is cuspidal, we can restrict the sum to~$r$ for which~$t$ is totally positive definite.
In particular, we have the component-wise bound~$\traceF(r_i \ov{r_i}) < \traceF( n_{i\!i} m )$ arising from the~$(i,g)$\nbd{}minor of~$t$.
Since~$m \asymp_F \traceF(m)$, there is a constant~$c_F$ such that~$\traceF( n_{i\!i} m ) \le c_F\, \traceF( n_{i\!i} ) \traceF(m)$.
Since the traces of the diagonal entries~$n_{i\!i}$ of~$n \in B(N_F N^2)$ are bounded in terms of~$\nu_{E,g,k}(N)$, we can restrict the sum to~$r$ contained in
\begin{multline*}
  R(N_F N^2, m)
  \defeq
  \big\{
  r \in \tfrac{1}{N_F N^2} \Mat{1,g-1}(\cOE^\vee)
  \condsep
  \\
  \traceF(r_i \ov{r_i}) \le c_F \nu_{E,g,k}(N_F N^2)\, \traceF(m)
  \tx{ for all } 1 \le i \le g-1
  \big\}
  \tx{.}
\end{multline*}

Summarizing, we have the estimate
\begin{gather*}
  \big\| \phi_m[\alpha,\beta] \big\|_{\rmF\rmE}
  \ll_{E,g,N}
  \sum_{\substack{
      n \in B(N_F N^2) \\
      r \in R(N_F N^2, m)
    }}
  \Big| c\bigl(f; t\bigl[ \begin{psmatrix} 1 & 0 \\ -\alpha & 1 \end{psmatrix} \bigr]\bigr) \Big|
  \quad\tx{with }
  t
  =
  \begin{psmatrix} n & \lT{\ovsmash{r}} \\ r & m \end{psmatrix}
  \tx{.}
\end{gather*}
The size of~$B(N_F N^2)$ by inspection of its definition is bounded as~$\# B(N_F N^2) \ll_{E,g,k,N} 1$.
We bound the size of~$R(N_F N^2, m)$ using Minkowski theory, which asserts that~$\cOE \subset E_\RR$ is a $2n_F$\nbd{}dimensional lattice, to conclude that
\begin{gather*}
  \# R(N_F N^2, m)
  \ll_{E,g,k,N}
  \traceF(m)^{n_F (g-1)}
  \asymp_{E,g,k,N}
  \normF(m)^{g-1}
  \tx{.}
\end{gather*}
This yields
\begin{gather}
  \label{eq:prf:thm:torsion_point_hecke_bound:reduction_to_coefficient_bound}
  \big\| \phi_m[\alpha,\beta] \big\|_{\rmF\rmE}
  \ll_{E,g,k,N}
  \normF(m)^{n_F (g-1)}\,
  \max_{\substack{
      n \in B(N_F N^2) \\
      r \in R(N_F N^2, m) \\
    }}
  \Big| c\bigl(f; t\bigl[ \begin{psmatrix} 1 & 0 \\ -\alpha & 1 \end{psmatrix} \bigr]\bigr) \Big|
  \tx{.}
\end{gather}
Note that the Fourier coefficient on the right hand side vanishes if~$t \not> 0$.
In the remainder of the proof, we fix~$n \in B(N_F N^2)$ and~$r \in R(N_F N^2, m)$ such that~$t$ is positive definite, and estimate that coefficient.

By our assumptions the torsion point~$(\alpha,\beta)$ is admissible, that is, the fractional ideal generated by~$1$ and the entries of~$\alpha$ equals~$M^{-1} \cOE$ for a positive integer~$M$ (see \fref{Definition}{def:hermitian_jacobi_forms_admissible_torsion_point}).
We write~$s \in \GL{g}(\cOE)$ for the permutation matrix swapping the first and last coordinate.
Since~$\alpha$ is part of an admissible torsion point, there is~$u \in \GL{g}(\cOE)$ such that
\begin{gather*}
  \begin{psmatrix} 1 & 0 \\ -\alpha & 1 \end{psmatrix} u s
  =
  \begin{psmatrix} \rho & \xi \\ 0 & M^{-1} \end{psmatrix}
  \quad
  \tx{with }
  \rho \in \Mat{g-1}(\cOE) \cap \GL{g-1}(E),\,
  \xi \in \cOE^{g-1}
  \tx{.}
\end{gather*}
We may choose~$u$ in such a way that~$n[\rho]$ is reduced in the sense of Koecher~\cite[\S12]{koecher-1960}.
Comparing determinants on both sides, we see that~$|\det(\rho)| = M$.
The symmetry of~$f$ guarantees that
\begin{gather*}
  \Big| c\bigl(f; t\bigl[ \begin{psmatrix} 1 & 0 \\ -\alpha & 1 \end{psmatrix} \bigr]\bigr) \Big|
  =
  \Big| c\bigl(f; t\bigl[ \begin{psmatrix} 1 & 0 \\ -\alpha & 1 \end{psmatrix} u \bigr]\bigr) \Big|
  =
  \big|
  c\big( f; t' \big)
  \big|
  =
  \big|
  c\big( \phi_{m'}; n', r' \big)
  \big|
  \quad\tx{with }
  t'
  =
  \begin{psmatrix} n' & \lTr{\ovsmash{r}}{^\prime} \\ r' & m' \end{psmatrix}
  =
  t \Bigl[ \begin{psmatrix} \rho & \xi \\ 0 & M^{-1} \end{psmatrix} s \Bigr]
  \tx{.}
\end{gather*}

The bottom right entry~$m'$ of~$t'$ equals the top left entry~$(n[\rho])_{1\!1}$ of the reduced Hermitian matrix~$n[\rho]$.
Koecher's reduction theory yields an analogue of the Hermite bound at every place of~$E$.
This, together with our bound for the diagonal entries of~$t$, allows us to estimate
\begin{gather*}
  \normF\bigl(
  \big( n[\rho] \big)_{1\!1}
  \bigr)
  \ll_{E,g}
  \bigl( \detF(\rho)^2\, \detF(n) \bigr)^{\frac{1}{g-1}}
  \ll_{E,g,k,N}
  1
  \tx{.}
\end{gather*}
Additionally, by Koecher's reduction theory we have~$( n[\rho] )_{1\!1} \asymp_F \traceF( ( n[\rho] )_{1\!1} )$.
In particular, we can estimate~$c(f; t)$ in terms of finitely many $\phi_{m'}$, $m' = ( n[\rho] )_{1\!1}$.
Using the Hecke bound in \fref{Corollary}{cor:hermitian_jacobi_forms_hecke_bound}, we obtain that
\begin{align*}
         &
  \big| c(\phi_{m'}; n', r') \big|
  \ll_{f}
  \detF(t')^{\frac{k-1}{2}}
  =
  \detF(t)^{\frac{k-1}{2}}
  \\
  \le {} &
  \big( \detF(n) \normF(m) \big)^{\frac{k-1}{2}}
  \ll_{E,g,N}
  \normF(m)^{\frac{k-1}{2}}
  \tx{,}
\end{align*}
since the norms of the diagonal entries of~$n \in B(N_F N^2)$ are bounded independently of~$m$.
To obtain the bound stated in the theorem, we combine this bound with the one in~\eqref{eq:prf:thm:torsion_point_hecke_bound:reduction_to_coefficient_bound}.
\end{proof}

\subsection{Convergence on torsion points}%
\label{ssec:convergence:torsion_point_subvarieties}

In this section we leverage \fref{Theorem}{thm:torsion_point_hecke_bound} to inspect the behavior of symmetric formal Fourier--Jacobi expansions on subspaces~$\HS_{E,g-1} \times \HS_F \hra \HSEg$ associated with admissible torsion points.
Recall the subdivision of~$\tau$ in~\eqref{eq:HSg_decomposition}.
We set
\begin{gather}
  \label{eq:torsion_point_subvarieties}
  \HSEg[\alpha,\beta]
  \defeq
  \big\{
  \tau \in \HSEg
  \condsep
  z = \alpha\, \lTr{\tau}{_1} + \beta,\,
  w = \ov\alpha \, \tau_1 + \ov\beta
  \big\}
  \tx{.}
\end{gather}
This definition is adjusted to the notion of torsion points in \fref{Definition}{def:hermitian_jacobi_forms_admissible_torsion_point}.

\begin{lemma}%
\label{la:torsion_point_subvarieties_positive_definite_condition}
Given a torsion point~$(\alpha,\beta)$, we have
\begin{gather*}
  \Im(\tau_2)
  >
  \Im(\tau_1)[\lT{\alpha}]
  \quad
  \tx{for all\/ }
  \tau \in \HSEg[\alpha,\beta]
  \tx{.}
\end{gather*}
\end{lemma}

\begin{proof}
For~$\tau \in \HSEg[\alpha,\beta]$ we have~$z - \ov{w} = 2 i \alpha\, \lT{\Im}(\tau_1)$ and thus using~$\lT{\Im}(\tau_2) = \ovsmash{\Im(\tau_2)}$:
\begin{gather}
  \label{eq:prf:la:torsion_point_subvarieties_positive_definite_condition:tau2_shift}
  \tfrac{1}{4}\, (w - \ov{z}) (\Im \tau_1)^{-1} \lT{(z-\ov{w})}
  =
  \tfrac{1}{4}\, 2 i \ov\alpha \Im(\tau_1)\, (\Im \tau_1)^{-1}\, 2 i \Im(\tau_1) \lT{\alpha}
  =
  - \ov\alpha \Im(\tau_1) \lT{\alpha}
  =
  - \Im(\tau_1) [\lT{\alpha}]
  \tx{.}
\end{gather}
We transform the block decomposition~\eqref{eq:HSg_decomposition} of~$\Im(\tau)$ into upper block diagonal shape to see
\begin{gather*}
  \det\bigl( \Im(\tau) \bigr)
  =
  \det\bigl( \Im(\tau_1) \bigr)\,
  \bigl(
  \Im(\tau_2)
  +
  \tfrac{1}{4}\, (w-\ov{z}) (\Im \tau_1)^{-1} \lT{(z - \ov{w})}
  \bigr)
  \tx{.}
\end{gather*}
Since~$\Im(\tau)$ and~$\Im(\tau_1)$ are totally positive definite for~$\tau \in \HSEg$, their determinants are totally positive, and we can insert the above expression for~$\Im(\tau_1) [\lT{\alpha}]$ to conclude the bound stated in the lemma.
\end{proof}

We obtain the next statement as a corollary to \fref{Theorem}{thm:torsion_point_hecke_bound}.

\begin{proposition}%
\label{prop:torsion_point_pointwise_convergence}
Given an admissible torsion point~$(\alpha, \beta)$, every cuspidal symmetric formal Fou\-rier--Ja\-cobi series~$\sum_m \phi_m(\tau_1,z,w)\, e(m \tau_2) \in \rmFS^{(g,1)}_k$ converges absolutely and locally uniformly for~$\tau \in \HSEg[\alpha,\beta]$.
\end{proposition}

\begin{proof}
We assume that~$(\alpha,\beta)$ are~$N$\nbd{}torsion points, and let~$z$ and~$w$ be as in~\eqref{eq:torsion_point_subvarieties}.
Expanding the definition of the specialization to the torsion point~$(\alpha,\beta)$ in~\eqref{eq:def:hermitian_jacobi_forms_torsion_points} we find by \fref{Lemma}{la:hermitian_jacobi_forms_torsion_points} that
\begin{gather*}
  \big| \phi_m(\tau_1,z,w) \big|
  =
  \bigl|
  \phi_m[\alpha,\beta](\tau_1)
  \bigr|
  =
  e\bigl( - m i \Im(\tau_1) [\lT{\alpha}] \bigr)\,
  \bigl|
  \phi_m[\alpha,\beta](\tau_1)
  \bigr|
  \tx{.}
\end{gather*}
This allows us to apply the bound in \fref{Theorem}{thm:torsion_point_hecke_bound}, to conclude that for totally positive~$m$ we have
\begin{gather*}
  \big| \phi_m(\tau_1,z,w) \big|
  \ll_{E,g,k,N}
  e\bigl( - m i \Im(\tau_1) [\lT{\alpha}] \bigr)\,
  \detF\bigl( \Im(\tau_1) \bigr)^{-\frac{k}{2}}\,
  \normF(m)^{\frac{k-3}{2} + g}
  \tx{.}
\end{gather*}

\fref{Lemma}{la:torsion_point_subvarieties_positive_definite_condition} implies that
\begin{gather*}
  \HSEg[\alpha,\beta]
  =
  \bigcup_{\eps > 0}
  \big\{
  \tau \in \HSEg[\alpha,\beta]
  \condsep
  \Im(\tau_2) > \Im(\tau_1)[\lT{\alpha}] + \eps
  \big\}
  \tx{.}
\end{gather*}
The desired absolute and locally uniform convergence thus follows if we show absolute and uniform convergence for~$\tau \in \HSEg[\alpha,\beta]$ with~$\Im(\tau_2) > \Im(\tau_1)[\lT{\alpha}] + \eps$ and fixed~$\eps > 0$.
For such~$\tau_2$ we obtain the estimate
\begin{multline*}
  \big| \phi_m(\tau_1,z,w)\, e(m \tau_2) \big|
  \le
  \big| \phi_m(\tau_1,z,w) \big|\,
  e\bigl( - m i (\Im(\tau_1) [\lT{\alpha}] + \eps) \bigr)
  \\
  \ll_{E,g,k,N}
  \detF\bigl( \Im(\tau_1) \bigr)^{-\frac{k}{2}}\,
  \normF(m)^{\frac{k-3}{2} + g}\,
  \exp\bigl( - 2\pi \traceF(m)\, \eps \bigr)
  \tx{.}
\end{multline*}
We finish the proof by comparing against the geometric series over totally positive~$m \in \cOF^\vee$ and invoking dominated convergence.
\end{proof}

Using the previous statement, we can uniformly bound a symmetric formal Fourier--Jacobi series, if it additionally satisfies a polynomial equation.

\begin{proposition}%
\label{prop:torsion_point_locally_bounded}
Given~$f = \sum_m \phi_m(\tau_1,z,w)\, e(m \tau_2) \in \rmFS^{(g,1)}_k$, assume that~$a(f) = 0$ for a monic polynomial
\begin{gather*}
  a(X)
  =
  \sum_{i=0}^d a_i X^i \in \rmM_\bullet^{(g)}[X]
  \quad\tx{with coefficients }
  a_i\in \rmM^{(g)}_{(d-i) k}
  \tx{.}
\end{gather*}
Then the following sequence of partial sums is locally bounded on~$\HSEg$ independently of~$M \in \ZZ_{> 0}$:
\begin{gather}
  \label{eq:prop:torsion_point_locally_bounded}
  \sum_{\substack{
      m \in \cOF^\vee, m > 0 \\
      \traceF(m) < M
    }}
  \mspace{-18mu}
  \phi_m(\tau_1,z,w)\,
  e(m \tau_2)
  \tx{.}
\end{gather}
\end{proposition}

\begin{proof}
We let
\begin{gather*}
  \rho
  \defcol
  \HSEg
  \lra
  F_\RR
  \tx{,}\quad
  \begin{psmatrix} \tau_1 & \lT{z} \\ w & \tau_2 \end{psmatrix}
  \lmto
  \Im(\tau_2)
  +
  \tfrac{1}{4}\, (w - \ov{z}) (\Im \tau_1)^{-1} \lT{(z-\ov{w})}
  \tx{.}
\end{gather*}
We fix a compact set~$K \subset \HS_{E,g-1} \times E_\RR^{g-1} \times E_\RR^{g-1}$ and~$0 < \eps < 1$.
We define a corresponding compact subset of~$\HSEg$ by
\begin{gather*}
  \HSEg(K, \eps)
  =
  \big\{
  \tau \in \HSEg
  \condsep
  (\tau_1, z, w) \in K,\;
  \eps \le \rho(\tau) \le \tfrac{1}{\eps},\,
  \tfrac{-1}{\eps} \le \Re(\tau_2) \le \tfrac{1}{\eps}
  \big\}
  \tx{.}
\end{gather*}

We consider an admissible torsion point~$(\alpha,\beta)$ and assume~$\tau \in \HSEg(K, \eps) \cap \HSEg[\alpha,\beta]$.
By \fref{Proposition}{prop:torsion_point_pointwise_convergence} the a priori formal series~$f(\tau)$ converges.
We either have~$|f(\tau)| < 1$ or if~$|f(\tau)| \ge 1$ we can rearrange the relation~$a(f(\tau)) \slash f(\tau)^{d-1} = 0$ to obtain the estimate
\begin{gather*}
  \big| f(\tau) \big|
  \le
  \max \Big\{
  1,\,
  \sum_{i=0}^{d-1} |a_i(\tau)|
  \Big\}
  \le
  1 + \sum_{i=0}^{d-1} |a_i(\tau)|
  \tx{.}
\end{gather*}
Since~$\HSEg(K, \eps)$ is compact and all~$a_i$ are continuous, we have a finite supremum
\begin{gather*}
  b(K, \eps)
  \defeq
  \sup_{\tau \in \HSEg(K, \eps)}
  \Big(
  1 + \sum_{i=0}^{d-1} |a_i(\tau)|
  \Big)
  <
  \infty
  \tx{.}
\end{gather*}
This yields a uniform bound for~$f(\tau)$ on the admissible torsion points of~$\HSEg(K,\eps)$.

For an admissible torsion point as before, we may compute $\phi_m$ by means of the integral
\begin{gather*}
  \phi_m(\tau_1, z, w)
  \asymp_E
  \int_{F_\RR \slash \cOF}
  f\Bigl( \begin{psmatrix}
      \tau_1 & \lT{z}      \\
      w      & u_2 + i v_2
    \end{psmatrix}\Bigr)\,
  e\bigl( -m (u_2 + i v_2) \bigr)
  \,\rmd\!u_2
\end{gather*}
for any~$v_2 > \Im(\tau_1)[\lT{\alpha}]$, where for convenience we suppress the normalization by the volume of~$F_\RR \slash \cOF$.
We choose~$v_2 = \Im(\tau_1)[\lT{\alpha}] + \eps$ and use the bound for~$f$ on~$\HSEg(K,\eps)$ to derive the bound
\begin{gather*}
  \big| \phi_m(\tau_1, z, w) \big|
  \ll_E
  b(K,\eps)\,
  e\bigl( - m i (\Im(\tau_1)[\lT{\alpha}] + \eps) \bigr)
  \tx{.}
\end{gather*}

For~$\tau \in \HSEg(K, 2\eps) \cap \HSEg[\alpha,\beta] \subset \HSEg(K, \eps) \cap \HSEg[\alpha,\beta]$, and~$M \in \ZZ_{> 0}$ the lower bound on~$\rho(\tau)$ together with the simplification in~\eqref{eq:prf:la:torsion_point_subvarieties_positive_definite_condition:tau2_shift} yields the estimate
\begin{align*}
         &
  \Bigg|
  \sum_{\substack{
  m \in \cOF^\vee, m > 0 \\
      \traceF(m) < M
    }}
  \mspace{-18mu}
  \phi_m(\tau_1,z,w)\,
  e(m \tau_2)
  \Bigg|
  \\
  \le {} &
  \sum_{\substack{
  m \in \cOF^\vee, m > 0 \\
      \traceF(m) < M
    }}
  \mspace{-18mu}
  \big| \phi_m(\tau_1,z,w) \big|\,
  e\bigl( m i (\Im(\tau_1)[\lT{\alpha}] + 2 \eps) \bigr)
  \le
  b(K,2\eps)
  \mspace{-18mu}
  \sum_{\substack{
  m \in \cOF^\vee, m > 0 \\
      \traceF(m) < M
    }}
  \mspace{-18mu}
  e( m i \eps )
  \tx{.}
\end{align*}
Using Riemann sums, we see that
\begin{gather*}
  \sum_{m \in \cOF^\vee, m > 0}
  \mspace{-12mu}
  e( m i \eps )
  =
  \sum_{m \in \cOF^\vee, m > 0}
  \mspace{-12mu}
  \exp\bigl( - 2 \pi \eps \traceF(m) \bigr)
  \ll_F
  \int_{m \in F_\RR, m > -\delta_F}
  \mspace{-18mu}
  \exp\bigl( - 2 \pi \eps \traceF(m) \bigr)
  \, \rmd\!m
  \ll_F
  1
  \tx{,}
\end{gather*}
where~$\delta_F$ is the radius of a Voronoi cell of~$\cOF^\vee \subset F_\RR \cong \RR^{n_F}$.

This shows that the partial sums in the proposition are bounded on the subset of admissible torsion points~$\tau \in \HSEg(K, 2\eps)$, which is a dense subset by \fref{Lemma}{la:admissible_torsion_points_dense}.
Since the partial sums are continuous, we may conclude that they are bounded on all of~$\HSEg(K, 2\eps)$.
Since~$\HSEg(K, 2\eps)$ by~\eqref{eq:prf:la:torsion_point_subvarieties_positive_definite_condition:tau2_shift} exhausts~$\HSEg$ as~$K$ varies and~$\eps \ra 0$, we obtain the assertion.
\end{proof}

\subsection{Torsion points on divisors}%
\label{ssec:convergence:torsion_point_divisors}

To control possible singularities of meromorphic Hermitian modular forms in \fref{Section}{ssec:convergence:meromorphic_and_regular}, we investigate admissible torsion points on arbitrary divisors.

The following basic statement from differential geometry, originating in the implicit function theorem, will be instrumental in this section.

\begin{lemma}%
\label{la:fibration_base_tangent_surjective}
Consider a fibration~$F \ra X \ra B$ of real\/~$\rmC^1$\nbd{}manifolds with projection~$\pi \defcol X \ra B$ and points~$x \in X$, $b = \pi(x) \in B$.
Given a $\rmC^1$\nbd{}submanifold~$Y \subseteq X$ that contains~$x$, assume that the differential of\/~$\pi$ maps the tangent space~$\fraky$ of\/~$Y$ at~$x$ onto the tangent space of\/~$B$ at~$b$.
Then~$\pi(Y) \subset B$ contains a neighborhood of\/~$b$.
\end{lemma}

\begin{proof}
After passing to suitable charts for neighborhoods of~$x$ and~$b$, we can assume that~$X = \RR^{m+n}$, $B = \RR^m$, $\pi$ is the projection onto the first~$m$ coordinates, and~$x = 0$, $b = 0$.
For ease of discussion, we identify the tangent space~$\frakx$ of~$X$ at~$x$ and its dual with~$\RR^{m+n}$.
Similarly for~$B$ and its tangent space~$\frakb$ at~$b$.
Under these identifications, the differential of~$\pi$ is the projection to the first~$m$ coordinates, and we continue to denote it~$\pi$.

By restricting to a sufficiently small neighborhood of~$x$ and again adjusting charts to maintain that~$X = \RR^{m+n}$, we can assume that~$Y$ is the common zero set of~$\rmC^1$\nbd{}functions~$f_j(x)$ with~$1 \le j \le d$, the codimension of~$Y$ at~$x$.

Using our identifications, the tangent space~$\fraky$ equals the kernel of the matrix~$N = (\partial_{x_i}\, f_j)$.
Since~$Y$ is a $\rmC^1$\nbd{}submanifold and thus~$\fraky$ has dimension~$m+n-d$, $N$ has full rank~$d$.
Using again the identifications that we have made, our assumption on the image of~$\fraky$ reads~$\pi(\fraky) = \frakb$.
That is, the submatrix of~$N$ with column indices~$m < i \le m+n$ has rank~$d$.
We can replace~$Y$ by a submanifold of~$Y$ by intersecting it with hyperplanes~$x_j = 0$ for suitable~$m < j \le m + n$ in such a way that~$d = n$.
In particular, $N$ and its submatrix with column indices~$m < i \le m+n$ have rank~$n$.

We are now in position to invoke the implicit function theorem to obtain a parametrization of~$x \in Y$ in terms of~$x_i$, $1 \le i \le m$.
This finishes the proof.
\end{proof}

For the purpose of this section, and only this section, we introduce the following notation adjusted to the proof of \fref{Proposition}{prop:divisor_intersect_torsion_point}.
We have maps
\begin{gather}
  \label{eq:def:HSg_fibration_projection}
  \begin{alignedat}{2}
    \pi_\alpha
    \defcol
    \HSEg
     & \lthra
    \Mat{g-1,1}(E_\RR)
    \tx{,}\quad
     &
    \tau
     & \lmto
    \tfrac{1}{2i} ( z - \ov{w} )\, \Im(\tau_1)^{-1}
    \tx{,}
    \\
    \pi_\beta
    \defcol
    \HSEg
     & \lthra
    \Mat{g-1,1}(E_\RR)
    \tx{,}\quad
     &
    \tau
     & \lmto
    \tfrac{1}{2} ( z + \ov{w} ) - 2 \pi_\alpha(\tau)\, \Re(\tau_1)
    \tx{,}
  \end{alignedat}
\end{gather}
which for~$\tau \in \HSEg[\alpha, \beta]$ recover~$\alpha$ and~$\beta$.
Using the action of~$\transJU(\ov\alpha,\ov\beta)$, we see that we have a fibration of smooth real manifolds
\begin{gather}
  \label{eq:def:HSg_fibration}
  \HS_{E,g-1} \times \HS_F
  \lra
  \HSEg
  \lthra
  \Mat{g-1,1}(E_\RR) \times \Mat{g-1,1}(E_\RR)
  \tx{,}
\end{gather}
where the projection map is given by~$(\pi_\alpha, \pi_\beta)$.

In standard notation, we write~$\frakg$ for the complexified Lie algebra of~$\SU{g,g}(F_\RR)$ and~$\frakk$ for the complexified Lie algebra of the stabilizer of~$i \in \HSEg$ in~$\SU{g,g}(F_\RR)$.
Using that~$\HSEg$ is a homogeneous space for~$\SU{g,g}(F_\RR)$, we see that the fibration in~\eqref{eq:def:HSg_fibration} yields the fiber product decomposition~$\frakb \times_\frakz \frakf$ of the tangent space at~$i$ for the following Lie algebras.
The base space of~\eqref{eq:def:HSg_fibration} yields the unipotent radical of the Lie algebra of~$Q_1(F_\RR)$, which we denote by~$\frakb$.
(Note here that~$\frakb$ is not a Borel subalgebra, $\frakb$ stands for base.)
The fiber at~$i$ of~\eqref{eq:def:HSg_fibration} yields the product of the Lie algebra of~$\rmS(\U{g-1,g-1} \times \U{1,1})(F_\RR)$, which we denote by~$\frakf$, where the unitary groups are defined in analogy with~\eqref{eq:def:SUgg}.
The intersection of~$\frakb$ and~$\frakf$ equals the center~$\frakz$ of~$\frakb$, which is isomorphic to~$F_\CC$.

\begin{proposition}%
\label{prop:divisor_intersect_torsion_point}
Given a divisor~$D \subset \HSEg$ and a complex submanifold~$W \subset \HSEg$ of codimension at least~$2$, there is an admissible torsion point~$(\alpha,\beta)$ such that\/~$\HSEg[\alpha,\beta]$ intersects~$(\SU{g,g}(\cOF)\, D) \setminus W$ transversally in at least one point.
\end{proposition}

\begin{proof}
We can replace~$D$ with one of its irreducible components.
The singular part of~$D$ has codimension at least~$2$.
We can thus assume that~$D \setminus W$ is smooth by enlarging~$W$.

We fix any point~$\tau \in D \setminus W$, an element~$\td{g} \in \SU{g,g}(F_\RR)$ with~$\tau = \td{g} i$, and write~$\td{g}\, \frakd$ for the tangent space of~$D$ at~$\tau$, where~$\frakd \subset \frakb \times_\frakz \frakf \subset \frakg$.
To show the proposition it suffices to find~$\ga \in \SU{g,g}(\cOF)$ such that~$\ga \td{g} (\frakd + \frakk)$ maps onto~$\frakb$ under the vector space projection~$\frakg \thra \frakg \slash (\frakf + \frakk)$.
Then~$\ga D$ in a neighborhood of~$\tau$ projects to an open set of the fibration base by \fref{Lemma}{la:fibration_base_tangent_surjective}.
By \fref{Lemma}{la:admissible_torsion_points_dense} this open set contains an admissible torsion point, as required in the statement.

To show that such a~$\ga$ exists, we observe that~$\frakd \subset \frakb \times_\frakz \frakf = \frakg \slash \frakk$ is the intersection of the kernel of two elements~$n^\pm \in (\frakb \times_\frakz \frakf)^\vee = (\frakg \slash \frakk)^\vee$, which are conjugate under the Cartan involution, since~$D$ is complex of codimension one.
The adjoint representation of~$\SU{g,g}(F_\RR)$ is a direct sum of~$n_F$ irreducibles, all intersected nontrivially by~$\frakk$.
Since~$\SU{g,g}(\cOF)$ is Zariski dense in~$\SU{g,g}(F_\RR)$, we find~$\ga$ such that~$\ga \td{g} n^\pm$ projects onto nonzero elements of~$\frakf^\vee \slash \frakz^\vee \subset (\frakb \times_\frakz \frakf)^\vee \slash \frakb^\vee$.
They span a space of~$\frakf^\vee \slash \frakz^\vee$ that is invariant under the Cartan involution, and therefore is of dimension~$2$.
That is, their kernel surjects onto~$\frakb$ as required.
\end{proof}

\subsection{Convergence to Hermitian modular forms}%
\label{ssec:convergence:meromorphic_and_regular}

In this section, we establish automatic convergence of symmetric formal Fourier--Jacobi series of cogenus~$1$ and trivial arithmetic type.

\begin{theorem}%
\label{thm:convergence_scalar_valued_cogenus_one}
For~$g \ge 2$ and~$k \in \ZZ$, we have
\begin{gather*}
  \rmFM^{(g,1)}_k
  =
  \rmM^{(g)}_k
  \tx{.}
\end{gather*}
\end{theorem}

\begin{proof}
We first show that given~$f \in \rmFM^{(g,1)}_k$ there is a nonzero Hermitian modular form~$f'$ such that~$f f'$ converges and is modular.
We may assume that~$f$ is cuspidal, since given any symmetric formal Fourier--Jacobi series~$f$ and a nonzero cusp form~$f'$, the product~$f f'$ is a cuspidal, symmetric formal Fourier--Jacobi series.

By \fref{Corollary}{cor:symmetric_formal_fourier_jacobi_series_algebraic} any~$f \in \rmFS^{(g,1)}_k \subseteq \rmFM^{(g,1)}_k$ satisfies a polynomial equation
\begin{gather*}
  \sum_{i = 0}^d a_i f^i
  =
  0
  \in
  \rmFM^{(g,1)}_{dk + k'}
  \quad
  \tx{with coefficients }
  a_i \in \rmM^{(g)}_{(d-i)k + k'}
\end{gather*}
with suitable~$k'$ and~$d$ and~$a_d \ne 0$.
Hence the product~$a_d f \in \rmFS^{(g,1)}_{k + k'}$ satisfies a monic polynomial relation.
Now \fref{Proposition}{prop:torsion_point_locally_bounded} asserts that the partial sums associated with~$a_d f$ are locally uniformly bounded on~$\HSEg$.
Thus there exists a subsequence of partial sums that converges to a holomorphic function by Montel's theorem.
Further, \fref{Proposition}{prop:torsion_point_pointwise_convergence} asserts that the Fourier--Jacobi series~$f$, and hence also~$a_d f$, converges pointwise absolutely for all~$\tau \in \HSEg[\alpha,\beta]$ and all admissible torsion points~$(\alpha,\beta)$.
Since such~$\tau$ are dense in~$\HSEg$ by \fref{Lemma}{la:admissible_torsion_points_dense}, the Fourier--Jacobi series~$a_d f$ converges locally uniformly.

The function defined by~$a_d f$ is modular invariant with respect to~$Q_1(\cOF)$, since each term in its Fourier--Jacobi series is.
It is manifestly periodic, and therefore has a locally uniformly convergent Fourier series expansion.
By the symmetry condition~\eqref{eq:hermitian_fourier_coefficient_symmetry} on the Fourier coefficients of a formal symmetric Fourier--Jacobi series,~$a_d f$ is modular invariant with respect to~$P_g(\cOF)$.
Using \fref{Proposition}{prop:siegel_and_klingen_parabolics_generate} we find that~$a_d f$ is a Hermitian modular form.

We have shown that for every~$f \in \rmFM^{(g,1)}_k$ there is a Hermitian modular form~$f'$ such that the formal Fourier--Jacobi series~$f f'$ converges on~$\HSEg$ and is modular.
For clarity, we set~$\td{f} = f f'$.
We next show that the meromorphic Hermitian modular form~$\td{f} \slash f'$ is indeed holomorphic.
Let~$D$ be the polar divisor of~$\td{f} \slash f'$ and assume that~$D$ is nontrivial, i.e., $\td{f} \slash f'$ is not holomorphic.
Since $\td{f} \slash f'$ is modular, we have~$D = \SU{g,g}(\cOF)\, D$.
Let~$W$ be the union of the intersections of~$D$ with all components of the divisors of~$\td{f}$ and~$f'$ that are distinct from~$D$.
This is a complex submanifold of codimension at least~$2$.
By \fref{Proposition}{prop:divisor_intersect_torsion_point} there is an admissible torsion point~$(\alpha,\beta)$ such that~$\HSEg[\alpha,\beta]$ intersects~$D \setminus W$ transversally, say in~$\tau$.

Let~$D'$ be an irreducible component of~$D$ that contains~$\tau$, and set~$D'[\alpha,\beta] = D \cap \HSEg[\alpha,\beta]$.
We analyze the restriction of~$\td{f} \slash f'$ to~$\HSEg[\alpha,\beta]$.
Since~$D$ and~$\HSEg[\alpha,\beta]$ intersect transversally, its order along the divisor~$D'[\alpha,\beta] \subset \HSEg[\alpha,\beta]$ is given by
\begin{gather*}
  \ord_{D'[\alpha,\beta]} \bigl( \bigl(\td{f} \slash f' \bigr) \big|_{\HSEg[\alpha,\beta]} \bigr)
  =
  \ord_{D'} \bigl( \td{f} \slash f' \bigr)
  <
  0
  \tx{.}
\end{gather*}

By \fref{Proposition}{prop:torsion_point_pointwise_convergence}, the formal Fourier--Jacobi series~$f$ converges on~$\HSEg[\alpha,\beta]$ to a holomorphic function.
We write~$f_{\alpha,\beta}$ for this function.
We have
\begin{gather*}
  0
  \le
  \ord_{D'[\alpha,\beta]} \bigl( f_{\alpha,\beta} \bigr)
  =
  \ord_{D'[\alpha,\beta]}\Bigl(
  \bigl( f_{\alpha,\beta} \cdot \bigl( f' \big|_{\HSEg[\alpha,\beta]} \bigr) \bigr)
  \bigslash
  \bigl( f' \big|_{\HSEg[\alpha,\beta]} \bigr)
  \Bigr)
  \tx{.}
\end{gather*}
Since evaluation of Fourier series and products commute, we have
\begin{gather*}
  f_{\alpha,\beta} \cdot \bigl( f' \big|_{\HSEg[\alpha,\beta]} \bigr)
  =
  \td{f} \big|_{\HSEg[\alpha,\beta]}
  \tx{.}
\end{gather*}
Combining this with the previous inequality, we arrive at the contradiction that the order of~$\td{f} \slash f'$ along~$D'$ is nonnegative.
Thus, our assumption that~$D$ were nontrivial was incorrect.
In other words,~$\td{f} \slash f'$ is a holomorphic Hermitian modular form.

We have~$\td{f} = f f'$ in the algebra of symmetric formal Fourier--Jacobi series.
Since this algebra is an integral domain, the Fourier--Jacobi expansion of~$\td{f} \slash f'$ and of~$f$ coincide.
That is,~$f$ converges and is modular as we claimed.
\end{proof}

%% file: sections/04_vector_valued.tex
\section{The vector-valued and higher cogenus case}%
\label{sec:vector_valued_higher_cogenus}

The goal of this section is to extend \fref{Theorem}{thm:convergence_scalar_valued_cogenus_one} to both vector-valued formal Fourier--Jacobi expansions and to arbitrary cogenus~$1 \le h < g$.
Our proof of \fref{Theorem}{mainthm:convergence} successively invokes the reduction to the scalar-valued case in \fref{Proposition}{prop:reduction_to_scalar_weights_trivial_type} and the reduction to cogenus~$h-1$ via \fref{Proposition}{prop:reduction_to_lower_cogenus}.

\subsection{Reduction to the scalar-valued cusp forms}%
\label{ssec:vector_valued_higher_cogenus:reduction_to_scalar}

The arguments in this section follow closely the lines of previous work~\cite{bruinier-2014}.
In this section we write~$\rmZ(\cOF)$ for the center of~$\SU{g,g}(\cOF)$.

\begin{lemma}
\label{la:large_weights_yield_global_sections}
Given a weight~$k_0 \in \ZZ$ and an arithmetic type~$\rho$ such that every central element~$\eps 1_{2g}$ of\/~$\SU{g,g}(\cOF)$ acts as~$\normF(\eps)^{-gk_0}$ under~$\rho$, there are~$k \in \ZZ$ and a finite set~$B \subset \rmS^{(g)}_{k_0 + k}(\rho)$ of modular forms such that for every~$\tau \in \HSEg$ with~$\Stab_{\SU{g,g}(\cOF)}(\tau) = \rmZ(\cOF)$ we have
\begin{gather}
  \label{eq:la:large_weights_yield_global_sections}
  \linspan\, \CC \big\{ f(\tau) \condsep f \in B \big\}
  =
  V(\rho)
  \tx{.}
\end{gather}
\end{lemma}

\begin{proof}
We consider the Satake compactification~$X_g$ of~$\SU{g,g}(\cOF) \backslash \HSEg$.
Let~$\cM_{k'}(\rho')$ be the sheaf of Hermitian modular forms on~$X_g$ of weight~$k'$ and arithmetic type~$\rho'$, and~$\cS_{k'}(\rho')$ the subsheaf of cusp forms.
If~$\rho'$ is trivial, we omit it from our notation, writing~$\cM_{k'}$ and~$\cS_{k'}$.
There is an even, positive~$k'$ such that~$\cM_{k'}$ and all its powers are very ample by Baily--Borel.
In particular, by Theorem~II-5.17 of~\cite{hartshorne-1977}, $\cS_{n k' + k_0}(\rho) = \cM_{n k'} \otimes \cS_{k_0}(\rho)$ is generated by global sections if~$n$ is sufficiently large.
We set~$k = n k'$ for one such~$n$ and let~$\cB$ be a basis of global sections of~$\cS_{k + k_0}(\rho)$.

For points~$\tau \in \HSEg$ with~$\Stab_{\SU{g,g}(\cOF)}(\tau) = \rmZ(\cOF)$, the stalk of~$\cS_{k_0+k}(\rho)$ at~$\SU{g,g}(\cOF)\, \tau$ equals the base change~$\cO_{X_g,\tau} \otimes_\CC V(\rho)$ of~$V(\rho)$, where~$\cO_{X_g,\tau}$ is the stalk of the structure sheaf of~$X_g$ at~$\SU{g,g}(\cOF)\, \tau$.
The elements of~$\cB$ span~$\cO_{X_g,\tau} \otimes V(\rho)$ as a module over~$\cO_{X_g,\tau}$, and we conclude the proof by letting~$B$ be the set of cusp forms corresponding to~$\cB$.
\end{proof}

\begin{proposition}%
\label{prop:reduction_to_scalar_weights_trivial_type}
Given~$1 \le h < g$, we assume that\/~$\rmFS^{(g,h)}_k = \rmS^{(g)}_k$ for all~$k \in \ZZ$.
Then for all~$k \in \ZZ$ and all arithmetic types~$\rho$, we have
\begin{gather*}
  \rmFM^{(g,h)}_k(\rho)
  =
  \rmM^{(g)}_k(\rho)
  \tx{.}
\end{gather*}
\end{proposition}

\begin{proof}
We can and will assume that~$\eps 1_{2g} \in \rmZ(\cOF)$ acts via~$\rho$ as~$\normF(\eps)^{-g k}$ by passing to the corresponding subrepresentation, since symmetric formal Fourier--Jacobi series vanish on the complement by the symmetry condition in \fref{Definition}{def:symmetric_formal_fourier_jacobi_series}.
Let~$k'$ and~$B$ be the weight and the set of cusp forms in \fref{Lemma}{la:large_weights_yield_global_sections} with~$k_0$ and~$\rho$ replaced by~$-k$ and~$\rho^\vee$, the dual of~$\rho$.
We let~$\pi_\rho$ denote the evaluation map from~$\rho^\vee \otimes \rho$ to the trivial representation.

We let~$N$ be the level of~$\rho$.
Given
\begin{gather}
  f(\tau)
  =
  \sum_{m \in \frac{1}{N} \MatD{h}(\cOE)^\vee}
  \phi_m(\tau_1, z, w)\,
  e(m \tau_2)
  \in
  \rmFM^{(g,h)}_k(\rho)
  \tx{,}
\end{gather}
and~$f^\vee \in B$, we have
\begin{gather*}
  \big\langle f^\vee, f \big\rangle
  \defeq
  \pi_\rho \circ \big( f^\vee \otimes f \big)
  \in
  \rmFS^{(g,h)}_{k + k'}
  \tx{,}
\end{gather*}
since the notion of symmetric formal Fourier--Jacobi series is compatible with both tensor products and homomorphisms of arithmetic types.
By our assumption~$\rmFS^{(g,h)}_{k+k'} = \rmS^{(g)}_{k+k'}$, we find that~$\langle f^\vee, f \rangle$ is a Hermitian modular form.

Fix any~$\tau$ with $\Stab_{\SU{g,g}(\cOF)}(\tau) = \rmZ(\cOF)$, and let~$B_\tau \subseteq B$ be such that the vectors~$f^\vee(\tau)$ for~$f^\vee \in B_\tau$ form a basis of~$V(\rho^\vee)$.
In particular, we can view~$B_\tau$ as a Hermitian modular form with values in
\begin{gather*}
  V(\rho^\vee) \otimes \CC\{B_\tau\}
  \cong
  \Hom_\CC\big( V(\rho), \CC\{B_\tau\} \big)
  \tx{,}
\end{gather*}
whose value at~$\tau$ is an invertible homomorphism.
For all~$\tau'$ at which~$B_\tau(\tau')$ is invertible, we can define a function by~$B_\tau^{-1}(\tau') = B_\tau(\tau')^{-1}$, which can be viewed as a meromorphic Hermitian modular form of weight~$-k'$ and type~$\rho \otimes \CC\{B_\tau\}^\vee$.

Recall that~$\langle f^\vee, f \rangle$ is a modular form as opposed to a mere formal Fourier--Jacobi series for every~$f^\vee \in B$.
We can view
\begin{gather*}
  \big\langle B_\tau, f \big\rangle
  \defeq
  \sum_{f^\vee \in B_\tau}
  \langle f^\vee, f \rangle \otimes f^\vee
  \in
  \rmS^{(g)}_{k+k'} \otimes \CC\{B_\tau\}
\end{gather*}
as a Hermitian modular form of isotrivial type~$\CC\{B_\tau\}$.

Combining the previous definition of~$B_\tau^{-1}$ we find that
\begin{gather}
  \label{eq:prop:reduction_to_scalar_weights_trivial_type:mul_div}
  \Bigl\langle B_\tau^{-1},\, \big\langle B_\tau, f \big\rangle \Bigr\rangle (\tau')
  \defeq
  B_\tau^{-1}(\tau') \big( \big\langle B_\tau, f \big\rangle(\tau') \big)
  \in
  V(\rho)
  \tx{.}
\end{gather}
Since modular forms are compatible with homomorphisms of arithmetic types, the right hand side defines a meromorphic modular form.
It is regular at~$\tau$, since~$B_\tau^{-1}$ is.
Moreover, it does not depend on~$B_\tau$, since changing the set~$B_\tau \subset B$ to, say,~$B'_\tau \subset B$ corresponds to composing~$B_\tau$ on the left with a meromorphic modular form with values in invertible~$\CC$-homomorphisms from~$\CC\{B_\tau\}$ to~$\CC\{B'_\tau\}$.
It is canceled by its inverse, which arises from the relation between~$B_\tau^{-1}$ and~$B_\tau^{\prime\,-1}$.

We obtain a meromorphic modular form~\eqref{eq:prop:reduction_to_scalar_weights_trivial_type:mul_div} that is regular at every~$\tau \in \HSEg$ whose stabilizer in~$\SU{g,g}(\cOF)$ equals~$\rmZ(\cOF)$.
Since fixed points on~$\HSEg$ have codimension at least~$2$ if~$g \ge 2$, we conclude that it is holomorphic.
Since tensor products and homomorphisms of arithmetic types can be applied on the level of Fourier coefficients, the Fourier--Jacobi coefficients of~\eqref{eq:prop:reduction_to_scalar_weights_trivial_type:mul_div} coincide with the ones of~$f$.
\end{proof}

\subsection{Reduction to lower cogenus}%
\label{ssec:vector_valued_higher_cogenus:reduction_to_lower_cogenus}

In this section we perform the reduction from cogenus~$h$ to cogenus~$h-1$.
In light of \fref{Section}{ssec:vector_valued_higher_cogenus:reduction_to_scalar} it suffices to consider cusp forms of trivial arithmetic types as input.
However, we will produce vector-valued formal Fourier--Jacobi series from them.

We recall the decomposition of~$\tau \in \HSEg$ and Fourier indices~$t \in \MatD{g}(\cOE)^\vee$ given in~\eqref{eq:HSg_decomposition} and~\eqref{eq:hermitian_fourier_index_decomposition} and their refinements in~\eqref{eq:def:variable_subdivision} and~\eqref{eq:def:fourier_index_subdivision}:
\begin{gather*}
  \tau
  =
  \left(
  \begin{array}{c|cc}
      \tau_1 & \lTr{z}{_{11}} & \lTr{z}{_{12}} \\
      \hline
      w_{11} & \tau_{21}      & \lTr{z}{_{2}}  \\
      w_{12} & w_{2}          & \tau_{22}
    \end{array}
  \right)
  =
  \left(
  \begin{array}{cc|c}
      \tau'_{11} & \lTr{z}{^\prime_1} & \lTr{z}{^\prime_{21}} \\
      w'_{1}     & \tau'_{12}         & \lTr{z}{^\prime_{22}} \\
      \hline
      w'_{21}    & w'_{22}            & \tau'_2
    \end{array}
  \right)
  \tx{,}\quad
  t
  =
  \left(
  \begin{array}{c|cc}
      n   & \lTr{\ovsmash{r}}{_1} & \lTr{\ovsmash{r}}{_2}    \\
      \hline
      r_1 & m_{11}                & \lTr{\ovsmash{m}}{_{21}} \\
      r_2 & m_{21}                & m_{22}
    \end{array}
  \right)
  =
  \left(
  \begin{array}{cc|c}
      n'_{11} & \lTr{\ovsmash{n}}{^\prime_{21}} & \lTr{\ovsmash{r}}{^\prime_1} \\
      n'_{21} & n'_{22}                         & \lTr{\ovsmash{r}}{^\prime_2} \\
      \hline
      r'_1    & r'_2                            & m'
    \end{array}
  \right)
  \tx{.}
\end{gather*}
Given a formal Fourier--Jacobi series~$f$ of cogenus~$h$, we define its cogenus~$h-1$ Fourier--Jacobi coefficient of index~$m'$ as
\begin{gather}
  \label{eq:def:sub_fourier_jacobi_coefficient}
  \psi_{m'}( \tau'_1, z', w' )
  \defeq
  \sum_{n', r'}
  c( \phi_m;\, n, r)\,
  e(n' \tau'_1 + \lTr{r}{^\prime} z' + \lTr{\ovsmash{r}}{^\prime} w')
  \tx{,}
\end{gather}
where the sum runs over indices~$t$ as in~\eqref{eq:def:fourier_index_subdivision} with fixed~$m' \in \MatD{h-1}(\cOE)^\vee$ and
\begin{gather*}
  n' \in \MatD{g-h+1}(\cOE)^\vee,
  \quad
  r'_1 \in \Mat{h-1,g-h}(\cOE^\vee),
  \quad
  r'_2 \in \Mat{h-1,1}(\cOE^\vee)
  \tx{.}
\end{gather*}

To describe~$\psi_{m'}$ in terms of symmetric formal Fourier--Jacobi series, we employ Hermitian Jacobi theta series of index~$m'$ defined in~\eqref{eq:def:hermitian_jacobi_theta}.
Since this definition requires positive definite~$m'$, we restrict to cusp forms in the next statement.

\begin{proposition}
\label{prop:reduction_to_lower_cogenus}
Given $f \in \rmFS^{(g,h)}_k$ with~$1 < h < g$ and a totally positive definite~$m' \in \MatD{h-1}(\cOE)^\vee$, the Fourier--Jacobi coefficient~$\psi_{m'}$ defined in~\eqref{eq:def:sub_fourier_jacobi_coefficient} has a formal theta decomposition
\begin{gather}
  \label{eq:prop:reduction_to_lower_cogenus:theta_decomposition}
  \psi_{m'}(\tau'_1, z', w')
  =
  \sum_{\mu' \in \disc(m')^{g-h+1}}
  \mspace{-18mu}
  \td{f}_{\mu'}(\tau'_1)\,
  \theta_{m',\mu'}(\tau'_1, z', w')
  \tx{,}
\end{gather}
where~$\disc(m')$ and the theta series are the ones defined in~\eqref{eq:def:hermitian_jacobi_theta}, and the product on the right hand side is a product of formal series in~$\CC\llbracket e(\tau_{i\!j}) \rrbracket_{i,j}$.
The vector-valued formal series with components~$\td{f}_{\mu'}$ satisfies
\begin{gather*}
  \bigl( \td{f}_{\mu'} \bigr)_{\mu'}
  \in
  \rmFS^{(g-h+1,1)}_{k - h + 1}\big( \rho_{m'}^{(g-h+1)\, \vee} \big)
  \tx{.}
\end{gather*}
\end{proposition}

\begin{proof}
Throughout the proof, we fix a set of representatives~$r'(\mu')$ for~$\mu' \in \disc(m')^{g-h+1}$ with decompositions~$r'_1(\mu') = r'_1(\mu'_1)$ and~$r'_2(\mu') = r'_2(\mu'_2)$ as in~\eqref{eq:def:fourier_index_subdivision}.

We first prove the formal decomposition in~\eqref{eq:prop:reduction_to_lower_cogenus:theta_decomposition}.
We set
\begin{gather*}
  u_\la
  \defeq
  \begin{psmatrix} 1_{g-h+1} & \\ \la & 1_{h-1} \end{psmatrix}
  \in
  \GL{g}(\cOE)
  \quad\tx{for }
  \la \in \Mat{h-1,g-h+1}(\cOE)
  \tx{.}
\end{gather*}
We recall that~$\psi_{m'}(\tau'_1, z', w')\, e(m' \tau'_2) = \sum c(f; t)\, e(t \tau)$, where the sum runs over~$t \in \MatD{g}(\cOE)^\vee$ with bottom right block~$m'$.
Throughout the proof, all references to~$t$ implicitly assume this bottom right block.
We have
\begin{gather*}
  t[u_\la]
  =
  \begin{psmatrix} 1_{g-h+1} & \lT{\ovsmash\la} \\ & 1_{h-1} \end{psmatrix}\,
  \begin{psmatrix} n' & \lTr{\ovsmash{r}}{^\prime} \\ r' & m ' \end{psmatrix}\,
  \begin{psmatrix} 1_{g-h+1} & \\ \la & 1_{h-1} \end{psmatrix}
  =
  \begin{psmatrix}
    n' + \lTr{\ovsmash{r}}{^\prime} \la + \lT{\ovsmash\la} r' + \lT{\ovsmash\la} m' \la & \lTr{\ovsmash{r}}{^\prime} + \lT{\ovsmash\la} m' \\
    r' + m' \la & m'
  \end{psmatrix}
  \tx{.}
\end{gather*}
For every~$t$, there is a unique~$\la$ such that~$t[u_\la]$ has bottom left block~$r'(\mu')$.
Since~$\det(u_\la) = 1$, the symmetry condition~\eqref{eq:hermitian_fourier_coefficient_symmetry} with respect to~$u_\la$ asserts that~$c(f; t) = c(f; t[u_\la])$.
This implies that
\begin{gather*}
  \psi_{m'}(\tau'_1, z', w')\,
  e(m' \tau'_2)
  =
  \sum_{\mu'}
  \sum_{t}
  c(f; t)\,
  \sum_{\la}
  e(t[u_\la] \tau)
  \tx{,}
\end{gather*}
where the outer sum runs over~$\mu'$ as in the statement, the sum over~$t \in \MatD{g}(\cOE)^\vee$ is restricted to~$t$ with bottom left block~$r'(\mu')$, and the inner sum runs over~$\la$ as in the definition of~$u_\la$.

To simplify the inner sum over~$\la$, we set~$\la' = r'(\mu') + m' \la$, which yields
\begin{gather*}
  m^{\prime\,-1}[\la']
  =
  m^{\prime\,-1}[r'(\mu')]
  +
  \lT{\ovsmash{r'(\mu')}} \la
  +
  \lT{\ovsmash\la} r'(\mu')
  +
  \lT{\ovsmash\la} m' \la
  \tx{.}
\end{gather*}
After comparing the right hand side of this identity with the top left block of~$t[u_\la]$ for~$r' = r'(\mu')$, we find that
\begin{gather}
  \label{eq:prf:prop:reduction_to_lower_cogenus:theta_decomposition}
  \psi_{m'}(\tau'_1, z', w')
  =
  \sum_{\mu'}
  \Bigl(
  \sum_{t}
  c(f; t)\,
  e\bigl( \bigl( n' - m^{\prime\,-1}[r'(\mu')] \bigr)\, \tau'_1 \bigr)\,
  \Bigr)\,
  \Bigl(
  \sum_{\la}
  e\bigl(
    m^{\prime\,-1}[\la'] \tau'_1
    +
    \lTr{\la}{^\prime} z'
    +
    \lTr{\ovsmash\la}{^\prime} w'
    \bigr)
  \Bigr)
  \tx{.}
\end{gather}
Each term~$e(t \tau - m' \tau'_2)$ that appears in the Fourier expansion of~$\psi_{m'}$ appears exactly once in the product on the right hand side of~\eqref{eq:prf:prop:reduction_to_lower_cogenus:theta_decomposition}, ensuring that the expression is well-defined as a formal series.

We define~$\td{f}_{\mu'}$ as the sum over~$t$ in~\eqref{eq:prf:prop:reduction_to_lower_cogenus:theta_decomposition}.
The sum over~$\la$ in~\eqref{eq:prf:prop:reduction_to_lower_cogenus:theta_decomposition} matches the definition of the Hermitian Jacobi theta series~$\theta_{m',\mu'}$.
Thus, we have established the formal theta decomposition in~\eqref{eq:prop:reduction_to_lower_cogenus:theta_decomposition}.

Let~$\td{f} = (\td{f}_{\mu'})_{\mu'}$.
It remains to show that~$\td{f}$ is a symmetric formal Fourier--Jacobi series.
To simplify the argument, we let~$\langle \cdot\,,\,\cdot \rangle$ be the standard scalar product on~$\RR[\disc(m')^{g-h+1}] \otimes \CC$.
In this notation, we have~$\psi_{m'}(\tau'_1, z', w') = \langle \td{f}(\tau'_1),\, \theta_{m'}(\tau'_1, z', w') \rangle$.
Since the components of~$\theta_{m'}$ are linearly independent as formal series and the ring of formal power series in $e(\tau_{i\!j})$ for~$1 \le i,j \le g$ is an integral domain, for any formal series~$\hat{f}$ in~$\CC\llbracket e(\tau_{i\!j}) \rrbracket_{1 \le i,j \le g-h+1} \otimes \CC[\disc(m')^{g-h+1}]$, the relation~$\langle \hat{f}, \theta_{m'} \rangle = 0$ implies that~$\hat{f} = 0$.

To verify the symmetry condition for the Fourier series expansion of~$\td{f}$, we consider an arbitrary~$a \in \GL{g-h+1}(\cOE)$.
Note that~$t[\diag(a, 1_{h-1})]$ has the same bottom right block~$m'$ as~$t$.
Since we have~$f(\tau) = \sum \psi_{m'}(\tau'_1, z', w')\, e(m' \tau'_2)$, the symmetry condition for~$f$ implies that~$\psi_{m'}$ is invariant under the weight~$k$ slash action of~$\rot(a)^{-1}$.
Thus, by the modular transformation behavior of~$\theta_{m'}$ we have
\begin{align*}
    &
  \big\langle \td{f}, \theta_{m'} \big\rangle
  =
  \psi_{m'}
  =
  \psi_{m'} \big|_k\, \rot(a)^{-1}
  =
  \Big\langle
  \td{f} \big|_{k-h+1}\, \rot(a)^{-1},\,
  \theta_{m'} \big|_{h-1}\, \rot(a)^{-1}
  \Big\rangle
  \\
  = &
  \Big\langle
  \td{f} \big|_{k-h+1}\, \rot(a)^{-1},\,
  \rho^{(g-h+1)}_{m'}\bigl( \rot(a)^{-1} \bigr)\,
  \theta_{m'}
  \Big\rangle
  =
  \Big\langle
  \rho^{(g-h+1)\,\vee}_{m'}\bigl( \rot(a) \bigr)\,
  \td{f} \big|_{k-h+1}\, \rot(a)^{-1},\,
  \theta_{m'}
  \Big\rangle
  \tx{.}
\end{align*}
Comparing the Fourier series coefficients of~$\td{f}$ on the left hand side with its transform on the right hand side yields
\begin{gather*}
  c(\td{f}; n')
  =
  \rho^{(g-h+1)\,\vee}_{m'}\bigl( \rot(a) \bigr)\,
  \detF(\ovsmash{a})^{-k+h-1}\,
  c(\td{f}; n'[a])
  \tx{,}
\end{gather*}
which confirms the symmetry condition~\eqref{eq:hermitian_fourier_coefficient_symmetry} for~$\td{f}$ after rearranging.

We next inspect the Fourier--Jacobi coefficients of~$\td{f}$.
The cogenus~$1$ Fourier--Jacobi coefficients of~$\td{f}_{\mu'}$ have index~$m_{11} - m^{\prime\,-1}[r'_2(\mu')]$ by~\eqref{eq:prf:prop:reduction_to_lower_cogenus:theta_decomposition}, where~$m'$ and~$\mu'$ are fixed and~$m_{11}$ may vary.
We write~$\phi_{\mu', m_{11}}$ for them, omitting dependence on~$m'$ from our notation.

First, we examine the convergence of~$\phi_{\mu', m_{11}}$.
By the formal Fourier series expansion in~\eqref{eq:prf:prop:reduction_to_lower_cogenus:theta_decomposition} and the change of variable names in~\eqref{eq:def:variable_subdivision} and~\eqref{eq:def:fourier_index_subdivision}, we have
\begin{align*}
  \td{f}_{\mu'}(\tau'_1)\,
  e\bigl( m^{\prime\,-1}[r'(\mu')] \tau'_1 \bigr)
   & =
  \sum_{n'_{22} \in \cOF^\vee}
  \Biggl(
  \sum_{n'_{11}, n'_{21}}
  c(f; t)\,
  e\bigl(
  n'_{11} \tau'_{11}
  +
  \lTr{n}{^\prime_{21}} z_{11}
  +
  \lTr{\ovsmash{n}}{^\prime_{21}} w_{11}
  \bigr)
  \Biggr)\,
  e\bigl( n'_{22} \tau'_{12} \bigr)
  \\
   & =
  \sum_{m_{11} \in \cOF^\vee}
  \Biggl(
  \sum_{n, r_1}
  c(f; t)\,
  e\bigl(
  n \tau_1
  +
  \lTr{r}{_{1}} z_{11}
  +
  \lTr{\ovsmash{r}}{_{1}} w_{11}
  \bigr)
  \Biggr)\,
  e\bigl( m_{11} \tau_{21} \bigr)
  \tx{,}
\end{align*}
where~$n'_{11} = n$ runs through~$\MatD{g-h}(\cOE)^\vee$ and~$n'_{21} = r_1$ runs through~$\MatD{1,g-h}(\cOE^\vee)$.
The coefficients of this series in~$m_{11}$ are~$\phi_{\mu',m_{11}}$.
We compare this with the Fourier expansion of the absolutely and locally uniformly convergent Fourier--Jacobi coefficient
\begin{align*}
  \phi_m(\tau_1, z, w)
   & =
  \sum_{n, r_1, r_2}
  c(f; t)\,
  e\bigl(
  n \tau_1
  +
  \lTr{r}{_{1}} z_{11}
  +
  \lTr{r}{_{2}} z_{12}
  +
  \lTr{\ovsmash{r}}{_{1}} w_{11}
  +
  \lTr{\ovsmash{r}}{_{2}} w_{12}
  \bigr)
  \\
   & =
  \sum_{r_2}
  e\bigl(
  \lTr{r}{_{2}} z_{12}
  +
  \lTr{\ovsmash{r}}{_{2}} w_{12}
  \bigr)\,
  \sum_{n, r_1}
  c(f; t)\,
  e\bigl(
  n \tau_1
  +
  \lTr{r}{_{1}} z_{11}
  +
  \lTr{\ovsmash{r}}{_{1}} w_{11}
  \bigr)
  \tx{,}
\end{align*}
where the sum runs additionally over~$r_2 \in \MatD{h-1,g-h}(\cOE^\vee)$.
The desired convergence follows, since the inner sum for~$r_2 = r'_1(\mu')$ equals~$\phi_{\mu',m_{11}}$ as a formal series.

Next, we inspect the transformation behavior of~$\phi_{\mu',m_{11}}$.
We employ the same argument used for the symmetry condition, observing that the action of~$Q_h(\cOF) \cap \SU{g-h+1,g-h+1}(\cOF) \subset \SU{g,g}(\cOF)$ fixes~$\psi_{m'}(\tau'_1, z', w')\, e(m' \tau'_2)$ with~$\psi_{m'}$ as in~\eqref{eq:prf:prop:reduction_to_lower_cogenus:theta_decomposition} termwise with respect to the additional formal expansion in~$e(\tau'_{12})$.
We let~$\ga \in Q_h(\cOF) \cap \SU{g-h+1,g-h+1}(\cOF)$ act on
\begin{align*}
    &
  \Big\langle \td{f}(\tau'_1),\, \theta_{m'}(\tau'_1, z', w')\, e(m' \tau'_2) \Big\rangle
  =
  \psi_{m'}(\tau'_1, z', w')\, e(m' \tau'_2)
  =
  \bigl( \psi_{m'}(\tau'_1, z', w')\, e(m' \tau'_2) \bigr) \big|_k\, \ga
  \\
  = &
  \Big\langle
  \rho^{(g-h+1)\,\vee}_{m'}( \ga )^{-1}\,
  \td{f}(\tau'_1) \big|_{k-h+1}\, \ga,\,
  \theta_{m'}(\tau'_1, z', w')\, e(m' \tau'_2)
  \Big\rangle
  \tx{.}
\end{align*}
Thus we have
\begin{gather*}
  \td{f}(\tau'_1)
  =
  \rho^{(g-h+1)\,\vee}_{m'}( \ga )^{-1}\,
  \td{f}(\tau'_1) \big|_{k-h+1}\, \ga
  \tx{,}
\end{gather*}
and inspecting the coefficient of~$e( (m_{11} - m^{\prime\,-1}[r'_2(\mu')]) \tau'_{12} )$ on both sides establishes the desired transformation behavior of~$\phi_{\mu', m_{11}}$.

To verify the cusp form condition for~$\phi_{\mu',m_{11}}$, and when~$F = \QQ$ the additional growth condition, we inspect directly the Fourier expansion of~$\td{f}_{\mu'}$ appearing in~\eqref{eq:prf:prop:reduction_to_lower_cogenus:theta_decomposition}.
The coefficient of~$f$ of index~$t$ with~$r' = r'(\mu')$ contributes to the coefficient of~$\td{f}_{\mu'}$ of index~$n' - m^{\prime\,-1}[r'(\mu')]$.
If~$c(f; t) \ne 0$, we have~$t > 0$ because~$f$ is cuspidal.
Since~$1_{g-h+1}$ has full rank, this implies
\begin{gather*}
  n' - m^{\prime\,-1}[r'(\mu')]
  =
  t\bigl[ \lT{\bigl(} 1_{g-h+1}, m^{\prime\,-1} r'(\mu') \bigr) \bigr]
  >
  0
  \tx{.}
\end{gather*}
\end{proof}

\subsection{The final convergence statement}%
\label{ssec:vector_valued_higher_cogenus:final_convergence}

Before proving \fref{Theorem}{mainthm:convergence}, for clarity we record one final relation between automatic convergence for symmetric formal Fourier--Jacobi series and for cuspidal series.

\begin{lemma}%
\label{la:convergence_implies_convergence_cusp_forms}
Given~$1 \le h < g$, a fixed weight~$k \in \ZZ$, and a fixed arithmetic type~$\rho$,
\begin{gather*}
  \tx{if\/ }
  \rmFM^{(g,h)}_k(\rho) = \rmM^{(g)}_k(\rho)
  \tx{,}\quad
  \tx{then also }
  \rmFS^{(g,h)}_k(\rho) = \rmS^{(g)}_k(\rho)
  \tx{.}
\end{gather*}
\end{lemma}

\begin{proof}
The cusp form condition for Hermitian modular forms says~$c(f; t) = 0$ if~$t \not>0$, matching that of Hermitian Jacobi forms, which appears in the definition of cuspidal symmetric formal Fourier--Jacobi series.
\end{proof}

\begin{proof}%
[Proof of \fref{Theorem}{mainthm:convergence}]
We show the theorem by induction on the cogenus.
If~$h = 1$, then \fref{Theorem}{thm:convergence_scalar_valued_cogenus_one} in conjunction with \fref{Lemma}{la:convergence_implies_convergence_cusp_forms} and \fref{Proposition}{prop:reduction_to_scalar_weights_trivial_type} proves the theorem.

Assume now that~$h > 1$ and that the theorem is true for cogenus less than~$h$.
By \fref{Lemma}{la:convergence_implies_convergence_cusp_forms} and \fref{Proposition}{prop:reduction_to_scalar_weights_trivial_type} we may reduce to the scalar-valued cuspidal case.
Fix a scalar-valued, cuspidal symmetric formal Fourier--Jacobi series~$f$.
If~$m'$ is totally positive definite, applying \fref{Proposition}{prop:reduction_to_lower_cogenus} we find that its cogenus~$h-1$ Fourier--Jacobi coefficient~$\psi_{m'}$ of index~$m'$ in~\eqref{eq:def:sub_fourier_jacobi_coefficient} decomposes as~$\sum_{\mu'} \td{f}_{\mu'}\, \theta_{m',\mu'}$, where the sum runs over~$\mu' \in \disc(m')^{g-h+1}$ and the vector~$(\td{f}_{\mu'})_{\mu'}$ is a symmetric formal Fourier--Jacobi series of cogenus~$1$.
This allows us to invoke the base case of the induction and conclude that~$\psi_{m'}$ is a Hermitian Jacobi form for all totally positive definite~$m'$.
Since~$f$ is cuspidal, this comprises all nonzero~$\psi_{m'}$.
In other words,~$f$ is a symmetric formal Fourier--Jacobi series of cogenus~$h-1$.
By the induction hypothesis, it converges and is a Hermitian modular form as desired.
\end{proof}

%% file: sections/05_kudla_conjecture.tex
\section{Application to the Kudla Conjecture}%
\label{sec:kudla_conjecture}

We start by revisiting the generating series of special cycles for unitary groups.
This generating series has been defined in various places, including work of Kudla~\cite{kudla-1997,kudla-2021} and work of Liu~\cite[Section~3]{liu-2011a}, where it was shown that the generating series in the unitary case is a symmetric formal Fourier--Jacobi series in the sense of this work.
Liu, however, did not provide separate expressions for the Fourier--Jacobi coefficients, which merely appear as subexpressions of his equations~(3-5) and~(3-6).
Further, Liu relies on the adelic language, while we have chosen the number field setting, which is sufficient to capture the convergence question that we address.
For this reason, we prefer to recast the argument of Liu in the more classical language of Zhang's thesis~\cite{zhang-2009}, which gives us the opportunity to translate the adelic expansions of Liu to those of the present work.
We will crucially invoke the analogue of the pullback formula and the base case~$g = 1$, both due to Liu~\cite{liu-2011a}.

Recall from \fref{Section}{sec:hermitian_modular_jacobi_forms} that~$E / F$ is a CM field,~$\cOE / \cOF$ its maximal order, and~$E_\RR = E \otimes_\QQ \RR$.
We consider the unitary group~$\U{L}$ associated with a Hermitian lattice~$L$ over~$\cOE / \cOF$.
We assume that~$L_\RR = L \otimes_\ZZ \RR$ has signature~$(p,1)$, $p \ge 1$ at exactly one place, and is positive definite at all other infinite places.
We write~$L_\QQ = L \otimes_\ZZ \QQ$ for the associated rational Hermitian space, $\langle \,\cdot\,,\,\cdot\, \rangle$ for the corresponding sesquilinear form, $L^\vee \subset L_\QQ$ for the dual of~$L$, and~$\disc(L) = L^\vee \slash L$ for its discriminant form.
Further, we fix a neat, stable congruence subgroup~$\Ga \subset \U{L}(\cOF)$, that is, the action of~$\Ga$ fixes~$\disc(L)$ elementwise.

We let~$L_\RR^- = L \otimes_{\cOF} \RR$, using the embedding~$F \hra \RR$ corresponding to the place at which~$L_\RR$ is indefinite.
The Grassmannian~$\rmGrM(L_\RR^-)$ of complex negative lines is a Hermitian symmetric space associated with~$\U{L}(F_\RR)$.
The quotient~$\Ga \backslash \rmGrM(L_\RR^-)$ is a connected component of a Shimura variety for~$\U{L}$.
Conversely, up to finite covers, every such connected component of a Shimura variety for a neat, compact open subgroup of~$\U{L}(\bbA_{F,\rmf})$, where~$\bbA_{F,\rmf}$ denotes the finite adeles of~$F$, arises in such a way.
Since for congruence subgroups~$\Ga' \subseteq \Ga$ we have an injection via pullback
\begin{gather}
  \label{eq:def:kudla_conjecture:cover_pullback}
  \rmCH^{\bullet}\bigl( \Ga \backslash \rmGrM(L_\RR^-) \bigr) \otimes \QQ
  \lhra
  \rmCH^{\bullet}\bigl( \Ga' \backslash \rmGrM(L_\RR^-) \bigr) \otimes \QQ
  \tx{,}
\end{gather}
our convergence statement about the generating series of algebraic cycles will be unambiguous, and extends to cycles on such Shimura varieties.
Further, Shimura varieties for groups~$\U{L}$ with~$L$ a Hermitian lattice as above uniformize the Shimura variety associated to a nearby totally positive definite but incoherent Hermitian space as in Liu's work~\cite{liu-2011a}.
In particular, by proving convergence of the generating series of special cycles on~$\Ga \backslash \rmGrM(L_\RR^-)$ we verify the convergence condition~(2) of~\cite[Theorem~3.5]{liu-2011a}.

Given~$\la \in L_\QQ^g$, we define a rational Hermitian cycle on~$\rmGrM(L_\RR^-)$ as the intersection
\begin{gather}
  \label{eq:def:rational_hermitian_cycle}
  \rmZ(\la)
  =
  \rmZ_L(\la)
  \defeq
  \bigcap_{i = 1}^g
  \big\{
  W \in \rmGrM(L_\RR^-)
  \condsep
  W \perp \la_i
  \big\}
  \in
  \rmCH^{\rk(\la)}\big( \rmGrM(L_\RR^-) \big)
\end{gather}
of expected codimension~$\rk(\la)$, the dimension of the $\RR$\nbd{}span of the components of~$\la$ in~$L_\RR^-$.
Note that if this span contains a vector that is negative at the indefinite place of~$L$, then~$\rmZ(\la) = 0$.
To obtain rational Hermitian cycles on the locally symmetric space, for the~$\Ga$\nbd{}class~$[\la]$ of~$\la$, we set
\begin{gather*}
  \rmZ([\la])
  =
  \rmZ_L([\la])
  \defeq
  \bigcup_{\la \in [\la]}
  \rmZ_L(\la)
  \in
  \rmCH^{\rk(\la)}\big( \Ga \backslash \rmGrM(L_\RR^-) \big)
  \tx{.}
\end{gather*}

Since~$\Ga$ acts trivially on~$\disc(L)$, $\Ga$\nbd{}classes uniquely define elements in the discriminant group.
Now formal summation over all~$\la \in L^{\vee\, g}$ in a given class~$\mu \in \disc(L)^g$ and with fixed moment matrix~$\frac{1}{2} \langle \la, \la \rangle \in \MatD{g}(E)$ yields a special cycle, which descends to the quotient by~$\Ga$ as follows:
\begin{gather}
  \label{eq:def:special_cycle}
  \rmZ(\mu; t)
  =
  \rmZ_L(\mu; t)
  \defeq
  \bigcup_{\substack{[\la] \in \mu + \Ga \backslash L^g \\ \frac{1}{2} \langle \la, \la \rangle = t}}
  \mspace{-12mu}
  \rmZ_L([\la])
  \in
  \rmCH^{\rk(t)}\big( \Ga \backslash \rmGrM(L_\RR^-) \big)
  \tx{.}
\end{gather}
Using the dual~$\omega$ of the Hodge bundle to force codimension~$g$ cycles, we obtain the Kudla generating series
\begin{gather}
  \label{eq:def:kudla_generating_series}
  \thetaKudla_{L,\mu}(\tau)
  \defeq
  \sum_{t \in \frac{1}{2} \langle \mu, \mu \rangle + \MatD{g}(\cOE)^\vee}
  \mspace{-32mu}
  \big( \rmZ(\mu; t) \cap \omega^{g - \rk(t)} \big)\,
  e(t \tau)
  \in
  \CC
  \otimes
  \rmCH^g\big( \Ga \backslash \rmGrM(L_\RR^-) \big) \bigl\llbracket e(\tau_{i\!j}) \bigr\rrbracket
  \tx{,}
\end{gather}
which is a priori a \emph{formal} series in the variables~$e(\tau_{i\!j})$, $1 \le i,j \le g$, with coefficients in the Chow group with complex coefficients of~$\Ga \backslash \rmGrM(L_\RR^-)$.
Note that we suppress the subgroup~$\Ga$ from our notation for convenience, which is justified by the injective pullback~\eqref{eq:def:kudla_conjecture:cover_pullback}.

In this section, we fix such a nearby incoherent totally positive definite Hermitian space~$\wtd{V}$.
To describe the expected modular transformation behavior of~\eqref{eq:def:kudla_generating_series}, we borrow from Liu~\cite[Section~3]{liu-2011a} the Weil representation for~$\wtd{V}$.
We restrict this representation to the span of shifts of the indicator function of~$L$ by elements of~$L^\vee$, which yields a representation~$\wtd\rho_L^{(g)}$ of~$\SU{g,g}(\cOF)$.
For brevity we call it the incoherent Weil representation associated with~$L$.
Its representation space is~$\CC[ \disc(L)^g ]$ with basis elements~$\frake_\mu$ for~$\mu \in \disc(L)^g$.
The vector-valued Kudla generating series is defined as
\begin{gather}
  \label{eq:def:kudla_generating_series_vector_valued}
  \thetaKudla_L(\tau)
  \defeq
  \sum_{\mu \in \disc(L)^g}
  \thetaKudla_{L,\mu}(\tau)\, \frake_\mu
  \in
  V\big( \wtd\rho_L^{(g)} \big)
  \otimes
  \rmCH^g\big( \Ga \backslash \rmGrM(L_\RR^-) \big) \bigl\llbracket e(\tau_{i\!j}) \bigr\rrbracket
  \tx{.}
\end{gather}

Since spaces of modular forms are finite dimensional, we can make the following identification of tensors
\begin{gather*}
  \sum_f
  f \otimes Z_f
  \in
  \rmM_k(\rho)
  \otimes
  \rmCH^g\big( \Ga \backslash \rmGrM(L_\RR^-) \big)
  \quad\tx{with}\quad
  \sum_t
  \Big( \sum_f c(f; t) Z_f \Big)\,
  e(t \tau)
\end{gather*}
via the Fourier expansion of modular forms, where~$f$ runs through any basis of~$\rmM_k(\rho)$ and~$Z_f$ are arbitrary cycles.
This allows for a more concise formulation of the Kudla Conjecture, stated in \fref{Theorem}{mainthm:kudla_conjecture}.

\fref{Theorem}{mainthm:kudla_conjecture} is a consequence of \fref{Theorem}{mainthm:convergence} if we show that the Kudla generating series yields a symmetric formal Fourier--Jacobi series of cogenus~$h = g - 1$.
We establish the two basic properties separately, before proceeding to our proof of \fref{Theorem}{mainthm:kudla_conjecture}.
We first record the symmetry, which follows directly from the definition of special cycles.

\begin{proposition}%
\label{prop:kudla_generating_series_symmetric}
The formal Fourier series\/~$\thetaKudla_L$ defined in~\eqref{eq:def:kudla_generating_series_vector_valued} satisfies the symmetry condition in~\eqref{eq:hermitian_fourier_coefficient_symmetry} for weight\/~$p+1$ and type\/~$\wtd\rho^{(g)}_L$.
\end{proposition}

\begin{proof}
For simplicity, we write~$\rho$ for the incoherent Weil representation of~$L$.
For clarity, we remark that the following argument does not make use of the condition~$\det(a) = \det(\ov{a})$ included in~\eqref{eq:hermitian_fourier_coefficient_symmetry}, since~$\rho$ extends to the unitary group.
From the definition of rational Hermitian cycles, we see that~$\rmZ(\la a) = \rmZ(\la)$ for all~$a \in \GL{g}(\cOE)$, and thus~$\rmZ(\mu a; t[a]) = \rmZ(\mu; t)$.
After inserting~$\rho(\rot(a))^{-1}\, \frake_{\mu} = \detF(\ov{a})^{-p-1}\, \frake_{\mu a}$, we find that
\begin{gather*}
  \sum_{\mu}
  \big( \rmZ(\mu a; t[a]) \cap \omega^{g-\rk(t[a])} \big)\,
  \frake_{\mu a}
  =
  \sum_{\mu}
  \big( \rmZ(\mu; t) \cap \omega^{g-\rk(t)} \big)\,
  \detF(\ov{a})^{p+1} \rho(\rot(a))^{-1}\, \frake_{\mu}
  \tx{,}
\end{gather*}
and hence
\begin{gather*}
  c\big( \thetaKudla_L; t[a] \big)
  =
  \detF(\ov{a})^{p+1}\,
  \rho(\rot(a))^{-1}\,
  c\big( \thetaKudla_L; t \big)
  \tx{.}
\end{gather*}
\end{proof}

The Kudla generating series is a modular form in the case of codimension one by results of Liu~\cite{liu-2011a}.

\begin{theorem}%
[Liu]%
\label{thm:kudla_conjecture_weak_form_codim1}
For~$g = 1$, we have
\begin{gather*}
  \thetaKudla_L
  \in
  \rmM_{p+1}\big( \wtd\rho^{(1)}_L \big)
  \otimes
  \rmCH^1\big( \Ga \backslash \rmGrM(L_\RR^-) \big)
  \tx{.}
\end{gather*}
\end{theorem}

\begin{remark}
Hofmann~\cite{hofmann-2014} obtained \fref{Theorem}{thm:kudla_conjecture_weak_form_codim1} in the special case~$F = \QQ$.
Note that he uses the opposite signature convention and therefore obtains a generating series for the dual of the Weil representation.
Further, in his special case the incoherent Weil representation agrees with the Weil representation on the group~$\SU{1,1}(\ZZ)$ that he considers, although it differs on~$\U{1,1}(\ZZ)$.
\end{remark}

This theorem is a key ingredient in the following proposition, whose proof follows the ideas of Zhang's thesis~\cite{zhang-2009}.

\begin{proposition}%
\label{prop:kudla_conjecture_fourier_jacobi_coefficient}
For~$m \in \MatD{g-1}(\cOE)^\vee$, let~$\phi_m$ be the~$m$\thdash{} Fourier--Jacobi coefficient of the formal Fourier series~$\thetaKudla_L(\tau)$ defined in~\eqref{eq:def:kudla_generating_series_vector_valued}.
Then
\begin{gather*}
  \phi_m
  \in
  \rmJ^{(1)}_{p+1, m}\big( \wtd\rho^{(g)}_L \big)
  \otimes
  \rmCH^g\big( \Ga \backslash \rmGrM(L_\RR^-) \big)
  \tx{.}
\end{gather*}
\end{proposition}

\begin{proof}
We will write~$\phi_m$ as a finite sum of Hermitian Jacobi forms with coefficients in the Chow group.
Each term is associated with a class~$[\wtd\la] \in \Ga \backslash L^{g-1}$.
We start by setting up notation that allows us to state this sum decomposition.

Given~$\la \in L^g$, we write~$\wtd\la = (\la_2, \ldots, \la_g)$ for its last~$g-1$ components and for convenience we write~$\la = (\la_1, \wtd\la)$.
Correspondingly, the image of~$\la$ in the discriminant group of~$L^g$ is~$\mu = (\mu_1, \wtd\mu)$.
We set~$m = \frac{1}{2} \langle \wtd\la, \wtd\la \rangle$; in the proof this will be compatible with the decomposition of Fourier indices in~\eqref{eq:hermitian_fourier_index_decomposition}, where~$m$ is the bottom right block of a Fourier index~$t$.
We let~$\wtd{L} \subseteq L$ be the~$\cOE$\nbd{}span of~$\wtd\la$ with orthogonal complement~$\wtd{L}^\perp$.
We write~$\wtd{\Ga} \subseteq \Ga$ for the stabilizer of~$\wtd\la$.

Throughout the proof, we assume that~$\rk(\wtd{L}) = \rk(m)$ and~$m$ is totally positive semi-definite, so that the cycles~$\rmZ(\la)$ and~$\rmZ(\mu; m)$ that appear later are nontrivial.
In particular, $\wtd{L}$ is a totally positive definite lattice, and we have a projection from~$L_\RR$ onto~$\wtd{L}_\RR$.
We will apply this projection to the first component~$\la_1$ of~$\la$, and write~$\ovsmash{\la}_1$ for the image and~$\ulsmash{\la}_1 = \la_1 - \ovsmash{\la}_1$ for its complement in~$\wtd{L}^\perp_\RR$.
We write~$\la_1 = (\ulsmash{\la}_1, \ovsmash{\la}_1)$ and correspondingly decompose~$\mu_1$.

Note that~$\wtd{L}$ in general does not agree with~$\ov{L} = \{ \ovsmash{\la}_1 \condsep \la_1 \in L \}$.
Instead, by a standard argument (see~\cite[Lemma~2.7]{zhang-2009}) we have the isomorphism
\begin{gather*}
  L \bigslash ( \wtd{L}^\perp \oplus \wtd{L} )
  \cong
  \ov{L} \bigslash \wtd{L}
  \tx{.}
\end{gather*}
The chain of inclusions~$\wtd{L}^\perp \oplus \wtd{L} \subseteq L \subseteq L^\vee \subseteq \wtd{L}^{\perp\,\vee} \oplus \wtd{L}^\vee$ yields a projection of (incoherent) Weil representations
\begin{align*}
  \pirho_{\wtd{L}}
  \defcol
  \wtd\rho_{\wtd{L}^\perp}^{(1)}
  \otimes
  \rho_{\wtd{L}}^{(1)}
   & \lthra
  \wtd\rho_{L}^{(1)}
  \tx{,}
  \\
  \frake_{\ulsmash{\mu}_1}
  \otimes
  \frake_{\ovsmash{\mu}_1}
   & \lmto
  \begin{cases}
    \frake_{\la_1 + L}
    \bigslash \# \bigl(L \slash (\wtd{L}^\perp \oplus \wtd{L}) \bigr)
    \tx{,}\quad
     &
    \tx{if } (\ulsmash{\mu}_1, \ovsmash{\mu}_1) \tx{ is represented by some } \la_1 \in L^\vee
    \tx{;}
    \\
    0
    \tx{,}\quad
     &
    \tx{otherwise}
    \tx{.}
  \end{cases}
\end{align*}
Here we use that, since~$\wtd{L}$ is totally positive definite, its associated Weil representation appears as a subrepresentation of the Weil representation associated to the incoherent space~$\wtd{V}$.

We also use the following map, defined via pushforward along the inclusion of varieties
\begin{gather*}
  \piCH_{\wtd{L}}
  \defcol
  \rmCH^{g - \rk(\wtd{L})}\bigl( \wtd{\Ga} \backslash \rmGrM\bigl( (\wtd{L}^\perp)_\RR^- \big) \bigr)
  \lra
  \rmCH^g\bigl( \Ga \backslash \rmGrM(L_\RR^-) \bigr)
  \tx{.}
\end{gather*}

Recall that the Hermitian Jacobi theta series in~\eqref{eq:def:hermitian_jacobi_theta} are modular.
By the Kudla Conjecture for divisors in \fref{Theorem}{thm:kudla_conjecture_weak_form_codim1}, we have
\begin{gather*}
  \thetaKudla_{\wtd{L}^\perp}(\tau_1)
  \otimes
  \theta_{\wtd{L}}(\tau_1, z, w)
  \in
  \rmJ^{(1)}_{p+1,m}\bigl( \wtd\rho^{(1)}_{\wtd{L}^\perp} \otimes \rho^{(1)}_{\wtd{L}} \bigr)
  \otimes
  \rmCH^{g - \rk(\wtd{L})}\bigl( \wtd{\Ga} \backslash \rmGrM\bigl( (\wtd{L}^\perp)_\RR^- \bigr) \bigr)
  \tx{,}
\end{gather*}
which implies that
\begin{gather*}
  \pirho_{\wtd{L}}\Bigl(
  \piCH_{\wtd{L}}\bigl(
    \thetaKudla_{\wtd{L}^\perp}(\tau_1)
    \bigr)
  \otimes
  \theta_{\wtd{L}}(\tau_1, z, w)
  \Bigr)
  \in
  \rmJ^{(1)}_{p+1,m}\bigl( \wtd\rho^{(1)}_L \bigr)
  \otimes
  \rmCH^g\bigl( \Ga \backslash \rmGrM(L_\RR^-) \bigr)
  \tx{.}
\end{gather*}
In analogy with Liu's pullback formula~\cite[Proposition~3.2]{liu-2011a}, we claim that
\begin{gather}
  \label{eq:prop:kudla_conjecture_fourier_jacobi_coefficient:sum_decomposition}
  \phi_m(\tau_1, z, w)\,
  e(m \tau_2)
  =
  \sum_{\wtd\mu}
  \sum_{\substack{
      [\wtd\la] \in \wtd\mu + \Ga \backslash L^{g-1} \\
      \frac{1}{2} \langle \wtd\la, \wtd\la \rangle = m
    }}
  \pirho_{\wtd{L}}\Bigl(
  \piCH_{\wtd{L}}\bigl(
    \thetaKudla_{\wtd{L}^\perp}(\tau_1)
    \bigr)
  \otimes
  \theta_{\wtd{L}}(\tau_1, z, w)\,
  \frake_{\wtd\mu}
  \Bigr)\,
  e(m \tau_2)
  \tx{.}
\end{gather}

For a proof, we decompose the coefficients on the left hand side.
For fixed~$\wtd\la$ we write~$[\la_1]$ for the~$\wtd{\Ga}$\nbd{}class of~$\la_1$.
Further, if~$\ovsmash{\la}_1$ is fixed, we let~$[\ulsmash{\la}_1]$ be the class of~$\ulsmash{\la}_1$ with respect to the stabilizer~$\wtd\Ga_1$ of~$\ovsmash{\la}_1$ in~$\wtd\Ga$.
In this way, it makes sense to decompose the class~$[\la_1]$ as~$([\ulsmash{\la}_1], [\ovsmash{\la}_1])$.
We will employ a special cycle with partially fixed~$\la$.
Specifically, we define
\begin{gather*}
  \rmZ\bigl( \ulsmash{\mu}_1, [\ovsmash{\la}_1], [\wtd\la]; n \bigr)
  \defeq
  \bigcup_{\substack{
      [\ulsmash{\la}_1] \in \ulsmash{\mu}_1 + \wtd\Ga_1 \backslash \wtd{L}^\perp \\
      \frac{1}{2} \langle \la_1, \la_1 \rangle = n
    }}
  \rmZ([\la])
  \tx{.}
\end{gather*}
Using notation for~$t$ as in~\eqref{eq:hermitian_fourier_index_decomposition}, we observe that the classes of~$\ovsmash{\la}_1$ and~$\wtd\la$ determine~$m$ and~$r$.
In particular, we have the finite decomposition
\begin{gather*}
  \rmZ(\mu; t)
  =
  \bigcup
  \rmZ\bigl( \ulsmash{\mu}_1, [\ovsmash{\la}_1], [\wtd\la]; n \bigr)
  \tx{,}
\end{gather*}
where the union ranges over~$[\ovsmash{\la}_1] \in \ovsmash{\mu}_1 + \wtd\Ga \backslash \wtd{L}$ with~$\frac{1}{2} \langle \wtd\la, \ovsmash{\la}_1 \rangle = r$ and $[\wtd\la] \in \wtd\mu + \Ga \backslash L^{g-1}$ with~$\frac{1}{2} \langle \wtd\la, \wtd\la \rangle = m$.
This is a decomposition of the Fourier coefficients of~$\phi_m$.

As for the right hand side of~\eqref{eq:prop:kudla_conjecture_fourier_jacobi_coefficient:sum_decomposition}, we start with the expansion
\begin{gather*}
  \thetaKudla_{\wtd{L}^\perp}(\tau_1)
  \otimes
  \theta_{\wtd{L}}(\tau_1, z, w)
  =
  \sum_{\ulsmash{\mu}_1, n', \ovsmash{\la}_1}
  \rmZ_{\wtd{L}^\perp}(\ulsmash{\mu}_1; n')\,
  e\bigl( \bigl( n' + \tfrac{1}{2} \langle \ovsmash{\la}_1, \ovsmash{\la}_1 \rangle \bigr)\, \tau_1 \bigr)\,
  e\bigl( \tfrac{1}{2} \langle \wtd{\la}, \ovsmash{\la}_1 \rangle z + \tfrac{1}{2} \ov{\langle \wtd{\la}, \ovsmash{\la}_1 \rangle} w \bigr)\,
  \frake_{(\ulsmash{\mu}_1, \ovsmash{\mu}_1)}
  \tx{.}
\end{gather*}
The Fourier coefficient of index~$n, r$ equals a sum over~$\rmZ_{\wtd{L}^\perp}(\ulsmash{\mu}_1; n - \frac{1}{2} \langle \ovsmash{\la}_1, \ovsmash{\la}_1 \rangle)$, where~$\frac{1}{2} \langle \wtd{\la}, \ovsmash{\la}_1 \rangle = r$.
The projection on Weil representations allows us to discard all~$(\ulsmash{\mu}_1, \ovsmash{\mu}_1)$ in the sum that do not lie in the image of~$L^\vee \ra \disc(\wtd{L}^\perp \oplus \wtd{L})$.
For all other~$\la_1$ we can express the projection of cycles as
\begin{gather*}
  \piCH_{\wtd{L}}\bigl(
  \rmZ_{\wtd{L}^\perp}(\ulsmash{\mu}_1; n - \tfrac{1}{2} \langle \ovsmash{\la}_1, \ovsmash{\la}_1 \rangle)
  \bigr)
  =
  \rmZ\bigl( \ulsmash{\mu}_1, [\ovsmash{\la}_1], [\wtd\la]; n \bigr)
  \tx{.}
\end{gather*}
This gives
\begin{gather*}
  c\Bigl(
  \pirho_{\wtd{L}}\Bigl(
    \piCH_{\wtd{L}}\bigl(
      \thetaKudla_{\wtd{L}^\perp}(\tau_1)
      \bigr)
    \otimes
    \theta_{\wtd{L}}(\tau_1, z, w)
    \Bigr)
  ;\,
  n, r
  \Bigr)
  =
  \sum_{\substack{
      \mu_1, [\ovsmash{\la}_1] \in \ovsmash{\mu}_1 + \wtd{\Ga} \backslash \wtd{L} \\
      \frac{1}{2} \langle \wtd{\la}, \ovsmash{\la}_1 \rangle = r
    }}
  \rmZ\bigl( \ulsmash{\mu}_1, [\ovsmash{\la}_1], [\wtd\la]; n \bigr)\,
  \frake_{\mu_1}
  \tx{.}
\end{gather*}
Comparing with the above decomposition of~$\rmZ(\mu; t)$, this establishes~\eqref{eq:prop:kudla_conjecture_fourier_jacobi_coefficient:sum_decomposition}.
\end{proof}

\begin{proof}%
[Proof of \fref{Theorem}{mainthm:kudla_conjecture}]
The case~$g = 1$ is a theorem by Liu~\cite{liu-2011a}, which we stated as \fref{Theorem}{thm:kudla_conjecture_weak_form_codim1}.
We assume that~$g > 1$.
Since Chow groups are not known to be finite dimensional, we have to proceed with care regarding the difference between formal series with coefficients in Chow groups and the tensor product of formal series with Chow groups.

Let~$f$ be any functional on~$\rmCH^g( \Ga \backslash \rmGrM(L_\RR^-) )$.
It yields a formal Fourier series~$f(\thetaKudla_L(\tau))$.
By \fref{Proposition}{prop:kudla_conjecture_fourier_jacobi_coefficient} it is a formal series of Jacobi forms and by \fref{Proposition}{prop:kudla_generating_series_symmetric}, it satisfies the symmetry condition in \fref{Definition}{def:symmetric_formal_fourier_jacobi_series}.
We conclude by \fref{Theorem}{mainthm:convergence} that
\begin{gather*}
  f\bigl( \thetaKudla_L(\tau) \bigr)
  \in
  \rmM_{p+1}\big( \wtd\rho^{(g)}_L \big)
  \quad
  \tx{ for all }
  f \in \rmCH^g\bigl( \Ga \backslash \rmGrM(L_\RR^-) \bigr)^\vee
  \tx{,}
\end{gather*}
which implies the theorem, since $\rmM_{p+1}( \wtd\rho^{(g)}_L )$ is finite dimensional.
\end{proof}

\begin{proof}%
[Proof of \fref{Corollary}{maincor:li_liu_hypothesis_removed}]
Hypothesis~4.5 in~\cite{li-liu-2021} asserts that the Kudla generating series formed with respect to an arbitrary Schwartz function on~$L^g \otimes \bbA_{F,\rmf}$ is a Chow-valued Hermitian Hilbert modular form for some congruence subgroup of the unitary group.
As~$\cOE^\times \slash \cOF^\times$ is finite, the group~$\SU{g,g}(\cOF)$ used in in \fref{Section}{sec:hermitian_modular_jacobi_forms} contains a congruence subgroup of~$\U{g,g}(\cOF)$.
\fref{Theorem}{mainthm:kudla_conjecture} establishes convergence of the Kudla generating series for translates of the characteristic function of~$L^g$ by~$L^{\vee\, g}$ for every Hermitian~$\cOE$\nbd{}lattice~$L$.
Since such characteristic functions span the space of Schwartz functions on~$L^g \otimes \bbA_{F,\rmf}$ as~$L$ varies in a fixed Hermitian space, this indeed verifies the hypothesis~\cite[Hypothesis~4.5]{li-liu-2021} for all Schwartz functions.
\end{proof}